\documentclass[12pt]{amsart}
\usepackage{tikz-cd}
\usepackage{amsfonts}
\usepackage{mathrsfs}
\usepackage{amsxtra,amsthm,amsmath,amssymb, amscd}
\usepackage{stmaryrd}

\theoremstyle{plain}

\newtheorem{thm}{Theorem}[section]
\newtheorem{pro}[thm]{Proposition}

\newtheorem{lem}[thm]{Lemma}

\newtheorem{cor}[thm]{Corollary}

\theoremstyle{definition} 
\newtheorem{exa}[thm]{Example}

\newtheorem{rem}[thm]{Remark}

\newtheorem{lem-defn}[thm]{Lemma-Definition}

\renewcommand{\qed}{\begin{flushright} {\bf Q.E.D.}\ \ \ \ \
                  \end{flushright} }

\newcommand{\hs}{\hspace{.2in}}

\newcommand{\IC}{\mathbb{C}}
\newcommand{\IF}{\mathbb{F}}

\newcommand{\IN}{\mathbb{N}}

\newcommand{\XX}{\mathfrak{X}}

\DeclareMathAlphabet{\mathbbold}{U}{bbold}{m}{n}

\def \g  {\mathfrak{a}}
\def \b  {\mathfrak{b}}

\def \d  {\mathfrak{d}}

\def \g  {\mathfrak{g}}
\def \h  {\mathfrak{h}}

\def \n  {\mathfrak{n}}

\def \bfu  {\mathbf{u}}

\def \bfv  {\mathbf{v}}
\def \bfw  {\mathbf{w}}

\def \G  {{\mathcal{G}}}

\def \K  {{\mathcal{K}}}
\def \L  {{\mathcal{L}}}

\def \O  {{\mathcal{O}}}

\def \R  {{\mathcal{R}}}

\def \Y  {{\mathcal{Y}}}
\def \Z  {{\mathcal{Z}}}

\DeclareMathOperator{\im}{Im}
\DeclareMathOperator{\Ad}{Ad}

\DeclareMathOperator{\diag}{diag}

\def \sB {{\scriptscriptstyle B}}

\def \sD {{\scriptscriptstyle D}}
\def \sF {{\scriptscriptstyle F}}
\def \sG {{\scriptscriptstyle G}}

\def \sK {{\scriptscriptstyle K}}
\def \sL {{\scriptscriptstyle L}}

\def \sY {{\scriptscriptstyle Y}}
\def \sZ {{\scriptscriptstyle Z}}
\def \sR {{\scriptscriptstyle R}}
\def \sS {{\scriptscriptstyle S}}
\def \sGs {{\scriptscriptstyle G^*}}
\def \sGam {{\scriptscriptstyle \Gamma}}
\def \ssL {{\scriptscriptstyle \L}}
\def \ssY {{\scriptscriptstyle \Y}}
\def \ssZ {{\scriptscriptstyle \Z}}

\newcommand{\del}{\partial}

\newcommand{\bbfu}{\bar{\bfu}}
\newcommand{\bbfv}{\bar{\bfv}}
\newcommand{\bbfw}{\bar{\bfw}}
\newcommand{\la}{\langle}
\newcommand{\ra}{\rangle}
\newcommand{\lara}{\la \, , \, \ra}
\newcommand{\gog}{\g \oplus \g}

\newcommand{\piD}{{\pi_{\scriptscriptstyle D}}}
\newcommand{\piG}{{\pi_{\scriptscriptstyle G}}}
\newcommand{\piGs}{{\pi_{\scriptscriptstyle G^*}}}

\newcommand{\piP}{{\pi_{\scriptscriptstyle P}}}

\newcommand{\pist}{\pi_{\rm st}}
\newcommand{\Pist}{\Pi_{\rm st}}

\newcommand{\piX}{{\pi_{\scriptscriptstyle X}}}
\newcommand \piY {{\pi_{\scriptscriptstyle Y}}}
\newcommand \piZ {{\pi_{\scriptscriptstyle Z}}}

\begin{document}

\setlength{\baselineskip}{1.2\baselineskip}

\title[Poisson groupoids and generalised double Bruhat cells]{Local Poisson groupoids over mixed product Poisson structures and generalised double Bruhat cells}
\author{Victor Mouquin}
\address{
School of Mathematical Sciences   \\
Shanghai Jiaotong University \\
Shanghai, China}               
\email{mouquinv@sjtu.edu.cn}
\date{}

\begin{abstract}
Given a standard complex semisimple Poisson Lie group $(G, \pist)$, generalised double Bruhat cells $G^{\bfu, \bfv}$ and generalised Bruhat cells $\O^{\bfu}$ equipped with naturally defined holomorphic Poisson structures, where $\bfu, \bfv$ are finite sequences of Weyl group elements, were defined and studied by Jiang-Hua Lu and the author. We prove in this paper that $G^{\bfu, \bfu}$ is naturally a Poisson groupoid over $\O^{\bfu}$, extending a result from the aforementioned authors about double Bruhat cells in $(G, \pist)$. 

Our result on $G^{\bfu, \bfu}$ is obtained as an application of a construction interesting in its own right, of a local Poisson groupoid over a mixed product Poisson structure associated to the action of a pair of Lie bialgebras. This construction involves using a local Lagrangian bisection in a double symplectic groupoid closely related to the global $\R$-matrix studied by Weinstein and Xu, to ``twist" a direct product of Poisson groupoids. 
\end{abstract}

\maketitle


\section{Introduction} \label{sec-intro}

Let $G$ be a connected complex semisimple Lie group and $\pist$ the standard holomorphic multiplicative Poisson structure on $G$ determined by a pair $(B, B_-)$ of opposite Borel subgroups and a symmetric non-degenerate ad-invariant bilinear form on the Lie algebra of $G$. It is well known \cite{hodges, reshe-4} that $\pist$ is invariant under left and right translation by the maximal torus $T = B \cap B_-$ , and that the $T$-orbits of symplectic leaves are the double Bruhat cells 
$$
G^{u,v} = BuB \cap B_- vB_-,
$$
where $u, v$ are elements of the Weyl group $W$ of $(G, T)$. The Poisson structure $\pist$ descends to a well defined Poisson structure $\pi_1$ on the flag variety $G/B$, and any Bruhat cell 
$$
\O^u = BuB/B
$$ 
is a Poisson submanifold of $(G/B, \pi_1)$, see e.g \cite{Goodearl-Yakimov:GP}. A surprising fact proven in \cite{Lu-Mou:double-B-cell} is that for any $u \in W$, $G^{u,u}$ has a natural groupoid structure compatible with $\pist$, making $(G^{u, u}, \pist)$ into a Poisson groupoid over $(\O^u, \pi_1)$, and that the symplectic leaf in $G^{u,u}$ containing the identity bisection is a symplectic groupoid. 

In \cite{Lu-Mou:mixed, Lu-victor:Tleaves}, natural generalisations of (double) Bruhat cells, associated to finite sequences of Weyl group elements, are constructed. If $\bfu = (u_1, \ldots, u_n) \in W^n$, one has the {\it generalised Bruhat cell}
$$
\O^{\bfu} = Bu_1B \times_\sB \cdots  \times_\sB Bu_nB/B,
$$
where our notation is explained in $\S$\ref{nota-quotient}, and the spaces 
$$
B \bfu B = Bu_1B \times_\sB \cdots  \times_\sB Bu_nB \subset \tilde{F}_n, \hs B_- \bfu B_- = B_-u_1B_- \times_{\sB_-} \cdots  \times_{\sB_-} B_-u_nB_- \subset \tilde{F}_{-n}, 
$$
where $\tilde{F}_n$, $\tilde{F}_{-n}$ are defined in $\S$\ref{subsec-F_n}. If $\bfv = (v_1, \ldots, v_n) \in W^n$ is another finite sequence, one has the {\it generalised double Bruhat cell} 
$$
G^{\bfu, \bfv} = \{ \left( [g_1, \ldots, g_n]_{\scriptscriptstyle \tilde{F}_n}, \; [h_1, \ldots, h_n]_{\scriptscriptstyle \tilde{F}_{-n}} \right) \in B \bfu B \times B_- \bfv B_-  : \; g_1 \cdots g_n = h_1 \cdots h_n \},
$$
and $\pist$ induces the holomorphic Poisson structures $\pi_n$ on $\O^{\bfu}$ and $\tilde{\pi}_{n,n}$  on $G^{\bfu, \bfv}$.  When $n=1$, $(G^{\bfu, \bfv}, \tilde{\pi}_{n,n})$ and $(\O^{\bfu}, \pi_n)$ are naturally isomorphic to $(G^{u_1, v_1}, \pist)$ and $(\O^{u_1}, \pi_1)$. 

One of the central theorem of this paper, Theorem \ref{thm-main-Guu}, is the generalisation of the results in \cite{Lu-Mou:double-B-cell}: for any $l \geq 1$ and $\bfw \in W^l$, $G^{\bfw, \bfw}$ has a natural groupoid structure compatible with $\tilde{\pi}_{l,l}$, giving rise to a Poisson groupoid $(\G^{\bbfw, \bbfw}, \pi_{\bbfw, \bbfw})$ over $(\O^{\bfw}, \pi_l)$. The groupoid structure depends on the choice $\bbfw \in N_\sG(T)^l$ of a representative of $\bfw$, where $N_\sG(T)$ is the normaliser subgroup in $G$ of $T$, but the isomorphism class of Poisson groupoids is independent of such a choice.

While the results in \cite{Lu-Mou:double-B-cell} were obtained by embedding $(G^{u,u}, \pist)$ into a bigger Poisson groupoid whose underlying groupoid structure is that of an action groupoid, the same method does not work in the generalised double Bruhat cell setting. Instead, another main result of this paper is a construction of local Poisson groupoids over mixed product Poisson structures, and we obtain Theorem \ref{thm-main-Guu} as an application of this construction.

\subsection{Local Poisson groupoids over mixed product Poisson structures} \label{intro-subsec1}

Let $(Z, \piZ)$ be a Poisson manifold with a left Poisson action $\lambda: \g \to \XX^1(Z)$ of a Lie bialgebra $(\g, \delta_\g)$, and let $(Y, \piY)$ be a Poisson manifold with a right Poisson action $\rho: \g^* \to \XX^1(Y)$ of the dual Lie bialgebra $(\g^*, \delta_{\g^*})$. Then $(\rho, \lambda)$ defines the so-called {\it mixed product Poisson structure} 
$$
\piY \times_{(\rho, \lambda)} \piZ : = (\piY, \: \piZ) - (\rho(\xi^i), \: 0) \wedge (0, \: \lambda(x_i)) \in \XX^2(Y \times Z) 
$$
on $Y \times Z$, where $(x_i)$ is a basis of $\g$, $(\xi^i)$ the dual basis of $\g^*$, and where here and throughout the paper, we adopt the Einstein summation convention. Mixed product Poisson structures associated to pairs of actions of Lie bialgebras were first studied in \cite{Lu-Mou:mixed}. Let $(\Y \rightrightarrows Y, \pi_\ssY)$, $(\Z \rightrightarrows Z, \pi_\ssZ)$ be Poisson groupoids respectively over $(Y, \piY)$ and $(Z, \piZ)$, and let 
$$
\mu_\ssY: (\Y, \pi_\ssY) \to (G, \piG), \hs \mu_\ssZ: (\Z, \pi_\ssZ) \to (G^*, -\piGs), 
$$
be Poisson groupoid morphisms, where $(G, \piG)$, $(G^*, \piGs)$ are Poisson Lie groups with respective Lie bialgebras $(\g, \delta_\g)$, $(\g^*, \delta_{\g^*})$, and assume that the {\it dressing actions} 
\begin{align*}
\varrho_\ssY: \g^* \to \XX^1(\Y), \hs  & \varrho_\ssY(\xi) = \pi_\ssY^\sharp(\mu_\ssY^*\xi^L), \hs \xi \in \g^*,   \\
\vartheta_\ssZ: \g \to \XX^1(\Z), \hs & \vartheta_\ssZ(x) = \pi_\ssZ^\sharp(\mu_\ssZ^* x^R), \hs x \in \g,
\end{align*}
restrict to respectively $\rho$ and $\lambda$ on the identity bisections of $\Y$ and $\Z$. When $\Y$, $\Z$ are source simply connected symplectic groupoids, $\mu_\ssY$, $\mu_\ssZ$ are the Poisson groupoid morphisms integrating the Lie bialgebroid morphisms $\rho^*: T^*Y \to \g$, $\lambda^*: T^*Z \to \g^*$, see Remark \ref {rem-ssc-gpoid}. Associated to the pair $((G, \piG), (G^*, \piGs))$, one has the {\it double symplectic groupoid} $(\Gamma, \pi_\sGam)$ constructed by Lu in \cite{lu:thesis} (see $\S$\ref{subsec-dble-gpoid} for details). That is, $(\Gamma, \pi_\sGam)$ is both a symplectic groupoid over $(G, \piG)$ and $(G^*, \piGs)$, and write $\Gamma_\sG$, $\Gamma_\sGs$ for $\Gamma$ thought of as a groupoid over $G$ and $G^*$ respectively. Given respective left Poisson actions $\rhd_\sG$, $\rhd_\sGs$ of $(\Gamma_\sG, \pi_\sGam)$ on $(\Y, \pi_\ssY)$ and of $(\Gamma_\sGs, -\pi_\sGam)$ on $(\Z, \pi_\ssZ)$, with respective moment maps $\mu_\ssY$, $\mu_\ssZ$, one can via a local Lagrangian bisection $\L$ of $ (\Gamma_\sGs \times \Gamma_\sG , (-\pi_\sGam) \times \pi_\sGam)$ ``twist" the multiplication on the direct product Poisson groupoid 
$$
(\Y \rightrightarrows Y, \pi_\ssY) \times (\Y \rightrightarrows Z, \pi_\ssZ)
$$
to obtain a local Poisson groupoid over $(Y \times Z, \:  \piY \times_{(\rho, \lambda)} \piZ)$. More precisely, 
\begin{equation} \label{intro-L}
\L = (O_\sGam)_{\diag} \subset \Gamma^2
\end{equation}
is the diagonal copy of a particular open subset $O_\sGam$ of $\Gamma$, and 
\begin{equation} \label{intro-mult}
(\tilde{y}_1, \tilde{z}_1) \cdot (\tilde{y}_2, \tilde{z}_2) = (\tilde{y}_1(\gamma \rhd_\sG \tilde{y}_2), \: (\gamma \rhd_\sGs \tilde{z}_1)\tilde{z}_2), \hs \hs \tilde{y}_i \in \Y, \; \tilde{z}_i \in \Z, 
\end{equation}
where $(\gamma, \gamma) \in \L$ and $(\tilde{y}_1, \: \gamma \rhd_\sG \tilde{y}_2)$, respectively $(\gamma \rhd_\sGs \tilde{z}_1, \: \tilde{z}_2)$ are composable pairs in $\Y$ and $\Z$, defines a local groupoid multiplication on $\Y \times \Z$ which is compatible with the direct product Poisson structure $\pi_\ssY \times \pi_\ssZ$, and such that $(\Y \times_\ssL \Z, \pi_\ssY \times \pi_\ssZ)$ is a Poisson groupoid over $(Y \times Z, \: \piY \times_{(\rho, \lambda)} \piZ)$, where $\Y \times_\ssL \Z$ denotes $\Y \times \Z$ equipped with the local groupoid multiplication in \eqref{intro-mult}. 

The Lagrangian bisection $\L$ is closely related to the {\it global $\R$-matrix} of the Drinfeld double $(D, \piD)$ of $(G, \piG)$ constructed by Weinstein and Xu in \cite{wein-xu}. These $\R$-matrices are Lagrangian submanifolds in the cartesian square of double symplectic groupoids of quasitriangular Poisson Lie groups, which satisfy a classical analogue of the quantum Yang-Baxter equation. We show that under an appropriate isomorphism, the global $\R$-matrix of $(D, \piD)$ is essentially the cartesian product in $\Gamma^4$ of $\L$ with the identity bisections in $\Gamma_\sG$ and $\Gamma_\sGs$.

The groupoid multiplication in \eqref{intro-mult} is an analogue of a construction in quantum algebra appearing in \cite{mou:pas}, which was used to quantize mixed product Poisson structures. If $H$ is a Hopf algebra with a quasitriangular $R$-matrix $R \in H \otimes H$ and if $A$, $B$ are $H$-module algebras, then 
\begin{equation} \label{intro-R}
(a_1 \otimes b_1) \cdot (a_2 \otimes b_2) = a_1(Y^i a_2) \otimes (X_ib_1)b_2, \hs a_j \in A, \; b_j \in B,
\end{equation}
where $R = X_i \otimes Y^i$, defines an associative multiplication, called in \cite{mou:pas} a ``twist" of the tensor product algebra $A \otimes B$. One observes the analogous role played by $\L$ and $R$ in \eqref{intro-mult} and \eqref{intro-R}.

\subsection{Actions of double symplectic groupoids on generalised double Bruhat cells}

Returning to a connected complex semisimple Lie group $G$ with the standard multiplicative Poisson structure $\pist$, one has the pair $((B, \pist), (B_-, -\pist))$ of dual Poisson Lie groups, and let $(\Gamma, \pi_\sGam)$ be its associated double symplectic groupoid. For any $\bfu, \bfv \in W^n$, one has Poisson maps 
$$
\mu_+: (G^{\bfu, \bfv}, \tilde{\pi}_{n,n}) \to (B, \pist), \hs  \mu_-: (G^{\bfu, \bfv}, \tilde{\pi}_{n,n}) \to (B_-, \pist), 
$$
see $\S$\ref{sec-lhd-BB_-} a precise definition, and when $\bfu = \bfv$, both $\mu_+: \G^{\bbfu, \bbfu} \to B$, $\mu_-: \G^{\bbfu, \bbfu} \to B_-$ are groupoid morphisms. A third main result of this paper, Theorem \ref{thm-lhd-BB_-}, is that $G^{\bfu, \bfv}$ admits a Poisson action of both symplectic groupoids $(\Gamma_\sB, \pi_\sGam)$ and $(\Gamma_{\sB_-}, -\pi_\sGam)$ with respective moment maps $\mu_+$ and $\mu_-$.

We can now explain how the proof of Theorem \ref{thm-main-Guu} proceeds. Let $\bfu \in W^n$ and $\bfv \in W^m$.  Arguing by induction on $n$ and $m$, one can assume that both $(\G^{\bbfu, \bbfu},\pi_{\bbfu, \bbfu})$ and $(\G^{\bbfv, \bbfv}, \pi_{\bbfv, \bbfv})$ are Poisson groupoids, and apply the theory described in $\S$\ref{intro-subsec1} to obtain a local Poisson groupoid
$$
(\G^{\bbfu, \bbfu} \times_\ssL \G^{\bbfv, \bbfv}, \; \pi_{\bbfu, \bbfu} \times \pi_{\bbfv, \bbfv}). 
$$
In fact, we show that a quotient of this local Poisson groupoid lies in $(\G^{\bbfw, \bbfw}, \pi_{\bbfw, \bbfw})$ as a Zariski open subset, where $\bfw = (\bfu, \bfv)$ and $l = n+m$, and it follows by continuity that $(\G^{\bbfw, \bbfw}, \pi_{\bbfw, \bbfw})$ is a Poisson groupoid. 

Finally, a word about what is not in this paper. We do not prove that the symplectic leaves in $(G^{\bfw, \bfw}, \tilde{\pi}_{l,l})$ inherit a structure of a symplectic groupoid. One would first need to have a description of these symplectic leaves, generalising the description in \cite{k-z:leaves, Lu-Mou:double-B-cell} of the symplectic leaves in the standard complex semisimple Poisson Lie group $(G, \pist)$. We plan to address this issue in a subsequent paper.

\subsection*{}

The paper is organised as follows. Section $\S$\ref{sec-PL-bialg} is a brief recall on the theory of Lie bialgebras and Poisson Lie groups, and in $\S$\ref{sec-action-double} we recall from \cite{lu:thesis} the notion of a double symplectic groupoid associated to a pair of dual Poisson Lie groups. In particular, we adapt the criteria appearing in \cite{Liu-Wei-Xu:dirac} for a Lie groupoid action of a Poisson groupoid to be Poisson to the case of a double symplectic groupoid. In $\S$\ref{sec-lag-PLgps}, we discuss the theory of global $\R$-matrices developed by Weinstein and Xu in \cite{wein-xu}. We show that for the Drinfeld double $(D, \piD)$ of a pair $((G, \piG), (G^*, \piGs))$ of dual Poisson Lie groups, the global $\R$-matrix is, under an appropriate isomorphism, the cartesian product in $\Gamma^4$ of $\L$ and the identity bisections in $\Gamma_\sG$ and $\Gamma_\sGs$, where $\L$ is as in \eqref{intro-L}. Section $\S$\ref{sec-local-p-gpoid} is where the first main result of this paper, Theorem \ref{main-thm-gpoid}, appears. It is a construction of a local Poisson groupoid over a mixed product Poisson structure as explained in $\S$\ref{intro-subsec1}. Section $\S$\ref{sec-std-PL} is about the pair $((B, \pist), (B_-, -\pist))$ of dual Poisson Lie groups associated to a standard complex semisimple Poisson Lie group $(G, \pist)$, and in $\S$\ref{sec-dble-bruhat}, we recall the construction of generalised (double) Bruhat cells from \cite{Lu-victor:Tleaves}. We prove in $\S$\ref{sec-lhd-BB_-} the second main result of this paper: every generalised double Bruhat cell admits a Poisson action by the symplectic groupoids of both $(B, \pist)$ and $(B_-, -\pist)$. In $\S$\ref{sec-main-thm-Gww},  we prove last main result of this paper, that every generalised double Bruhat cell $(\G^{\bbfw, \bbfw}, \pi_{\bbfw, \bbfw})$ is naturally a Poisson groupoid over $(\O^{\bfw}, \pi_l)$.

\subsection{Notation} \label{subsec-nota}

\subsubsection{}

By manifold, we mean either a real smooth or a complex manifold. Maps between manifolds and tensor fields on manifolds are understood to be either smooth or holomorphic. By diffeomorphism, we means either $C^\infty$ diffeomorphism or holomorphic diffeomorphism.  

\subsubsection{} \label{nota-quotient}
If $G$ is a group and $Q_0, Q_1, \ldots, Q_n$ subgroups of $G$, we will denote by 
$$
Q_0 \backslash G \times_{\scriptscriptstyle Q_1} G \times_{\scriptscriptstyle Q_2} \cdots  \times_{\scriptscriptstyle Q_{n-1}} G/Q_n 
$$
the quotient of $G^n$ be the right action of $Q_0 \times \cdots \times Q_n$ given by 
$$
(g_1, g_2, \ldots, g_n) \cdot (q_0, q_1, \ldots, q_n) = (q_0^{-1}g_1q_1, q_1^{-1}g_2q_2, \ldots, q_{n-1}^{-1}g_nq_n), \hs g_j \in G, \; q_j \in Q_j,
$$
and if $Q_0 \backslash G \times_{\scriptscriptstyle Q_1} \cdots  \times_{\scriptscriptstyle Q_{n-1}} G/Q_n$ is denoted by $Z$, we will denote the quotient map by $\varpi_\sZ: G^n \to Z$ and elements of $Z$ by 
$$
[g_1, \ldots, g_n]_\sZ = \varpi_\sZ(g_1, \ldots, g_n), \hs g_j \in G. 
$$
If $S_1, \ldots, S_n$ are subsets of $G$ such that $S_j$ is left $Q_{j-1}$-invariant and right $Q_j$-invariant for $1 \leq j \leq n$, let 
$$
Q_0 \backslash S_1 \times_{\scriptscriptstyle Q_1} S_2 \times_{\scriptscriptstyle Q_2} \cdots  \times_{\scriptscriptstyle Q_{n-1}}S_n/Q_n = \varpi_\sZ(S_1 \times \cdots \times S_n) \subset Z. 
$$
If $Q_0 = \{e\}$ we denote $Z$ by $G \times_{\scriptscriptstyle Q_1} \cdots  \times_{\scriptscriptstyle Q_{n-1}} G/Q_n$, and if $Q_n = \{e \}$, we denote $Z$ by $Q_0 \backslash G \times_{\scriptscriptstyle Q_1} \cdots  \times_{\scriptscriptstyle Q_{n-1}} G$. 

\subsubsection{}

If $G$ is a real or complex Lie group with Lie algebra $\g$, we denote by $l_g, r_g: G \to G$ respectively the left and right multiplication by $g \in G$, and for $x \in (\otimes \g) \oplus (\otimes \g^*)$, we denote respectively by $x^L$ and $x^R$ the left and right invariant tensor field on $G$ whose value at the identity $e \in G$ is $x$. If $\lambda: G \times M \to M$ is a left action of $G$ on a manifold $M$, we use the same symbol to denote the left Lie algebra action $\lambda: \g \to \XX^1(M)$, defined by 
$$
\lambda(x)(m) = \frac{d}{dt} \!\mid_{t=0} \lambda(\exp(tx), m), \hs x \in \g, \; m \in M. 
$$
Similarly, if $\rho: M \times G \to M$ is a right action, let $\rho: \g \to \XX^1(M)$ be defined as $\rho(x)(m) =  d/dt \mid_{t=0} \rho(m, \exp(tx))$, $x \in \g$, $m \in M$.

\subsection*{Acknowledgements}

The author would like to thank Jiang-Hua Lu for stimulating discussions, as well as Rui Loja Fernandes, Ioan Marcut, and Travis Li Songhao for their helpful explanations. 


\section{Poisson Lie groups and Lie bialgebras} \label{sec-PL-bialg}

We recall in this section basic facts about Poisson Lie groups and Lie bialgebras, and refer to \cite[Section 2]{Lu-Mou:mixed} for additional details. 

\subsection{Poisson Lie groups and Lie bialgebras}

Let $\g$ be a finite dimensional, real or complex Lie algebra. A {\it Lie bialgebra structure on $\g$} is a map $\delta_\g: \g \to \wedge^2 \g$ whose dual map $\delta_\g^*: \wedge^2 \g^* \to \g^*$ is a Lie bracket on $\g^*$, and which satisfies the cocycle condition 
$$
\delta_\g[x, y] = [x, \delta_\g(y)] + [\delta_\g(x), y], \hs x, y \in \g, 
$$
and one says that $(\g, \delta_\g)$ is a {\it Lie bialgebra}. Then $(\g^*, \delta_{\g^*})$ is also a Lie bialgebra, called the {\it dual Lie bialgebra} of $(\g, \delta_\g)$, where $\g^*$ is equipped with the Lie bracket $\delta_\g^*$ and $\delta_{\g^*}: \g^* \to \wedge^2 \g^*$ is the dual of the Lie bracket on $\g$. Equip $\d = \g \oplus \g^*$ with the symmetric non-degenerate bilinear form
\begin{equation} \label{eq-lara_d}
\la x + \xi, y + \eta \ra_\d = \la x, \eta \ra + \la y, \xi \ra, \hs x, y \in \g, \; \xi, \eta \in \g^*.
\end{equation}
There is a unique Lie bracket \cite[Formula (2.2)]{Lu-Mou:mixed} on $\d$ such that $\g$, $\g^*$ are Lie subalgebras and such that $\lara_\d$ is ad-invariant, and one says that $(\d, \lara_\d)$ is the {\it double Lie algebra} of $(\g, \delta_\g)$. Equivalently, a {\it Manin triple} $((\d, \lara_\d), \g, \g')$ consists of a quadratic Lie algebra $(\d, \lara_\d)$, and two Lagrangian subalgebras $\g$, $\g'$ of $(\d, \lara_\d)$ such that $\d = \g \oplus \g'$ as a vector space. Identifying $\g$ and $\g'$ as dual vector spaces via $\lara_\d$, one obtains Lie bialgebra structures $\delta_\g: \g \to \wedge^2 \g$ and $\delta_{\g'}: \g' \to \wedge^2 \g'$, respectively the dual of of the Lie bracket on $\g'$ and $\g$, and one says that $((\g, \delta_\g), (\g', \delta_{\g'}))$ is a {\it pair of dual Lie bialgebras}, with double Lie algebra $(\d, \lara_\d)$. The element 
\begin{equation} \label{eq-r-matrix} 
\Lambda_{\g, \g'} =  x_i \wedge \xi^i \in \d \otimes \d,
\end{equation}
where $(x_i)$ is any basis of $\g$ and $(\xi^i)$ the dual basis of $\g'$ with respect to $\lara_\d$, is called the {\it skew-symmetric $r$-matrix associated to the Lagrangian splitting $\d = \g + \g'$}, and  
$$
\delta_\d: \d \to \wedge^2 \d,  \hs \delta_\d(a) = [a, \Lambda_{\g, \g'}], \hs a \in \d,
$$
is a Lie bialgebra structure on $\d$ such that both $(\g, \delta_\g)$ and $(\g^*, -\delta_{\g^*})$ are sub- Lie bialgebras of $(\d, \delta_\d)$. One says that $(\d, \delta_\d)$ is the {\it double Lie bialgebra of $(\g, \delta_\g)$}. In particular, let $(\g, \delta_\g)$ be any Lie bialgebra with dual Lie bialgebra $(\g^*, \delta_{\g^*})$ and double Lie bialgebra $(\d, \delta_\d)$. Equip the direct product Lie algebra $\d \oplus \d$ with the bilinear form
$$
\la (a_1, b_1), (a_2, b_2) \ra_{\d \oplus \d} = \la a_1, a_2 \ra_\d - \la b_1, b_2 \ra_\d, \hs a_i, b_i \in \d,
$$
and let $\d_{diag} \subset \d \oplus \d$ be the diagonal subalgebra and $\d' =  \g^* \oplus \g \subset \d \oplus \d$. Then $((\d \oplus \d, \lara_{\d \oplus \d}), \d_{diag}, \d')$ is a Manin triple, and the skew-symmetric $r$-matrix associated to the Lagrangian splitting $\d \oplus \d = \d_{\diag} + \d'$ is 
\begin{equation} \label{eq-r^(2)}
\Lambda_{\g, \g'}^{(2)} = (\Lambda_{\g, \g'}, \Lambda_{\g, \g'}) - \Lambda \in (\d \oplus \d) \wedge (\d \oplus \d),
\end{equation}
where 
\begin{equation} \label{eq-t}
\Lambda = (\xi^i, 0) \wedge (0, x_i) \in \wedge^2 \d'. 
\end{equation}
One has $(\d, \delta_\d) \cong (\d_{\diag}, \delta_{\d_{\diag}})$ under the isomorphism $a \mapsto (a, a)$, $a \in \d$, thus $(\d', \delta_{\d'})$ is isomorphic to the dual Lie bialgebra of $(\d, \delta_\d)$. 

A {\it multiplicative} Poisson structure on a real or complex Lie group $G$ is a smooth or holomorphic Poisson bivector field $\piG$ on $G$ such that the multiplication map $G \times G \to G$ is Poisson, when $G \times G$ is equipped with $\piG \times \piG$, and one also says that $(G, \piG)$ is a {\it Poisson Lie group}. Equivalently, $\piG$ is multiplicative if it satisfies 
$$
l_g \piG(h)  + r_h \piG(g) = \piG(gh), \hs g, h \in G. 
$$
Let $\g$ be the Lie algebra of $G$. Then $\delta_\g(x) = [x^R, \piG](e)$, $x \in \g$, is a Lie bialgebra structure on $\g$, and one says that $(\g, \delta_\g)$ is the {\it Lie bialgebra of $(G, \piG)$}. The adjoint action of $\g$ on $\d$ integrates to an action of $G$ on $\d$, given by 
\begin{equation} \label{eq-Ad}
\Ad_g(x+ \xi) = \Ad_gx + r_{g^{-1}}\piG^\sharp(\xi^L)(g) + \Ad_{g^{-1}}^* \xi, \hs x \in \g, \; \xi \in \g^*, \; g \in G.
\end{equation}
A {\it pair of dual  Poisson Lie groups} is a pair of Poisson Lie groups $((G, \piG), (G^*, \piGs))$, where the Lie bialgebra $(\g, \delta_{\g})$ of $(G, \piG)$ and the Lie bialgebra $(\g^*, \delta_{\g^*})$ of $(G^*, \piGs)$ form a pair of dual Lie bialgebras. 

Let $(Y, \piY)$ be a Poisson manifold, $(G, \piG)$ a Poisson Lie group with Lie bialgebra $(\g, \delta_\g)$, and $\rho: Y \times G \to Y$ a right action of $G$ on $Y$. One says that $\rho$ is a {\it right Poisson action of $(G, \piG)$ on $(Y, \piY)$} if $\rho$ is a Poisson map, when $Y \times G$ is equipped with $\piY \times \piG$, and left Poisson actions are similarly defined. Equivalently, $\rho$ is a right Poisson action if 
$$
l_y \piG(g) + r_g\piY(y) = \piY(\rho(y, g)),
$$
where for $y \in Y$ and $g \in G$, one has the maps $l_y: G \to Y$, $l_yg = \rho(y, g)$ and $r_g: Y \to Y$, $r_gy = \rho(y, g)$. Let $\sigma: \g \to \XX^1(Y)$ be a right or left Lie algebra action of $\g$ on $Y$. One says that $\sigma$ is a {\it Poisson action of $(\g, \delta_\g)$ on $(Y, \piY)$} if 
$$
[\sigma(x), \piY] = \sigma(\delta_\g(x)), \hs x \in \g.
$$
It is well known that a Poisson action of a Poisson Lie group induces a Poisson action of its Lie bialgebra.

\subsection{Mixed product Poisson structures associated to actions of Lie bialgebras} \label{subsec-mixed-prod}

Let $(\g, \delta_\g)$ be a Lie bialgebra with dual Lie bialgebra $(\g^*, \delta_{\g^*})$ and double Lie bialgebra $(\d, \delta_\d)$, let $(Y, \piY)$, $(Z, \piZ)$ be Poisson manifolds, and let 
$$
\rho: \g^* \to \XX^1(Y), \hs \lambda: \g \to \XX^1(Z),
$$
be respectively a right Poisson action of $(\g^*, \delta_{\g^*})$ on $(Y, \piY)$ and a left Poisson action of $(\g, \delta_\g)$ on $(Z, \piZ)$. Then 
$$
\piY \times_{(\rho, \lambda)} \piZ : = (\piY, \: \piZ) - (\rho(\xi^i), \: 0) \wedge (0, \: \lambda(x_i)) \in \XX^2(Y \times Z), 
$$
where $(x_i)$ is any basis of $\g$ and $(\xi^i)$ the dual basis of $\g^*$, is a Poisson structure on $Y \times Z$, called in \cite{Lu-Mou:mixed} the {\it mixed product Poisson structure associated to the pair $(\rho, \lambda)$}, and 
$$
\sigma: \d' \to \XX^1(Y \times Z), \hs \sigma(\xi, x) = (\rho(\xi), -\lambda(x)), \hs (\xi, x) \in \d',
$$
is a right Poisson action of $(\d', \delta_{\d'})$ on $(Y \times Z, \: \piX \times_{(\rho, \lambda)} \piY)$. More generally, if $\rho$ is the restriction to $\g^*$ of a right Poisson action $\sigma_\sY: \d \to \XX^1(Y)$ of $(\d, -\delta_\d)$ on $(Y, \piY)$ and if $-\lambda$ is the restriction to $\g$ of a right Poisson action $\sigma_\sZ: \d \to \XX^1(Z)$ of $(\d, -\delta_\d)$ on $(Z, \piZ)$, then 
$$
\sigma: \d \oplus \d \to \XX^1(Y \times Z), \hs \sigma(a, b) = (\sigma_\sY(a), \sigma_\sZ(b)), \hs a, b \in \d,
$$
is a right Poisson action of $(\d \oplus \d, -\delta_{\d \oplus \d})$ on $(Y \times Z, \: \piX \times_{(\rho, \lambda)} \piY)$, where 
$$
\delta_{\d \oplus \d}(a,b) = [(a,b), \Lambda_{\g, \g'}^{(2)}], \hs a, b \in \d. 
$$

\section{Poisson actions of double symplectic groupoids} \label{sec-action-double}

\subsection{Symplectic and Poisson groupoids} \label{subsec-poss-gpoid}

We recall in this subsection basic facts about symplectic and Poisson groupoids that will be needed later, and refer to \cite{Mackenzie-Xu, wein:cois-calc, Xu:Poioid} for further details.

If $\G \rightrightarrows Y$ is a Lie groupoid with source and target maps $\theta, \tau: \G \to Y$, inverse map $\iota: \G \to \G$, and identity bisection $\varepsilon: Y \to \G$, we use the convention that the multiplication is defined on 
$$
\G^{(2)} = \{ (g_1, g_2) \in \G^2 : \tau(g_1) = \theta(g_2) \}, 
$$
and when no confusion is possible, we will write the groupoid multiplication as concatenation $g_1g_2$, for $(g_1, g_2) \in \G^{(2)}$. Throughout this paper, by a {\it local Lie groupoid}, we will mean a $3$-associative local Lie groupoid in the sense of \cite[Definition 2.7]{assoc_int}. That is, a manifold $\G$ equipped with submersions $\theta, \tau: \G \to Y$ to a base manifold $Y$, an embedding $\varepsilon: Y \to \G$, a multiplication defined on an open neighbourhood $\G^{(2)}_0 \subset \G^{(2)}$ of 
$$
\{ (\varepsilon(y), g) : y = \theta(g), \; y \in Y, g \in \G \} \cup \{ (g, \varepsilon(y)) : \tau(g) = y, \; y \in Y, g \in \G \},
$$
an involution $\iota: \G^{(-1)} \to \G^{(-1)}$ of an open neighbourhood $\G^{(-1)} \subset \G$ of $\varepsilon(Y)$, and satisfying the usual axioms of a Lie groupoid wherever these make sense. A {\it local Poisson groupoid} is a pair $(\G \rightrightarrows Y, \pi)$ where $\G \rightrightarrows Y$ is a local Lie groupoid, and $\pi$ a Poisson bivector field on $\G$ such that the graph of the multiplication 
$$
\{(g_1, \: g_2, \: g_1g_2) : (g_1, g_2) \in \G^{(2)}_0 \} \subset \G \times \G \times \G
$$
is a coisotropic submanifold for the Poisson structure $\pi \times \pi \times (-\pi)$ on $\G \times \G \times \G$. One has a well defined Poisson bivector field $\piY$ on $Y$ such that $\theta(\pi) = -\tau(\pi) = \piY$, and we say that $(\G \rightrightarrows Y, \pi)$ is a local Poisson groupoid {\it over} $(Y, \piY)$. If $\pi$ is non-degenerate and $\dim \G = 2 \dim Y$, one says that $(\G \rightrightarrows Y, \pi)$ is a {\it local symplectic groupoid}, or a {\it symplectic groupoid} if $\G \rightrightarrows Y$ is a Lie groupoid. 


Let $(\G \rightrightarrows Y, \pi)$ be a symplectic groupoid over $(Y, \piY)$ with source map $\theta: \G \to Y$, let $(X, \piX)$ be a Poisson manifold with a map $\mu: X \to Y$, and let 
$$
X \ast \G := \{ (x, g) \in X \times \G: \mu(x) = \theta(g) \} .
$$
A {\it right Poisson action} of $(\G, \pi)$ on $(X, \piX)$ with moment map $\mu$ is a right Lie groupoid action $\lhd$ of $\G$ on $X$ with moment map $\mu$, such that the graph of the action map 
$$
\{(x, g, x \lhd g): (x, g) \in X \ast \G\}
$$
is a coisotropic submanifold of $X \times \G \times X$, equipped with the Poisson structure $\piX \times \pi \times (-\piX) $. Then $\mu: (X, \piX) \to (Y, \piY)$ is automatically anti-Poisson, and we also say that $\lhd$ is a {\it Poisson action of $(\G, \pi)$ on $\mu$}. Left Poisson actions of symplectic groupoids are similarly defined, and in particular,
$$
g \rhd x :  = x \lhd \iota(g), \hs (x ,\iota(g)) \in X \ast \G,
$$
where $\iota$ is the inverse of $\G$, defines a left Poisson action of $(\G, -\pi)$ on $\mu$.  If $S$ is a local bisection of $\G \rightrightarrows Y$, $R_\sS$ denotes the action of $S$ on $X$, i.e $R_\sS(x) = x \lhd g$, where $g \in S$ with $\mu(x) = \theta(s)$. We use the same symbol $R_\sS$ for the action of $S$ on $\G$ itself by right multiplication, and similarly $L_\sS$ denotes the action of $S$ on $\G$ by left multiplication. In particular, by \cite[Theorem 7.1]{Liu-Wei-Xu:dirac} if $S$ is a local Lagrangian bisection of $(\G \rightrightarrows Y, \pi)$, $R_\sS$ is a local Poisson isomorphism of $(X, \piX)$.

\subsection{Double symplectic groupoids} \label{subsec-dble-gpoid}

We recall in this subsection the construction of a double symplectic groupoid of a pair of dual Poisson Lie groups given in \cite[Chapter 4]{lu:thesis} and \cite{LW1}. 

Let $((G, \piG), (G^*, \piGs))$ be a pair of dual Poisson Lie groups, with pair of dual Lie bialgebras $((\g, \delta_\g)$, $(\g^*, \delta_{\g^*}))$ and double Lie algebra $(\d, \lara_\d)$. Throughout this paper, we will make the simplifying assumption that $G$ and $G^*$ are subgroups in a {\it Drinfeld double} $D$, that is a Lie group $D$ with Lie algebra $\d$. With $\Lambda_{\g, \g^*} \in \wedge^2 \d$ as in  \eqref{eq-r-matrix}, one has the multiplicative Poisson structure $\piD = \Lambda_{\g, \g^*}^L - \Lambda_{\g, \g^*}^R$ such that $(\d, \delta_\d)$ is the Lie bialgebra of $(D, \piD)$, and such that both $(G, \piG)$ and $(G^*, -\piGs)$ are Poisson Lie subgroups of $(D, \piD)$. One also has the Poisson structure $\pi^+_\sD = \Lambda_{\g, \g^*}^R + \Lambda_{\g, \g^*}^L$, which is non-degenerate on the open subset $D_0 = GG^* \cap G^*G$ in $D$. Let 
$$
\Gamma = \{ (g, u, u', g') \in G \times G^* \times G^* \times G : \; gu = u' g' \},  
$$
let $Q = G \cap G^*$, and let $Q^2$ act on $\Gamma$ by 
$$
(g, u, u', g') \cdot (q_1, q_2) = (gq_1, \: q_1^{-1}u, \: u' q_2, \: q_2^{-1}g'), \hs  (g, u, u', g') \in \Gamma, \; q_i \in Q.
$$
Then 
\begin{equation} \label{eq-p}
p: \Gamma \to D, \hs p(g, u, u', g') = gu = u' g', \hs (g, u, u', g') \in \Gamma, 
\end{equation}
induces an isomorphism between $\Gamma/ (Q^2)$ and $D_0$, and since $Q$ is a discrete subgroup of $D$, there is a unique non-degenerate Poisson structure $\pi_\sGam$ on $\Gamma$ which lifts $\pi^+_\sD \!\mid_{\sD_0}$. Let 
$$
p_\sL: \Gamma \to G \times G^* \hs \text{and} \hs p_\sR: \Gamma \to G^* \times G
$$ 
be respectively the projections onto the first two and last two factors. Then both $p_\sL$ and $p_\sR$ are local diffeomorphisms, and it is shown in \cite{lu:thesis, LW1} that one has 
\begin{align} 
p_\sL(\pi_\sGam) & = (\piG, \piGs) - (x_i^L, 0) \wedge (0, (\xi^i)^R),  \label{eq-comp-pi+}   \\
p_\sR(\pi_\sGam) & = (\piGs, \piG) - ((\xi^i)^L, 0) \wedge (0, x_i^R), \notag
\end{align}
where $(x_i)$ is a basis of $\g$ and $(\xi^i)$ the dual basis of $\g^*$. I.e $\pi_\sGam$ is locally via $p_\sL$ and $p_\sR$ a mixed product Poisson structure. Moreover, $\Gamma$ has two groupoid structures: one of a groupoid over $G^*$, given by 
\begin{align*}
&\mbox{source map}: \; \theta_\sGs(g, u, u', g')= u,   \\
&\mbox{target map}: \; \tau_\sGs(g, u, u', g') = u',  \\
&\mbox{identity bisection}: \; \varepsilon_\sGs(u) = (e, u, u, e), \hs  u \in G^*,  \\
&\mbox{inverse map}: \; \iota_\sGs(g, u, u', g') = (g^{-1}, u', u, g'^{-1}),  \\
&\mbox{multiplication}: \; \mbox{when} \; u'_1 = u_2,  \hs (g_1, u_1, u'_1, g'_1) \star (g_2, u_2, u'_2, g'_2) = (g_2g_1, \: u_1, \: u'_2, \: g'_2g'_1);
\end{align*}
and one of a groupoid over $G$, given by 
\begin{align*}
&\mbox{source map}: \; \theta_\sG(g, u, u', g')= g,   \\
&\mbox{target map}: \; \tau_\sG(g, u, u', g') = g',  \\
&\mbox{identity bisection}: \; \varepsilon_\sG(g) = (g, e, e, g), \hs  g \in G,  \\
&\mbox{inverse map}: \; \iota_\sG(g, u, u', g') = (g', u^{-1}, u'^{-1}, g),  \\
&\mbox{multiplication}: \; \mbox{when} \; g'_1 = g_2,  \hs (g_1, u_1, u'_1, g'_1) (g_2, u_2, u'_2, g'_2) = (g_1, \: u_1u_2, \: u'_1u'_2, \: g'_2).
\end{align*}
We will denote $\Gamma$ as $\Gamma_\sGs$ and $\Gamma_\sG$, when thought of as a groupoid over $G^*$ and $G$ respectively (notice that we write multiplication in $\Gamma_\sG$ using concatenation and in $\Gamma_\sGs$ using $\star$). Then $(\Gamma_\sGs \rightrightarrows G^*, \pi_\sGam)$ is a symplectic groupoid over $(G^*, \piGs)$ and  $(\Gamma_\sG \rightrightarrows G, \pi_\sGam)$ is a symplectic groupoid over $(G, \piG)$. Note that 
\begin{equation} \label{eq-iota-morphism}
\iota_\sG(\gamma_1 \star \gamma_2) = \iota_\sG(\gamma_1) \star \iota_\sG(\gamma_2) \hs \text{and} \hs \iota_\sGs(\gamma'_1 \gamma'_2) = \iota_\sGs(\gamma'_1) \iota_\sGs(\gamma'_2),
\end{equation}
for $(\gamma_1, \gamma_2) \in \Gamma_\sGs^{(2)}$ and $(\gamma'_1, \gamma'_2) \in \Gamma_\sG^{(2)}$.


\subsection{Local dressing actions} \label{subsec-dressing}

Let $O_\sGam \subset \Gamma$ be the maximal open subset containing $\varepsilon_\sG(G) \cup \varepsilon_\sGs(G^*)$ on which the restriction $p \!\!\mid_{\scriptscriptstyle O_\sGam}: O_\sGam \to D_0$ is a diffeomorphism onto its image $O_\sD = p(O_\sGam) \subset D$. Let $O_{\scriptscriptstyle G, G^*}= p_\sL(O_\sGam)$ and $O_{\scriptscriptstyle G^*, G} = p_\sR(O_\sGam)$, so that $O_\sGam$ is the graph of an invertible map   
$$
O_{\sG, \sGs} \to O_{\sGs, \sG}, \hs (g, u) \mapsto (g[u], g^u), \hs (g, u) \in O_{\sG, \sGs},
$$
with inverse 
$$
O_{\sGs, \sG} \to O_{\sG, \sGs}, \hs (u', g') \mapsto (u'[g'], u'^{g'}), \hs (u', g') \in O_{\sGs, \sG},
$$
and for $(g, u) \in O_{\sG, \sGs}$ and $(u', g') \in O_{\sGs, \sG}$, let 
\begin{equation} \label{eq-gamma_gu}
\gamma_{g,u} = (g, \: u, \: g[u], \: g^u) \in O_{\sGam} \hs \text{and} \hs \gamma^{u', g'} = (u'[g'],\: u'^{g'}, \: u',\: g') \in O_{\sGam}.  
\end{equation}
Then one has the {\it local dressing actions}
$$
(g, u) \mapsto g^u, \hs (g, u) \mapsto g[u], \hs (g, u) \in O_{\sG, \sGs}, 
$$
respectively a right local action of $G^*$ on $G$ and a left local action of $G$ on $G^*$, and 
$$
(u', g') \mapsto u'^{g'}, \hs (u', g') \mapsto u'[g'], \hs (u', g') \in O_{\sGs, \sG}, 
$$
a right local action of $G$ on $G^*$ and a left local action of $G^*$ on $G$. One checks that the local dressing actions satisfy the multiplicativity conditions
\begin{align}
\gamma^{u_1u_2, \: g} & = \gamma^{u_1, \: u_2[g]} \gamma^{u_2, \: g},   \label{eq-gam-mult}  \\
\gamma^{u, \: g_1g_2} & = \gamma^{u^{g_1}, \: g_2} \star \gamma^{u, \: g_1}, \notag
\end{align}
whenever the terms involved are defined. 

\begin{lem}
One has 
\begin{equation} \label{lem-Ad}
\frac{d}{dt} \mid_{t = 0} g[\exp(t\xi)] = \Ad_{g^{-1}}^*\xi \hs \text{and} \hs \frac{d}{dt} \mid_{t = 0} \exp(tx)^u = \Ad_u^*x. 
\end{equation}
\end{lem}
\begin{proof}
Since 
$$
\Ad_g \xi = \frac{d}{dt} \mid_{t = 0} g \exp(t\xi)g^{-1} =  \frac{d}{dt} \mid_{t = 0} g[\exp(t\xi)] g^{\exp(t\xi)}g^{-1}
$$
and since multiplication in $D$ induces a local isomorphism between $G^* \times G$ and $D$, the first relation follows from \eqref{eq-Ad}. The second relation is similarly proved. 
\end{proof}

\begin{rem} \label{rem-complete}
Recall that $(G, \piG)$ is said to be {\it complete} if the multiplication in $D$ induces a diffeomorphism $G \times G^* \cong D$. In such a case one can identify $(\Gamma, \pi_\sGam)$ with $(D, \pi^+_\sD)$ via the map $p: \Gamma \cong D$ in \eqref{eq-p}, and $\Gamma_\sG$ becomes the action groupoid associated to the right group action $G \times G^* \to G$, $(g, u) \mapsto g^u$, and similarly, $\Gamma_\sGs$ becomes the action groupoid associated to the left action $G \times G^* \to G^*$, $(g, u) \mapsto g[u]$.
\hfill $\diamond$
\end{rem}

\subsection{Poisson actions of double symplectic groupoids}

We specialise in this subsection the criteria in \cite[Theorem 7.1]{Liu-Wei-Xu:dirac} for a Lie groupoid action of a Poisson groupoid to be Poisson to the case of $(\Gamma, \pi_\sGam)$. We fix first a particular local bisection of $\Gamma_\sG \rightrightarrows G$ through any point of $\Gamma$. For $u \in G^*$, one has the open neighbourhood  
$$
O_{\sG, u} = \{ g \in G : (g, u) \in O_{\sG, \sGs} \}
$$
of $e$ in $G$. The next Lemma \ref{lem-B_gamma} is straightforward. 

\begin{lem} \label{lem-B_gamma}
For $\gamma = (g, u, u', g') \in \Gamma$ with $g, g' \in G$ and $u, u' \in G^*$,
$$
S_\gamma = \{ (hg, \: u, \: h[u'], \: h^{u'}g') : h \in O_{\sG, u'} \}
$$ 
is a local bisection of $\Gamma_\sG \rightrightarrows G$ through $\gamma$. 
\end{lem}

\begin{lem} \label{pi-Rpi}
Let $\gamma = (g, u, u', g') \in \Gamma$ with $g, g' \in G$ and $u, u' \in G^*$. Then 
$$
p_\sL \left( \pi_\sGam(\gamma) - R_{\sS_\gamma} \pi_\sGam(\varepsilon_\sG(g))  \right) = (0, \piGs(u)).
$$
In particular, $\pi_\sGam(\gamma) - R_{\sS_\gamma} \pi_\sGam(\varepsilon_\sG(g))$ is tangent to the fibers of $\theta_\sG$. 
\end{lem}
\begin{proof}
It is clear from the definition of $S_\gamma$ that $p_\sL R_{\sS_\gamma} = r_{(e, u)} p_\sL$, thus by \eqref{eq-comp-pi+}, 
\begin{align*}
p_\sL \left( \pi_\sGam(\gamma) - R_{\sS_\gamma} \pi_\sGam(\varepsilon_\sG(g))  \right) & = (\piG(g), \piGs(u)) - (l_gx_i, 0) \wedge (0, r_u\xi^i) \\
   & \hs - r_{(e, u)} \left((\piG(g), 0) - (l_gx_i, 0) \wedge (0, \xi^i) \right)   \\ 
      & = (0, \piGs(u)). 
\end{align*}
\end{proof}

\begin{lem}
For $\xi \in \g^*$ and $g \in G$, one has 
$$
\pi_\sGam^\sharp(\tau_\sG^*(\xi^L))(\varepsilon_\sG(g)) = \frac{d}{dt} \!\mid_{t = 0} \gamma_{g, \:\exp(t\xi)}. 
$$
\end{lem}
\begin{proof}
Let $m_\sD: G^* \times G \to D$ be the multiplication map. By \eqref{eq-comp-pi+} and \eqref{eq-Ad}, one has 
\begin{align*}
(m_\sD \circ p_\sR) \left( \pi_\sGam^\sharp(\tau_\sG^*(\xi^L))(\varepsilon_\sG(g)) \right) & = m_\sD \left( p_\sR(\pi_\sGam)^\sharp(0, \: l_{g^{-1}}^*\xi) \right) = m_\sD(\Ad_{g^{-1}}^*\xi, \: \piG^\sharp(l_{g^{-1}}^*\xi))  \\
   & = r_g(\Ad_{g^{-1}}^*\xi) + \piG^\sharp(l_{g^{-1}}^*\xi) = l_g \xi,
\end{align*}
thus
$$
p_\sL \left( \pi_\sGam^\sharp(\tau_\sG^*(\xi^L))(\varepsilon_\sG(g)) \right) = (0, \xi).
$$
As $p_\sL$ is a local diffeomorphism, 
$$
\pi_\sGam^\sharp(\tau_\sG^*(\xi^L))(\varepsilon_\sG(g)) = p_\sL^{-1}(0, \xi) = \frac{d}{dt} \!\mid_{t = 0} \gamma_{g, \:\exp(t\xi)}. 
$$
\end{proof}

Let $(Y, \piY)$ be a Poisson manifold, 
$$
\mu: (Y, \piY) \to (G, \piG)
$$ 
a Poisson map, and $\lhd$ a right Lie groupoid action of $\Gamma_\sG$ on $\mu$. Recall from \cite[Section 2]{Lu-Mou:mixed} that one has the dressing action 
\begin{equation} \label{varrhoY}
\varrho_\sY : \g^* \to \XX^1(Y), \hs \hs  \varrho_\sY(\xi) =  \piY^\sharp(\mu^*\xi^L), \hs \xi \in \g^*,
\end{equation}
a right Poisson action of $(\g^*, -\delta_{\g^*})$ on $(Y, \piY)$. The next Proposition \ref{pro-liu.wei.xu-adapt} is a direct application of \cite[Theorem 7.1]{Liu-Wei-Xu:dirac} to $(\Gamma_\sG \rightrightarrows G, -\pi_\sGam)$ and the local bisections $S_\gamma$'s. 

\begin{pro} \label{pro-liu.wei.xu-adapt}
The Lie groupoid action $\lhd$ is a Poisson action of $(\Gamma_\sG \rightrightarrows G, -\pi_\sGam)$ if and only if 
\begin{align}
\varrho_\sY(\xi)(y) & = \frac{d}{dt} \!\mid_{t = 0} (y \lhd \gamma_{\mu(y), \:\exp(t\xi)}), \hs y \in Y, \; \xi \in \g^*,   \label{eq-dirac1}  \\
\piY(y \lhd \gamma) & = L_{y} \left( R_{\sS_\gamma} \pi_\sGam(\varepsilon_\sG(g)) - \pi_\sGam(\gamma) \right) + R_{\sS_\gamma} \piY(y), \hs (y, \gamma) \in Y \ast \Gamma_\sG, \label{eq-dirac2}
\end{align}
where in \eqref{eq-dirac2}, $g = \theta_\sG(\gamma)$ and $L_y: \theta_\sG(\mu(y))^{-1} \to Y$ is the map $L_y(\gamma) = y \lhd \gamma$. By Lemma \ref{pi-Rpi}, the first term in the right hand side of \eqref{eq-dirac2} is well defined. 
\end{pro}

The next Proposition \ref{pro-gp/gpoid-action} below, which will be used in  $\S$\ref{sec-lhd-BB_-},  gives a situation in which \eqref{eq-dirac2} is easily verified. Let $(P, \piP)$ be a Poisson manifold with a right Poisson action 
$$
\rho: (P, \piP) \times (G^*, -\piGs) \to (P, \piP)
$$
of $(G^*, -\piGs)$. For $p \in P$ and $u \in G^*$, one has the maps $l_p: G^* \to P$, $l_pu = \rho(p, u)$ and $r_u: P \to P$, $r_up = \rho(p, u)$.

\begin{pro} \label{pro-gp/gpoid-action}
Let $\varphi: (Y, \piY) \to (P, \piP)$ be an immersive Poisson map such that 
$$
\varphi(y \lhd \gamma) = \rho(\varphi(y), \: \theta_\sGs(\gamma)), \hs (y, \gamma) \in Y \ast \Gamma_\sG. 
$$
Then \eqref{eq-dirac2} holds. 
\end{pro}
\begin{proof}
For $y \in Y$ and $\gamma \in \Gamma_\sG$, it is clear that $\varphi L_y = l_{\varphi(y)} \theta_\sGs$ and $\varphi R_{\sS_\gamma} = r_u \varphi$, where $u = \theta_\sGs(\gamma)$. Since $\rho$ is a Poisson action, one has for $(y, \gamma) \in Y \ast \Gamma_\sG$, 
\begin{align*}
\varphi \left( L_{y} \left(R_{\sS_\gamma} \pi_\sGam(\varepsilon_\sG(g)) - \pi_\sGam(\gamma)\right) + R_{\sS_\gamma}\piY(y) \right) & = - l_{\varphi(y)} \piGs(u) + r_u \piP(\varphi(y))  \\
   & = \piP(\rho(\varphi(y), u)) = \piP(\varphi(y \lhd \gamma))   \\
   & = \varphi(\pi(y \lhd \gamma)),
\end{align*}
and since $\varphi$ is an immersion, this establishes \eqref{eq-dirac2}. 
\end{proof}

\begin{rem}
1) One has similar criteria as in Propositions \ref{pro-liu.wei.xu-adapt} and \ref{pro-gp/gpoid-action}  for a right Lie groupoid action of $(\Gamma_\sGs , \pi_\sGam)$ to be Poisson. In particular, if $\lhd$ is a right Poisson action of $(\Gamma_\sGs \rightrightarrows G^*, \pi_\sGam)$ on a Poisson map $\mu: (Z, \piZ) \to (G^*, -\piGs)$, one has 
\begin{equation} \label{eq-dirac1Z}
\vartheta_\sZ(x)(z) = \frac{d}{dt} \!\mid_{t = 0} (z \lhd \gamma_{\exp(tx), \: \mu(z)}), \hs z \in Z, \; x \in \g, 
\end{equation}
where the dressing action
\begin{equation} \label{lambdaZ}
\vartheta_\sZ : \g \to \XX^1(Z),  \hs \vartheta_\sZ(x) =  \piZ^\sharp(\mu^*x^R), \hs x \in \g, 
\end{equation}
is a left Poisson action of $(\g, \delta_\g)$ on $(Z, \piZ)$. 

2) Applying \eqref{eq-dirac1} and \eqref{eq-dirac1Z} to the identity maps $Id_\sG: G \to G$ and $Id_\sGs: G^* \to G^*$, the dressing actions 
\begin{align}
\varrho_\sG : \g^* \to \XX^1(G), & \hs \hs  \varrho_\sG(\xi) =  \piG^\sharp(\xi^L), \hs \xi \in \g^*, \label{eq-varrhoG}  \\
\vartheta_\sGs : \g \to \XX^1(G^*), & \hs \hs \vartheta_\sGs(x) = - \piGs^\sharp(x^R), \hs x \in \g, \label{eq-lambdaG*}
\end{align}
satisfy 
$$
\varrho_\sG(\xi)(g) = \frac{d}{dt}\! \mid_{t=0} g^{\exp(t\xi)}  \hs \text{and} \hs \vartheta_\sGs(x)(u) = \frac{d}{dt}\! \mid_{t=0} \exp(tx)[u],
$$
a fact which can also be proven directly using \eqref{eq-Ad}. 
\hfill $\diamond$
\end{rem}

\section{A Lagrangian bisection associated to a pair of dual Poisson Lie groups} \label{sec-lag-PLgps}

In \cite{wein-xu}, Weinstein and Xu construct a Lagrangian submanifold in the cartesian square of  the symplectic groupoid of a quasitriangular Poisson Lie group, which they interpret as a classical analogue of a solution to the quantum Yang-Baxter equation. We show in this subsection that when the quasitriangular Poisson Lie group is taken to be the Drinfeld double $(D, \piD)$ of a pair of dual Poisson Lie groups $((G, \piG), (G^*, \piGs))$, this Lagrangian submanifold is essentially the cartesian product of the diagonal in $\Gamma^2$ with the identity bisections of $\Gamma_\sG$ and $\Gamma_\sGs$.

\subsection{The global $\R$-matrix of a Drinfeld double} \label{subsec-wein-xu}

Let $((G, \piG), (G^*, \piGs))$ be a pair of dual Poisson Lie groups as in $\S$\ref{subsec-dble-gpoid}. The theory in \cite{wein-xu} is developed under the simplifying assumption that $(G, \piG)$ is complete. Although completeness of $(G, \piG)$ is not needed in this paper (and indeed our main application is with non-complete Poisson Lie groups), we will assume in this subsection $\S$\ref{subsec-wein-xu} that $(G, \piG)$ is complete, as to relate more easily our presentation to that of Weinstein and Xu.

Thus recall from Remark \ref{rem-complete} that one has a natural identification $(\Gamma, \pi_\sGam) \cong (D, \pi^+_\sD)$, and let $D_\sG$ and $D_\sGs$ denote respectively $D$ with its structure of groupoid over $G$ and $G^*$. Let $D_{diag} \subset D^2$ be the diagonal subgroup and $D' = G^* \times G \subset D^2$, and recall the $r$-matrix $\Lambda_{\g, \g^*}^{(2)} \in \wedge^2 (\d \oplus \d)$ defined in \eqref{eq-r^(2)}. It is clear that the multiplication in $D^2$ restricts to a diffeomorphism $D_{diag} \times D' \cong D^2$, hence $(D, \piD)$ is complete, and by Remark \ref{rem-complete}, $(D^2, \: \pi^+_{\sD^2})$ has the structure of a symplectic groupoid over $(D \cong D_{\diag}, \piD)$, where 
\begin{align*}
\pi^+_{\sD^2} & = (\Lambda_{\g, \g^*}^{(2)})^L + (\Lambda_{\g, \g^*}^{(2)})^R   \\
   & = (\Lambda_{\g, \g^*}^L + \Lambda_{\g, \g^*}^R, \: \Lambda_{\g, \g^*}^L + \Lambda_{\g, \g^*}^R) - (\Lambda^L + \Lambda^R)   \\
   & = (\pi^+_\sD, \: \pi^+_\sD) - (\Lambda^L + \Lambda^R), 
\end{align*}
and where $\Lambda$ is defined in \eqref{eq-t}. Let $D_\sGs^{op}$ denote $D_\sGs$ with the opposite groupoid structure, so that $(D_\sGs^{op}, \pi^+_\sD)$ is a symplectic groupoid over $(G^*, -\piGs)$. We lift in the following Proposition \ref{lem-Psi} the Poisson isomorphism $(G, \piG) \times (G^*, -\piGs) \cong (D, \piD)$ to an isomorphism of symplectic groupoids $(D_\sG, \pi^+_\sD) \times (D_\sGs^{op}, \pi^+_\sD) \cong (D^2, \pi^+_{\sD^2})$.

\begin{pro} \label{lem-Psi}
The map $\Psi: (D_\sG, \pi^+_\sD) \times (D_\sGs^{op}, \pi^+_\sD) \to (D^2, \:  \pi^+_{\sD^2})$ given by 
$$
\Psi(g_1u_1, \: g_2u_2) = (g_1u_1u_2, \: g_1g_2u_2), \hs g_i \in G, \; u_i \in G^*,
$$
is an isomorphism of symplectic groupoids.
\end{pro}
\begin{proof}
We first prove that $\Psi$ is a Poisson map. Let $(x_i)$ be a basis of $\g$ and $(\xi^i)$ the dual basis of $\g^*$. By definition of $\piG$, $\piGs$, and the dressing actions in \eqref{eq-varrhoG} - \eqref{eq-lambdaG*}, one has 
\begin{align*}
\piG(g) & = l_g x_i \otimes \varrho_\sG(\xi^i)(g), \hs g \in G,    \\
\piGs(u) & = - r_u \xi^i \otimes \vartheta_\sGs(x_i)(u), \hs u \in G^*. 
\end{align*}
Hence using \eqref{eq-comp-pi+} and that the multiplication in $D$ gives an isomorphism $G \times G^* \cong D$, one has for $g_i \in G$, $u_i \in G^*$, 
\begin{align*}
\Psi(\pi^+_\sD(g_1u_1),  0) & = (l_{g_1}r_{u_1u_2}x_i, \: l_{g_1}r_{g_2u_2}x_i) \otimes (r_{u_1u_2}\varrho_\sG(\xi^i)(g_1), \: r_{g_2u_2}\varrho_\sG(\xi^i)(g_1))  \\
  & \hs \: + (l_{g_1}r_{u_2} \piGs(u_1), 0) - (l_{g_1}r_{u_1u_2} x_i, \: l_{g_1}r_{g_2u_2}x_i) \wedge (l_{g_1}r_{u_1u_2} \xi^i, 0)    \\
  & = (r_{u_1u_2}\piG(g_1), 0) + (0,  r_{g_2u_2}\piG(g_1)) + (l_{g_1}r_{u_1u_2}x_i, 0) \wedge (0, r_{g_2u_2}\varrho_\sG(\xi^i)(g_1))   \\
   & \hs \: + (l_{g_1}r_{u_2} \piGs(u_1), 0) - (l_{g_1}r_{u_1u_2} x_i, \: l_{g_1}r_{g_2u_2}x_i) \wedge (l_{g_1}r_{u_1u_2} \xi^i, 0)    \\
   & = (r_{u_2}\pi^+_\sD(g_1u_1), 0 ) +  (0,  r_{g_2u_2}\piG(g_1))   \\
   & \hs \: + (l_{g_1}r_{u_1u_2}x_i, 0) \wedge (0, r_{g_2u_2}\varrho_\sG(\xi^i)(g_1)) - (0, l_{g_1}r_{g_2u_2}x_i) \wedge (l_{g_1}r_{u_1u_2} \xi^i, 0)   \\
   & = (r_{u_2}\pi^+_\sD(g_1u_1), 0) +  (0,r_{g_2u_2}\piG(g_1)) +  (l_{g_1}r_{u_1u_2}x_i, 0) \wedge (0, r_{g_2u_2}\varrho_\sG(\xi^i)(g_1))   \\
   &  \hs \: + (r_{g_1u_1u_2}\Ad_{g_1}\xi^i, 0) \wedge (0, r_{g_1g_2u_2}\Ad_{g_1}x_i) \\
   & = (r_{u_2}\pi^+_\sD(g_1u_1), 0) +  (0, r_{g_2u_2}\piG(g_1)) + (r_{g_1u_1u_2}, \: r_{g_1g_2u_2})(\Lambda)   \\
   & = (r_{u_2}\pi^+_\sD(g_1u_1), 0) +  (0, r_{g_2u_2}\piG(g_1)) - r_{(g_1u_1u_2, \: g_1g_2u_2)}\Lambda 
\end{align*}
where the second to last line is obtained using \eqref{eq-Ad}. Similarly, one has 
$$
\Psi(0, \pi^+_\sD(g_2u_2)) = (l_{g_1u_1}\piGs(u_2), \:0) + (0, l_{g_1}\pi^+_\sD(g_2u_2)) - l_{(g_1u_1u_2, \: g_1g_2u_2)}\Lambda. 
$$
Thus again using \eqref{eq-comp-pi+} and the multiplicativity of $\piG$ and $\piGs$, one has 
\begin{align*}
\Psi(\pi^+_\sD(g_1u_1), \: \pi^+_\sD(g_2u_2)) & = \left( \pi^+_\sD(g_1u_1u_2), \: \pi^+_\sD(g_1g_2u_2) \right) -  r_{(g_1u_1u_2, \: g_1g_2u_2)}\Lambda  - l_{(g_1u_1u_2, \: g_1g_2u_2)}\Lambda   \\
   & = \pi^+_{\sD^2}(g_1u_1u_2, \: g_1g_2u_2),
\end{align*}
hence $\Psi$ is Poisson. Now, for $g_1, g_2 \in G$ and $u_1, u_2 \in G^*$, the relations
$$
(g_1u_1, \: g_2u_2) = (gu, gu)(v,k) = (v', k')(g'u', g'u'), 
$$
with 
\begin{align*}
& g = g_1 \in G, \hs \;\;\:  u = (g_1^{-1}g_2)[u_2] \in G^*, \; v = k[u_2^{-1}] \in G^*, \hs\; k = (g
_1^{-1}g_2)^{u_2} \in G,    \\
& g' = g_1^{u_1u_2^{-1}} \in G, \; u' = u_2 \in G^*, \hs\hs\hs\; v' = g_1[u_1u_2^{-1}] \in G^*, \; k' = g_2 (g_1^{-1})^v \in G,
\end{align*}
completely determine the structure of $D^2$, as a groupoid over $D$. In particular, letting $\theta_\sD$, $\tau_\sD$ be the source and target map of $D^2 \rightrightarrows D$, one has 
\begin{align*}
\theta_\sD(\Psi(g_1u_1, \: g_2u_2)) & =  \theta_\sD(g_1u_1u_2, \: g_1g_2u_2) = g_1 g_2[u_2] = \theta_\sG(g_1u_1) \theta_\sGs^{op}(g_2u_2),  \\
\tau_\sD(\Psi(g_1u_1, \: g_2u_2)) & = \tau_\sD(g_1u_1u_2, \: g_1g_2u_2) = g_1^{u_1}u_2 = \tau_\sG(g_1u_1) \tau_\sGs^{op}(g_2u_2),
\end{align*}
where $\theta_\sGs^{op} = \tau_\sGs$, $\tau_\sGs^{op} = \theta_\sGs$ are the source and target maps of $D_\sGs^{op} \rightrightarrows G^*$. Showing that $\Psi$ commutes with the other groupoid structure maps is a straightforward verification, which is left to the reader. 
\end{proof}

Recall that $\iota_\sGs$ denotes the inverse in the groupoid $D_\sGs \rightrightarrows G^*$, so that 
$$
(Id_\sD \times \iota_\sGs) \Psi^{-1}: (D^2,  \pi^+_{\sD^2}) \to (D_\sG, \pi^+_\sD) \times  (D_\sGs, -\pi^+_\sD)
$$ 
is an isomorphism of symplectic groupoids. As $- \xi^i \otimes x_i \subset \d \otimes \d$ is a quasitriangular $r$-matrix for $(\d, \delta_\d)$, where $(x_i)$ is a basis of $\g$ and $(\xi^i)$ the dual basis of $\g^*$, by a straightforward application of \cite[Definition 4.3]{wein-xu} to the quasitriangular Poisson Lie group $(D, \piD)$, 
$$
\R = \{ \left( (uv, \: uk), \: (u[k]^{-1}u, \: u[k]^{-1}g) \right) : \; g, k \in G, \; u, v \in G^* \} \subset D^2 \times D^2
$$
is the global $\R$-matrix of $(D^2, \pi^+_{\sD^2})$. The following Lemma \ref{lem-R} is straightforward.

\begin{lem} \label{lem-R}
One has 
$$
(Id_\sD \times \iota_\sGs \times Id_\sD \times \iota_\sGs)(\Psi^{-1} \times \Psi^{-1})(\R) = G^* \times D_{\rm diag} \times G. 
$$
\end{lem}


\subsection{A Lagrangian bisection associated to a pair of dual Poisson Lie groups}

We no longer assume from now on that $(G, \piG)$ is complete. In the spirit of Lemma \ref{lem-R},  the Lagrangian submanifold 
$$
\Gamma_{\diag} \subset \left( \Gamma_\sGs \times \Gamma_\sG, \: (-\pi_\sGam) \times \pi_\sGam \right),
$$
should be thought of as a ``reduced global $\R$-matrix" associated to the pair of dual Poisson Lie groups $((G, \piG), (G^*, \piGs))$. In particular, 
\begin{equation}  \label{eq-Lag-bis}
\L := (O_\sGam)_{\diag} \subset \Gamma_{\diag}
\end{equation}
is an open subset in $\Gamma_{\diag}$, hence a Lagrangian submanifold of $\left( \Gamma_\sGs \times \Gamma_\sG, \: (-\pi_\sGam) \times \pi_\sGam \right)$, and a local bisection of $\Gamma_\sGs \times \Gamma_\sG$, thus induces the local Poisson isomorphism $R_\ssL$ of $(G^*, -\piGs) \times (G, \piG)$,  
$$
R_\ssL: \O^{21}_{\sGs, \sG} \to O_{\sGs, \sG}, \hs R_\ssL(u, g) = (g[u], g^u), \hs (u, g) \in O^{21}_{\sGs, \sG},
$$
with inverse 
$$
R_{\ssL^{-1}}: O_{\sGs, \sG} \to O^{21}_{\sGs, \sG}, \hs R_{\ssL^{-1}}(u, g) =  (u^g, u[g]), \hs (u, g) \in O_{\sGs, \sG}, 
$$
where 
$$
O^{21}_{\sGs, \sG} = \{ (u, g) : (g, u) \in O_{\sG, \sGs} \}. 
$$
The local bisection $\L$, or rather $\L^{-1}$, has the role of ``twisting" the direct product Lie group multiplication in $G \times G^*$ into the multiplication in $D$. Indeed, let 
$$
m_{\scriptscriptstyle G, G^*}: G \times G^* \times G \times G^* \to G \times G^*
$$
be the direct product multiplication, that is $m_{\scriptscriptstyle G, G^*}(g_1, u_1, g_2, u_2) = (g_1g_2, u_1u_2)$, $g_i \in G$, $u_i \in G^*$, and let $m_\sD: D \times D \to D$ be the multiplication in $D$. Then one has 
\begin{equation} \label{lem-twist-mult}
m_\sD(g_1u_1, g_2u_2) = m_\sD(m_{\scriptscriptstyle G, G^*}(g_1, u_1^{g_2}, u_1[g_2], u_2)), \hs g_1 \in G, \; u_2 \in G^*, \; (u_1, g_2) \in O_{\sGs, \sG}. 
\end{equation}

We show in $\S$\ref{sec-local-p-gpoid} below, how $\L$ can be used to construct (local) Poisson groupoids over mixed product Poisson structures.

\section{Local Poisson groupoids over mixed product Poisson structures} \label{sec-local-p-gpoid}

The central result of $\S$\ref{sec-local-p-gpoid} is Theorem \ref{main-thm-gpoid}, which is a construction of a local Poisson groupoid over a mixed product Poisson structure associated to the actions of a pair of dual Lie bialgebras. We fix a pair of dual Poisson Lie groups $((G, \piG), (G^*, \piGs))$ as in $\S$\ref{subsec-dble-gpoid}. 

\subsection{Twisted multiplicative groupoid actions} \label{subsec-twist-mult}

Let $(\Y, \pi_\ssY)$, $(\Z, \pi_\ssZ)$ be local Poisson groupoids over Poisson manifolds $(Y, \piY)$ and $(Z, \piZ)$, and let 
$$
\mu_\ssY: (\Y, \pi_\ssY) \to (G, \piG), \hs \mu_\ssZ: (\Z, \pi_\ssZ) \to (G^*, -\piGs),
$$
be morphisms of local Poisson groupoids, inducing the dressing actions $\varrho_\ssY:  \g^* \to \XX^1(\Y)$ and $\vartheta_\ssZ: \g \to \XX^1(\Z)$ defined in \eqref{varrhoY} and \eqref{lambdaZ}. Assume given right Poisson actions $\lhd_\sG$, $\lhd_\sGs$ of $(\Gamma_\sG, -\pi_\sGam)$ on $\mu_\ssY$ and of $(\Gamma_\sGs, \pi_\sGam)$ on $\mu_\ssZ$ respectively, which satisfy the {\it twisted multiplicativity} properties
\begin{align}
\tilde{y}_1 \tilde{y}_2 \lhd_\sG \gamma_2 \star \gamma_1 & = (\tilde{y}_1 \lhd_\sG \gamma_1) (\tilde{y}_2 \lhd_\sG \gamma_2), \hs (\tilde{y}_1, \tilde{y}_2) \in \Y^{(2)}_0, \; (\gamma_2, \gamma_1) \in \Gamma_\sGs^{(2)}, \label{eq-Y-twist} \\
\tilde{z}_1\tilde{z}_2 \lhd_\sGs \gamma_1 \gamma_2 & = (\tilde{z}_1 \lhd_\sGs \gamma_1) (\tilde{z}_2 \lhd_\sGs \gamma_2), \hs (\tilde{z}_1, \tilde{z}_2) \in \Z^{(2)}_0, \; (\gamma_1, \gamma_2) \in \Gamma_\sG^{(2)}, \label{eq-Z-twist}
\end{align}
whenever the left and right hand side of \eqref{eq-Y-twist} - \eqref{eq-Z-twist} are defined. 

\begin{rem}
Assume that $(\Z, \pi_\ssZ)$ is a Poisson groupoid. Using \cite[Theorem 2.4] {Xu:Poioid}, it is easy to see that $\vartheta_\ssZ$ satisfies
\begin{equation} \label{eq-m_ssZ}
\vartheta_\ssZ(x)(\tilde{z}_1\tilde{z}_2) = dm_\ssZ(\vartheta_\ssZ(x)(\tilde{z}_1), \: \vartheta_\ssZ(\Ad^*_{\mu_\ssZ(\tilde{z}_1)}x)(\tilde{z}_2)), \hs (\tilde{z}_1, \tilde{z}_2) \in \Z^{(2)}, \; x \in \g, 
\end{equation}
where $m_\ssZ: \Z^{(2)} \to \Z$ is the groupoid multiplication map. If furthermore $(G, \piG)$ is complete and $\vartheta_\ssZ$ integrates to a group action $(g, \tilde{z}) \mapsto g[\tilde{z}]$ of $G$ on $\Z$, \eqref{eq-m_ssZ} is equivalent to  
\begin{equation} \label{eq-fernan-ponte}
g[\tilde{z}_1\tilde{z}_2] = g[\tilde{z}_1]g^{\mu_\ssZ(\tilde{z}_1)}[\tilde{z}_2], \hs (\tilde{z}_1, \tilde{z}_2) \in \Z^{(2)}, \; g \in G,
\end{equation}
which is nothing but \eqref{eq-Z-twist}, rewritten using the fact that $\Gamma_\sGs$ is an action groupoid, see Remark \ref{rem-complete}. Condition \eqref{eq-fernan-ponte} first appeared in  \cite{rui-poissonactions} (see also \cite{Lu:Duke-Poi}), from which we have borrowed the expression ``twisted multiplicative". Fernandes and Ponte prove that if $(\Z, \pi_\ssZ)$ is a source simply connected symplectic groupoid and if $(G, \piG)$ is complete, $\vartheta_\ssZ$ integrates to a group action satisfying \eqref{eq-fernan-ponte}.
\hfill $\diamond$
\end{rem}

We write the groupoid structure maps of $\Y \rightrightarrows Y$ and $\Z \rightrightarrows Z$ using subscripts, e.g $\theta_\ssY$, $\theta_\ssZ$ are the respective source maps, etc. By letting the identity bisection of $\Gamma_\sG$ and $\Gamma_\sGs$ act on the identity bisection of $(\Y, \pi_\ssY)$ and $(\Z, \pi_\ssZ)$ respectively, one obtains actions
$$
\varrho_\sY: Y \times G^* \to Y  \hs \text{and} \hs \vartheta_\sZ: G \times Z \to Z, 
$$
which we denote by $\varrho_\sY(y, u) = y^u$ and $\vartheta_\sZ(g, z) = g[z]$, satisfying 
\begin{align}
&\theta_{\ssY}(\tilde{y} \lhd_\sG \gamma) = \theta_{\ssY}(\tilde{y})^{\tau_{\sGs}(\gamma)},  \hs \;\;\; \tau_{\ssY}(\tilde{y} \lhd_\sG \gamma) = \tau_{\ssY}(\tilde{y})^{\theta_{\sGs}(\gamma)}, \hs \hs (\tilde{y}, \gamma) \in \Y \ast \Gamma_\sG, \label{eq-twist-sourcetarget} \\
&\theta_{\ssZ}(\tilde{z} \lhd_\sGs \gamma) = \theta_\sG(\gamma)[\theta_{\ssZ}(\tilde{z})],  \hs \tau_{\ssZ}(\tilde{z} \lhd_\sGs \gamma) = \tau_\sG(\gamma)[\tau_{\ssZ}(\tilde{z})], \hs \hs (\tilde{z}, \gamma) \in \Z \ast \Gamma_\sGs. \notag
\end{align}

\begin{lem} \label{lem-varrho-Ad}
Let $\tilde{y} \in \Y$, $\tilde{z} \in \Z$, $x \in \g$ and $\xi \in \g^*$. Then one has 
$$
\theta_\ssY(\varrho_\ssY(\xi)(\tilde{y})) = \varrho_\sY(\Ad^*_{\mu_\ssY(\tilde{y})^{-1}} \xi)(\theta_\ssY(\tilde{y})) \hs \text{and} \hs \tau_\ssZ(\vartheta_\ssZ(x)(\tilde{z})) = \vartheta_\sZ(\Ad^*_{\mu_\ssZ(\tilde{z})}x)(\tau_\ssZ(\tilde{y})). 
$$
\end{lem}
\begin{proof}
Let $u \in G^*$. The first relation is obtained by differentiating 
$$
\theta_{\ssY}(\tilde{y} \lhd_\sG \gamma_{\theta_\ssY(\tilde{y}), \: u}) = \theta_{\ssY}(\tilde{y})^{\mu_\ssY(\tilde{y})[u]}
$$
with respect to $u$ and using \eqref{lem-Ad} and \eqref{eq-dirac1}. The second relation is similarly proved. 
\end{proof}

\begin{lem} \label{lem-rhoY-lambdaZ}
The actions $\varrho_\sY: (Y, \piY) \times (G^*, \piGs) \to (Y, \piY)$ and $\vartheta_\sZ: (G, \piG) \times (Z, \piZ) \to (Z, \piZ)$ are Poisson actions. 
\end{lem}
\begin{proof}
Indeed, the graph of $\varrho_\sY$ in $Y \times G^* \times Y$ is the image of the graph of $\lhd_\sG$ in $\Y \times \Gamma_\sG \times \Y$ under the anti-Poisson map 
$$
\tau_\ssY \times \theta_\sGs \times \tau_\ssY: (\Y \times \Gamma_\sG \times \Y, \: \pi_\ssY \times (-\pi_\sGam) \times (-\pi_\ssY)) \to (Y \times G^* \times Y, \: \piY \times \piGs \times (-\piY)).
$$
Thus the graph of $\varrho_\sY$ is coisotropic in $(Y \times G^* \times Y, \: \piY \times \piGs \times (-\piY))$, which means that $\varrho_\sY$ is Poisson. A similar argument shows that $\vartheta_\sZ$ is Poisson. 
\end{proof}

Hence one can equip $Y \times Z$ with the mixed product Poisson structure $\piY \times_{(\varrho_\sY, \vartheta_\sZ)} \piZ$.

\subsection{The local Poisson isomorphism $R_\ssL$}

As $(\Gamma_\sGs \times \Gamma_\sG, \: \pi_\sGam \times (-\pi_\sGam))$ acts on $\mu_\ssZ \times \mu_\ssY$, the local Lagrangian bisection $\L$ induces the local Poisson isomorphism 
$$
R_\ssL:  (O^{21}_{\ssZ, \ssY}, \: \pi_\ssZ \times \pi_\ssY) \to (O_{\ssZ, \ssY}, \: \pi_\ssZ \times \pi_\ssY),
$$
where 
$$
O^{21}_{\ssZ, \ssY} = (\mu_\ssZ \times \mu_\ssY)^{-1}(O^{21}_{\sGs, \sG}) \subset \Z \times \Y \hs \text{and} \hs O_{\ssZ, \ssY} = (\mu_\ssZ\times \mu_\ssY)^{-1}(O_{\sGs, \sG}) \subset \Z \times \Y,
$$
given by 
$$
R_\ssL(\tilde{z}, \tilde{y}) = (\tilde{z} \lhd_\sGs \gamma_{\mu_\ssY(\tilde{y}), \: \mu_\ssZ(\tilde{z})}, \;\; \tilde{y} \lhd_\sG \gamma_{\mu_\ssY(\tilde{y}), \: \mu_\ssZ(\tilde{z})}), \hs (\tilde{z}, \tilde{y}) \in O^{21}_{\ssZ, \ssY},
$$
recall \eqref{eq-gamma_gu}.  One has the left action 
$$
\gamma \rhd_\sG \tilde{y} := \tilde{y} \lhd_\sG \iota_\sG(\gamma), \hs (\tilde{z}, \iota_\sG(\gamma)) \in \Y \ast \Gamma_\sG,
$$
of $(\Gamma_\sG, \pi_\sGam)$ on $\mu_\ssY$, and similarly, the left action  
$$
\gamma \rhd_\sGs \tilde{z} := \tilde{z} \lhd_\sGs \iota_\sGs(\gamma), \hs  (\tilde{z}, \iota_\sGs(\gamma)) \in \Z \ast \Gamma_\sGs
$$
of $(\Gamma_\sGs, -\pi_\sGam)$ on $\mu_\ssY$. Then the inverse of $R_\ssL$, 
$$
R_{\ssL^{-1}}: (O_{\ssZ, \ssY}, \: \pi_\ssZ \times \pi_\ssY) \to  (O^{21}_{\ssZ, \ssY}, \: \pi_\ssZ \times \pi_\ssY)
$$
is given by 
$$
R_{\ssL^{-1}}(\tilde{z}, \tilde{y}) = (\gamma^{\mu_\ssZ(\tilde{z}), \: \mu_\ssY(\tilde{y})} \rhd_\sGs \tilde{z}, \;\; \gamma^{\mu_\ssZ(\tilde{z}), \:\mu_\ssY(\tilde{y})} \rhd_\sG  \tilde{y}), \hs (\tilde{z}, \tilde{y}) \in O_{\ssZ, \ssY}. 
$$

\subsection{Local Poisson groupoids over mixed product Poisson structures}

Let 
$$
\Y \times_\ssL \Z
$$
denote $\Y \times \Z$ equipped with the following local groupoid structure maps: 
\begin{align}
&\mbox{source map}: \; \theta(\tilde{y}, \: \tilde{z})= (\theta_\ssY(\tilde{y}), \; \mu_\ssY(\tilde{y})[\theta_\ssZ(\tilde{z})]), \hs (\tilde{y}, \: \tilde{z}) \in \Y \times_{(\mu_\ssY, \mu_\ssZ)} \Z,  \label{eq-new-gpoid}  \\
&\mbox{target map}: \; \tau(\tilde{y}, \tilde{z}) = (\tau_\ssY(\tilde{y})^{\mu_\ssZ(\tilde{z})}, \; \tau_\ssZ(\tilde{z})),  \hs (\tilde{y}, \: \tilde{z}) \in \Y \times_{(\mu_\ssY, \mu_\ssZ)} \Z, \notag  \\
&\mbox{identity bisection}: \; \varepsilon(y, z) = (\varepsilon_\ssY(y), \; \varepsilon_\ssZ(z)), \hs (y, z) \in Y \times Z, \notag  \\
&\mbox{inverse map}: \; \mbox{when} \; (\tilde{z}, \tilde{y}) \in R_\ssL^{-1}\left(O_{\ssZ, \ssY} \cap (\Z^{(-1)} \times \Y^{(-1)}) \right),  \notag   \\
& \iota(\tilde{y}, \tilde{z}) = \left( \iota_\ssY(\tilde{y} \lhd_\sG \gamma_{\mu_\ssY(\tilde{y}), \: \mu_\ssZ(\tilde{z})}), \; \iota_\ssZ(\tilde{z} \lhd_\sGs \gamma_{\mu_\ssY(\tilde{y}), \: \mu_\ssZ(\tilde{z})} \right), \notag  \\
&\mbox{multiplication}: \; \mbox{when} \;  \tau(\tilde{y}_1, \tilde{z}_1) = \theta(\tilde{y}_2, \: \tilde{z}_2) \; \mbox{and} \; (\tilde{z}_1, \tilde{y}_2) \in O_{\ssZ, \ssY}, \notag \\ 
& (\tilde{y}_1, \tilde{z}_1) \cdot (\tilde{y}_2, \tilde{z}_2) = \left( \tilde{y}_1\left(\gamma^{\mu_\ssZ(\tilde{z}_1), \: \mu_\ssY(\tilde{y}_2)} \rhd_\sG \tilde{y}_2 \right), \;\; \left( \gamma^{\mu_\ssZ(\tilde{z}_1), \:\mu_\ssY(\tilde{y}_2)} \rhd_\sGs \tilde{z}_1 \right) \tilde{z}_2 \right). \notag
\end{align}

\begin{thm} \label{main-thm-gpoid}
The maps in \eqref{eq-new-gpoid} determine a well-defined local groupoid structure on $\Y \times \Z$ and $(\Y \times_\ssL \Z, \: \pi_\ssY \times \pi_\ssZ)$ is a local Poisson groupoid over $(Y \times Z, \: \piY \times_{(\varrho_\sY, \vartheta_\sZ)} \piZ)$, such that the map 
$$
\mu: (\Y \times_\ssL \Z, \: \pi_\ssY \times \pi_\ssZ) \to (D, \piD), \hs \mu(\tilde{y}, \tilde{z}) = \mu_\ssY(\tilde{y}) \mu_\ssZ(\tilde{z}), \hs  (\tilde{y}, \: \tilde{z}) \in \Y \times_\ssL \Z,
$$
is a morphism of local Poisson groupoids. 
\end{thm}
\begin{proof}
Checking that \eqref{eq-new-gpoid} sastisfies the axioms of a local groupoid is lengthy but straightforward. For example, let $(\tilde{y}_1, \tilde{z}_1), (\tilde{y}_2, \tilde{z}_2)$ be composable elements of $\Y \times_\ssL \Z$ and write $g_i = \mu_\ssY(\tilde{y}_i)$, $u_i = \mu_\ssZ(\tilde{z}_i)$. Using \eqref{eq-twist-sourcetarget}, one has 
\begin{align*}
\theta \left( (\tilde{y}_1, \tilde{z}_1) \cdot (\tilde{y}_2, \tilde{z}_2) \right) & = \left( \theta_\ssY(\tilde{y_1} (\gamma^{u_1, \: g_2} \rhd_\sG \tilde{y}_2)), \;\; ( g_1u_1[g_2]) [\theta_\ssZ((\gamma^{u_1, \: g_2} \rhd_\sGs \tilde{z}_1) \tilde{z}_2)]\right)    \\
   & = \left( \theta_\ssY(\tilde{y}_1), \;\; (g_1u_1[g_2])[\theta_\ssZ(\gamma^{u_1, \: g_2} \rhd_\sGs \tilde{z}_1)] \right)    \\
   & = \left( \theta_\ssY(\tilde{y}_1), \;\; (g_1u_1[g_2]\theta_\sG( \tilde{\gamma}_{u_1^{g_2}, \: g_2^{-1}}))[\theta_\ssZ(\tilde{z}_1)]  \right)   \\
   & =  \left( \theta_\ssY(\tilde{y}_1), \;\; g_1[\theta_\ssZ(\tilde{z}_1)]  \right) = \theta(\tilde{y}_1, \: \tilde{z}_1),
\end{align*}
and a similar calculation gives $\tau \left( (\tilde{y}_1, \tilde{z}_1) \cdot (\tilde{y}_2, \tilde{z}_2) \right) = \tau(\tilde{y}_2, \tilde{z}_2)$. Let now $(\tilde{y}_3, \tilde{z}_3)$ be a third element such that both $\left( (\tilde{y}_1, \tilde{z}_1) \cdot (\tilde{y}_2, \tilde{z}_2) \right) \cdot (\tilde{y}_3, \tilde{z}_3)$ and $(\tilde{y}_1, \tilde{z}_1) \cdot \left(  (\tilde{y}_2, \tilde{z}_2) \cdot (\tilde{y}_3, \tilde{z}_3) \right)$ are defined. Then writing $g_3 = \mu_\ssY(\tilde{y}_3)$, $u_3 = \mu_\ssZ(\tilde{z}_3)$, and using \eqref{eq-iota-morphism}, \eqref{eq-gam-mult}, \eqref{eq-Y-twist}, and \eqref{eq-Z-twist}, one has  
\begin{align*}
\left( (\tilde{y}_1, \tilde{z}_1) \cdot (\tilde{y}_2, \tilde{z}_2) \right) \cdot (\tilde{y}_3, \tilde{z}_3) & =  \left( \tilde{y}_1 (\gamma^{u_1, \:g_2} \rhd_\sG \tilde{y}_2) (\gamma^{u_1^{g_2}u_2, \: g_3} \rhd_\sG \tilde{y}_3), \right.   \\
    & \hs \left. \left(\gamma^{u_1^{g_2}u_2, \: g_3} \rhd_\sGs (\gamma^{u_1, \:g_2} \rhd_\sGs \tilde{z}_1 ) \tilde{z}_2 \right) \tilde{z}_3 \right)  \\
    & =\left( \tilde{y}_1 (\gamma^{u_1, \:g_2} \rhd_\sG \tilde{y}_2) (\gamma^{u_1^{g_2}, \: u_2[g_3]} \gamma^{u_2, \: g_3} \rhd_\sG \tilde{y}_3), \right.   \\
    &\hs \left. \left((\gamma^{u_1^{g_2}, \: u_2[g_3]} \gamma^{u_2, \: g_3}) \rhd_\sGs (\gamma^{u_1, \:g_2} \rhd_\sGs \tilde{z}_1 ) \tilde{z}_2 \right) \tilde{z}_3 \right)    \\
    & = \left( \tilde{y}_1 (\gamma^{u_1, \: g_2u_2[g_3]} \rhd_\sG \tilde{y}_2 (\gamma^{u_2, \: g_3} \rhd_\sG \tilde{y}_3)) , \right.   \\
    & \hs \left. (\gamma^{u_1, \: g_2 u_2[g_3]} \rhd_\sGs \tilde{z}_1) (\gamma^{u_2, \: g_3} \rhd_\sGs \tilde{z}_2 ) \tilde{z}_3 \right)  \\
  & = (\tilde{y}_1, \tilde{z}_1) \cdot \left(  (\tilde{y}_2, \tilde{z}_2) \cdot (\tilde{y}_3, \tilde{z}_3) \right).
\end{align*}
We leave to the reader to check the remaining axioms. 

By construction, the graph $Gr_\ssL \subset (\Y \times \Z)^3$ of the multiplication in $Y \times_\ssL \Z$ is  
$$
Gr_\ssL = (Id_\ssY \times R_\ssL \times Id_\ssZ \times Id_\ssY \times Id_\ssZ) \left(Gr \cap (\Y \times O^{21}_{\ssZ, \ssY} \times \Z \times \Y \times \Z \right),
$$
where $Gr \subset (\Y \times \Z)^3$ is the graph of the direct product groupoid multiplication. Hence $Gr_\ssL$ is a coisotropic submanifold, when $(\Y \times \Z)^3$ is equipped with the Poisson structure $\pi_\ssY \times \pi_\ssZ \times \pi_\ssY \times \pi_\ssZ \times (-\pi_\ssY) \times (-\pi_\ssZ)$, thus $(\Y \times_\ssL \Z, \: \pi_\ssY \times \pi_\ssZ)$ is a local Poisson groupoid. 

As the multiplication map in $D$ induces a Poisson map $(G, \piG) \times (G^*, -\piGs) \to (D, \piD)$, the map $\mu$ is Poisson, and it is a morphism of groupoids by \eqref{lem-twist-mult}.  We prove in the Proposition \ref{pro-theta-tau} below that $\theta(\pi_\ssY \times \pi_\ssZ) = \piY \times_{(\varrho_\sY, \vartheta_\sZ)} \piZ$, which will then conclude the proof of Theorem \ref{main-thm-gpoid}. 
\end{proof}

\begin{pro} \label{pro-theta-tau} 
One has 
$$
\theta(\pi_\ssY \times \pi_\ssZ) = \piY \times_{(\varrho_\sY, \vartheta_\sZ)} \piZ = - \tau(\pi_\ssY \times \pi_\ssZ).
$$
\end{pro}
\begin{proof}
We only prove the first equality, as the second is treated similarly. Let $p_\sY$, $p_\sZ$ be the projections from $Y \times Z$ to the first and second factor. By Lemma \ref{lem-rhoY-lambdaZ}, one has $p_\sY \theta(\pi_\ssY \times \pi_\ssZ) = \piY$ and $p_\sZ \theta(\pi_\ssY \times \pi_\ssZ) = \piZ$. Let
$$
\mu_{(\varrho_\sY, \vartheta_\sZ)} = -(\varrho_\sY(\xi^i), 0) \wedge (0, \vartheta_\sZ(x_i))
$$
be the mixed term of $\piY \times_{(\varrho_\sY, \vartheta_\sZ)} \piZ$, where $(x_i)$ is a basis of $\g$ and $(\xi^i)$ the dual basis of $\g^*$, and let $(\tilde{y}, \tilde{z}) \in \Y \times \Z$, $(y, z) = (\theta_\ssY(\tilde{y}), \: \theta_\ssZ(\tilde{z}))$, $g = \mu_\ssY(\tilde{y})$, and let $\alpha \in T^*_yY$, $\beta \in T^*_{g[z]}Z$. In order to complete the proof of Proposition \ref{pro-theta-tau}, one only needs to show that 
$$
\la \pi_\ssY \times \pi_\ssZ, \: \theta^*( p_\sY^*\alpha \wedge p_\sZ^*\beta) \ra = \la \mu_{(\varrho_\sY, \vartheta_\sZ)}, \: p_\sY^*\alpha \wedge p_\sZ^*\beta \ra. 
$$
Let $p_\ssY$, $p_\ssZ$ be the projections from $\Y \times \Z$ to the first and second factor. As  
\begin{align*}
\theta^*p^*_\sY \alpha & = p_\ssY^* \theta_\ssY^* \alpha,    \\
\theta^*p^*_\sZ \beta & = p_\ssY^* \mu_\ssY^*r_{g^{-1}}^* \vartheta_\sZ^* \beta + p_\ssZ^* \theta_\ssZ^*g^*[\beta],
\end{align*}  
using Lemma \ref{lem-varrho-Ad}, one gets 
\begin{align*}
\la \pi_\ssY \times \pi_\ssZ, \:\theta^*( p_\sY^*\alpha \wedge p_\sZ^*\beta) \ra & = \la \pi_\ssY \times \pi_\ssZ,  \: p_\ssY^* \theta_\ssY^* \alpha \wedge (p_\ssY^* \mu_\ssY^*r_{g^{-1}}^* \vartheta_\sZ^* \beta + p_\ssZ^* \theta_\ssZ^*g^*[\beta]) \ra   \\
   & = \la \pi_\ssY, \:  \theta_\ssY^* \alpha \wedge  \mu_\ssY^*r_{g^{-1}}^* \vartheta_\sZ^* \beta \ra  = -\la \varrho_\ssY(\Ad^*_g\vartheta_\sZ^* \beta), \:  \theta_\ssY^* \alpha \ra   \\
   & = -\la \theta_\ssY \left(\varrho_\ssY(\Ad^*_g\vartheta_\sZ^* \beta) \right), \: \alpha \ra   = -\la \varrho_\sY (\vartheta_\sZ^* \beta), \: \alpha \ra   \\
   & = \la \mu_{(\varrho_\sY, \vartheta_\sZ)}, \: p_\sY^*\alpha \wedge p_\sZ^*\beta \ra,
\end{align*}
which concludes the proof. 
\end{proof}

\begin{rem} \label{rem-ssc-gpoid}
When $(\Y \rightrightarrows Y, \pi_\ssY)$, $(\Z \rightrightarrows Y, \pi_\ssZ)$ are taken to be the source simply connected symplectic groupoids integrating $(Y, \piY)$ and $(Z, \piZ)$, and $\mu_\ssY$, $\mu_\ssZ$ the groupoid morphisms integrating the Lie algebroid morphisms 
\begin{align*}
\varrho_\sY^*: T^*Y \to \g, \hs \la \varrho_\sY^*(\alpha), \xi \ra & = \la \alpha, \varrho_\sY(\xi)(y) \ra, \hs \xi \in \g^*, \; \alpha \in T_y^* Y, \; y \in Y,   \\
\vartheta_\sZ^*: T^*Z \to \g^*, \hs \la \vartheta_\sZ^*(\beta), x \ra & = \la \beta, \vartheta_\sZ(x)(z) \ra, \hs x \in \g, \; \beta \in T_z^* Z, \; z \in Z, 
\end{align*}
(see \cite[Proposition 6.1]{Xu:Poioid}), Theorem \ref{main-thm-gpoid} constructs a local symplectic groupoid over $(Y \times Z, \: \piY \times_{(\varrho_\sY, \vartheta_\sZ)} \piZ)$. 
\hfill $\diamond$
\end{rem}

\begin{exa}
Identify $T^*\IC$ with $\IC^2$ and let $(p, q) \mapsto q$ be the standard projection, and equip $T^*\IC$ with its canonical Poisson structure $\pi = \del_p \wedge \del_q$. Let $\mu: T^*\IC \to \IC^*$ be the map $\mu(p, q) = e^{pq}$, $(p,q) \in T^*\IC$. Then $\mu: (T^*\IC, \pi) \to (\IC^*, \piG = 0)$ is a morphism of Poisson groupoids, where $(\IC^*, 0)$ is a complete Poisson Lie group. As $\Gamma_{\scriptscriptstyle \IC^*} \cong \IC^* \times \IC^*$ is the action groupoid associated to the trivial action of $\IC^*$ on itself,   
$$
(p, q) \lhd (e^{pq}, z) = (z^{-1}p, zq), \hs z \in \IC^*, \; (p,q) \in T^*\IC. 
$$
defines Lie groupoid action of $\Gamma_{\scriptscriptstyle \IC^*}$ on $\mu$. As $((\IC^*, 0), (\IC^*, 0))$ is a pair of dual Poisson Lie groups, applying Theorem \ref{main-thm-gpoid} to $\mu_\ssY = \mu$ and $\mu_\ssZ = \mu$, 
$$
\left( (T^*\IC)^2 \cong T^*(\IC^2), \: \del_{p_1} \wedge \del_{q_1} + \del_{p_2} \wedge \del_{q_2} \right)
$$
becomes a symplectic groupoid over $(\IC^2, -q_1q_2\del_{q_1} \wedge \del_{q_2})$, with groupoid structure given by 
\begin{align*}
&\mbox{source map}: \; \theta(p_1, p_2, q_1, q_2) = (q_1, \: e^{p_1q_1}q_2)   \\
&\mbox{target map}: \; \tau(p_1, p_2, q_1, q_2) = (e^{p_2q_2} q_1, \: q_2),  \\
&\mbox{identity bisection}: \; \varepsilon(q_1, q_2) = (0, \: 0, \: q_1, \: q_2),  \\
&\mbox{inverse map}: \; \iota(p_1, p_2, q_1, q_2) = (-e^{p_1q_1}q_2, \: -e^{p_2q_2}q_1),  \\
&\mbox{multiplication}: \; \mbox{when} \; (e^{p_2q_2} q_1, \: q_2) = (q'_1, \: e^{p'_1q'_1}q'_2), \\ 
&  (p_1, p_2, q_1, q_2) \cdot (p'_1, p'_2, q'_1, q'_2) = (p_1 + e^{p_2q_2}p'_1, \: p'_2 + e^{p'_1q'_1}p_2, \: q_1, \: q'_2). 
\end{align*}
See \cite{li-rupel:cluster}, where this symplectic groupoid was constructed by different methods. 
\hfill $\diamond$
\end{exa}

\section{A pair of dual Poisson Lie groups associated to a standard semisimple Poisson Lie group} \label{sec-std-PL}

We recall in this section the standard complex semisimple Poisson Lie groups and an associated pair of dual Poisson Lie groups. Everything in this section is standard, and we refer to \cite{Lu-Mou:double-B-cell, Lu-Mou:mixed, Lu-victor:Tleaves} for details.

\subsection{Standard complex semisimple Poisson Lie groups} \label{subsec-ss-gps}

Let $\g$ be a complex semisimple Lie algebra with a fixed pair $(\b, \b_-)$ of opposite Borel subalgebras and a fixed non-degenerate symmetric ad-invariant bilinear form $\lara_\g$, and let $\h =  \b \cap \b_-$. Let $\triangle \subset \h^*$ be the roots of $\g$ with respect to $\h$ and $\triangle_+ \subset \triangle$ the positive roots defined by $\b$. One has the triangular decomposition $\g = \h + \sum_{\alpha \in \triangle} \g_\alpha$, and let $\n = \sum_{\alpha \in \triangle_+} \g_\alpha$, $\n_- = \sum_{\alpha \in \triangle_+} \g_{-\alpha}$. For each $\alpha \in \triangle_+$, we fix root vectors $E_{\pm \alpha} \in \g_{\pm \alpha}$ such that $\la E_\alpha, E_{-\alpha} \ra_\g = 1$. 

Let $G$ be a connected complex Lie group with Lie algebra $\g$, and let $B$, $B_-$, $T$, $N$, $N_-$, be the connected subgroups of $G$ with respective Lie algebras $\b$, $\b_-$, $\h$, $\n$, $\n_-$. Let $W = N_\sG(T)/T$ be the Weyl group of $(G, T)$, where $N_\sG(T)$ is the normaliser subgroup of $T$, and let $l: W \to \IN$ be the length function of $W$. Let 
$$
\Lambda_{st} = \sum_{\alpha \in \triangle_+} E_{-\alpha} \wedge E_\alpha \in \wedge^2 \g
$$
be the standard skew-symmetric $r$-matrix associated to the triple $(\b, \b_-, \lara_\g)$, and 
$$
\delta_{st}: \g \to \wedge^2 \g, \hs  \delta_{st}(x) = [x, \Lambda_{st}], \hs x \in \g,
$$
the corresponding standard Lie bialgebra structure on $\g$. The bivector field $\pist = \Lambda_{st}^L - \Lambda_{st}^R$ is a holomorphic multiplicative Poisson structure on $G$ such that $(G, \pist)$ has Lie bialgebra $(\g, \delta_\g)$, and $(G, \pist)$ is called a {\it standard complex semisimple Poisson Lie group}. Equipping the direct sum Lie algebra $\gog$ with the bilinear form 
$$
\la (x, x'), (y, y') \ra_{\gog} = \la x, y \ra_\g - \la x', y' \ra_\g, \hs x, x', y, y' \in \g,  
$$
one has the Manin triple $((\gog, \lara_{\gog}), \g_{diag}, \g')$, where $\g_{diag}$ is the diagonal subalgebra and 
$$
\g' = \{ (x_+ + x,\: -x + x_-) : x_+ \in \n, \; x_- \in \n_-, \; x \in \h \},
$$
such that $(\g, \delta_{st}) \cong (\g_{diag}, \delta_{\g_{diag}})$ under the isomorphism $x \mapsto (x,x)$, $x \in \g$. Thus $(\gog, \lara_{\gog})$ is the double Lie algebra of $(\g, \delta_{st})$.

\subsection{The coisotropic submanifold $C_u$}

For $u \in W$ and any representative $\bar{u} \in N_\sG(T)$ of $u$, let 
$$
C_{\bar{u}} = N\bar{u} \cap \bar{u}N_- \subset G.
$$
By \cite[Lemma 10]{Lu-Mou:double-B-cell}, $C_{\bar{u}}$ is a coisotropic submanifold of $(G, \pist)$ well known to be isomorphic to $\IC^{l(u)}$, and the multiplication in $G$ induces algebraic isomorphisms
\begin{equation} \label{BwBcongCB}
C_{\bar{u}} \times B \to BuB, \;\; (c, b) \mapsto cb \hs \text{and} \hs B_- \times C_{\bar{u}} \to B_-uB_-, \;\;(b_-, c) \mapsto b_-c.
\end{equation}
If $\bfu = (u_1, \ldots, u_n) \in W^n$, where $n \geq 1$, and if $\bar{\bfu} = (\bar{u}_1, \ldots, \bar{u}_n) \in N_\sG(T)^n$ is a representative for $\bfu$, let 
$$
C_{\bar{\bfu}} = C_{\bar{u}_1} \times \cdots \times C_{\bar{u}_n},
$$
and we will make use of the following notation: for $c = (c_1, \ldots, c_n) \in C_{\bar{\bfu}}$, we write 
\begin{equation} \label{eq-ddot}
\underline{c} = c_1 \cdots c_n \in G. 
\end{equation}

\subsection{The pair $((B, \pist)$, $(B_-, -\pist))$ of dual Poisson Lie groups} \label{subsec-bb_-}

The Lie algebras $\b$, $\b_-$ are sub-Lie bialgebras of $(\g, \delta_{st})$, and it is shown in \cite[Section 4]{Lu-Mou:double-B-cell} that when identifying $\b$ and $\b_-$ as dual vector spaces under the bilinear pairing 
$$
\la x_+ + x_0, \; y_- + y_0\ra_{(\b, \b_-)} = -\la x_+, y_- \ra_\g - 2\la x_0, y_0\ra_\g, \hs x_+ \in \n, \; x_0, y_0 \in \h, \; y_- \in \n_-,
$$
between $\b$ and $\b_-$, $((\b, \delta_{st}), (\b_-, -\delta_{st}))$ is a pair of dual Lie bialgebras. In particular, if $\{H_i\}_{i=1}^r$ is a basis of $\h$ satisfying $2\la H_i, H_j \ra_\g = \delta_{i, j}$, the bases 
$$
\{-E_\alpha\}_{\alpha \in \triangle_+} \cup \{-H_i\}_{i=1}^r \hs \text{and} \hs \{E_{-\alpha}\}_{\alpha \in \triangle_+} \cup \{H_i\}_{i=1}^r
$$
of $\b$ and $\b_-$ are dual with respect to $\lara_{(\b, \b_-)}$. Let $\d = \g \oplus \h$ as a direct sum Lie algebra, let $\lara_\d$ be the restriction to $\d \subset \gog$ of the bilinear form $- \lara_{\gog}$, and embed $\b$, $\b_-$ in $\d$ as $\bar{\b} = \{\bar{x}: x \in \b \}$ and $\bar{\bar{\b}}_- = \{ \bar{\bar{\xi}} : \xi \in \b_-\}$, where 
\begin{align*}
\bar{x} & = (x_+ + x_0, \: x_0), \hs \;\: \text{if} \hs x = x_+ + x_0, \hs \text{with} \;  x_+ \in \n, \; x_0 \in \h,   \\
\bar{\bar{\xi}} & = (y_- + y_0, \: -y_0), \hs \text{if} \hs \xi = y_- + y_0, \hs \text{with} \;  y_- \in \n_-, \; y_0 \in \h. 
\end{align*}
One checks that $((\d, \lara_\d), \bar{\b}, \bar{\bar{\b}}_-)$ is a Manin triple with $\la \bar{x}, \bar{\bar{\xi}} \ra_\d = \la x, \xi \ra_{(\b, \b_-)}$, for $x \in \b$, $\xi \in \b_-$, thus $(\d, \delta_\d)$ is the double Lie bialgebra of $(\b, \delta_{st})$.

Both $B$ and $B_-$ are Poisson Lie subgroups of $(G, \pist)$ and $((B, \pist)$, $(B_-, -\pist))$ is a pair of dual Poisson Lie groups with Drinfeld double $D = G \times T$, in which $B$ and $B_-$ are embedded as $\bar{B} = \{ \bar{b} : b \in B \}$ and $\bar{\bar{B}}_- = \{ \bar{\bar{b}}_- : b_- \in B_-\}$, where 
\begin{align*}
\bar{b} & = (hn, h), \hs \hs \text{if} \hs  b = hn,  \hs \text{with} \;  n \in N, \; h \in T,   \\
\bar{\bar{b}}_- & = (mh, h^{-1}),  \hs \text{if} \hs b_- = mh, \hs \text{with} \; m \in N_-, \; h \in T.  
\end{align*} 
The intersection of $\bar{B}$ and $\bar{\bar{B}}_-$ is $\bar{B} \cap \bar{\bar{B}}_- = \bar{T}^{(2)} = \bar{\bar{T}}^{(2)}$, where $T^{(2)} = \{ t \in T : t^2 = e \}$. Applying the theory recalled in $\S$\ref{subsec-dble-gpoid}, one has the two Poisson structures on $D$, $\piD = \Lambda_{\b, \b_-}^L - \Lambda_{\b, \b_-}^R$, $\pi^+_\sD = \Lambda_{\b, \b_-}^L + \Lambda_{\b, \b_-}^R$, where 
$$
\Lambda_{\b, \b_-} = \sum_{\alpha \in \triangle_+} \bar{\bar{E}}_{-\alpha} \wedge \bar{E}_{\alpha},  
$$
and in particular, the projection to the first factor 
\begin{equation} \label{eq-q}
(D, \piD) \to (G, \pist), \hs (g, h) \mapsto g, \hs g \in G, \; h \in T,
\end{equation}
is a morphism of Poisson Lie groups. Let 
$$
\Gamma = \{ (b, \: b_-, \: b'_-, \: b') \in B \times B_- \times B_- \times B : \; \bar{b}\bar{\bar{b}}_- = \bar{\bar{b'}}_-\bar{b'} \},
$$
and let $\pi_\sGam$ be the non-degenerate Poisson structure on $\Gamma$ defined in $\S$\ref{subsec-dble-gpoid}. Recall that $(\Gamma, \pi_\sGam)$ has two symplectic groupoid structures: one over $(B, \pist)$, denoted by $\Gamma_\sB$, and one over $(B_-, -\pist)$, denoted by $\Gamma_{\sB_-}$. Finally recall the dressing actions $\varrho_{\sB}: \b_- \to \XX^1(B)$ and $\vartheta_{\sB_-}: \b \to \XX^1(B_-)$ defined in \eqref{eq-varrhoG}, \eqref{eq-lambdaG*}.

\section{Generalised Double Bruhat cells} \label{sec-dble-bruhat}

In this section, we recall from \cite{Lu-Mou:mixed, Lu-victor:Tleaves} the generalised double Bruhat cell $G^{\bfu, \bfv}$ and the generalised Bruhat cell $\O^{\bfu}$ associated to finite sequences $\bfu, \bfv \in W^n$, where $n \geq 1$, and the holomorphic Poisson structures $\tilde{\pi}_{n, n}$ on $G^{\bfu, \bfv}$ and $\pi_n$ on $\O^{\bfu}$. 

We fix in this section a connected standard semisimple Poisson Lie group $(G, \pist)$ as in $\S$\ref{sec-std-PL}, and recall our notational conventions in $\S$\ref{nota-quotient}.

\subsection{Some quotient manifolds associated to $(G, \pist)$} \label{subsec-F_n}

We recall in this subsection some quotient manifolds associated to $(G, \pist)$ which were introduced in \cite{Lu-Mou:mixed, Lu-victor:Tleaves}. For an integer $n \geq 1$, let 
\begin{align*}
\tilde{F}_n & = G \times_\sB \cdots \times_\sB G, \hs \hs \hs \; F_n = G \times_\sB \cdots \times_\sB G/B,  \\
\tilde{F}_{-n} & = G \times_{\sB_-} \cdots \times_{\sB_-} G,  \hs \hs F'_{-n} = B_- \backslash G \times_{\sB_-} \cdots \times_{\sB_-} G.
\end{align*}
Then $\pist^n$ descends to well defined Poisson structures
\begin{align*}
\tilde{\pi}_n & = \varpi_{\scriptscriptstyle \tilde{F}_n}(\pist^n), \hs \hs \hs \pi_n  = \varpi_{\scriptscriptstyle F_n}(\pist^n),  \\
\tilde{\pi}_{-n} & = \varpi_{\scriptscriptstyle \tilde{F}_{-n}}(\pist^n), \hs \hs \; \pi'_{-n}  = \varpi_{\scriptscriptstyle F'_{-n}}(\pist^n),
\end{align*}
and one has a left Poisson action of $(B, \pist)$ on $(F_n, \pi_n)$ and a right Poisson action of $(B_-, \pist)$ on $(F'_{-n}, \pi'_{-n})$ given by 
\begin{align*}
\lambda_{\sF_n}(b, [g_1, \ldots, g_n]_{\sF_n}) & = [bg_1, g_2, \ldots, g_n]_{\sF_n}, \hs  b \in B, \; g_i \in G,  \\
\rho_{\sF'_{-n}}([g_1, \ldots, g_n]_{\sF'_{-n}}, b_-) & = [g_1, \ldots, g_{n-1}, g_nb_-]_{\sF'_{-n}},  \hs  b_- \in B_-, \; g_i \in G. 
\end{align*}
By $\S$\ref{subsec-bb_-}, one has the dual Poisson Lie groups 
\begin{equation} \label{LL^*}
(L^*, \pi_{\scriptscriptstyle L^*}) = (B, -\pist) \times (B, \pist)  \hs \text{and} \hs (L, \pi_{\scriptscriptstyle L}) = (B_-, \pist) \times (B_-, -\pist),  
\end{equation}
and  
\begin{align*}
\rho_{\scriptscriptstyle \tilde{F}_n} \left([g_1, \ldots, g_n]_{\scriptscriptstyle \tilde{F}_n}, \; (b_1, b_2) \right) & = [b_1^{-1}g_1, g_2, \ldots, g_{n-1}, g_nb_2]_{\scriptscriptstyle \tilde{F}_m}, \hs g_j \in G, \; b_i \in B,   \\
\lambda_{\scriptscriptstyle \tilde{F}_{-n}} \left( (b_{-1}, b_{-2}), \; [g_1, \ldots, g_n]_{\scriptscriptstyle \tilde{F}_{-n}} \right) & = [b_{-1}g_1, g_2, \ldots, g_{n-1}, g_nb_{-2}^{-1}]_{\scriptscriptstyle \tilde{F}_{-n}}, \hs g_j \in G, \; b_{-i} \in B_-, 
\end{align*}
are a right Poisson action of $(L^*, \pi_{\scriptscriptstyle L^*})$ on $(\tilde{F}_n, \tilde{\pi}_n)$ and a left Poisson action of $(L, \pi_{\scriptscriptstyle L})$ on $(\tilde{F}_{-n}, \tilde{\pi}_{-n})$, and let 
$$
\tilde{\pi}_{n,n} = \tilde{\pi}_n \times_{(\rho_{\scriptscriptstyle \tilde{F}_n}, \lambda_{\scriptscriptstyle \tilde{F}_{-n}})} \tilde{\pi}_{-n}
$$
be the mixed product Poisson structure on $\tilde{F}_{n,n} = \tilde{F}_n \times \tilde{F}_{-n}$ associated to the pair $(\rho_{\scriptscriptstyle \tilde{F}_n}, \lambda_{\scriptscriptstyle \tilde{F}_{-n}})$. The maximal torus $T$ acts by Poisson isomorphisms on the Poisson manifolds $(\tilde{F}_{\pm n}, \tilde{\pi}_{\pm n})$, $(F_n, \pi_n)$, and $(\tilde{F}_{n,n}, \tilde{\pi}_{n,n})$ by 
\begin{align*}
& t \cdot [g_1, \ldots, g_n]_{\scriptscriptstyle \tilde{F}_{\pm n}}  = [tg_1, g_2, \ldots, g_n]_{\scriptscriptstyle \tilde{F}_{\pm n}},  \hs \hs t \cdot [g_1, \ldots, g_n]_{\sF_n} = [tg_1, g_2, \ldots, g_n]_{\sF_n},   \\
& t \cdot \left( [g_1, \ldots, g_n]_{\scriptscriptstyle \tilde{F}_n}, \; [h_1, \ldots, h_n]_{\scriptscriptstyle \tilde{F}_{-n}}   \right) = \left( [tg_1, g_2, \ldots, g_n]_{\scriptscriptstyle \tilde{F}_n}, \; [th_1, h_2, \ldots, h_n]_{\scriptscriptstyle \tilde{F}_{-n}}   \right),
\end{align*}
where $t \in T$, $g_j, h_j \in G$, and the $T$-orbits of symplectic leaves of these Poisson manifolds are described in \cite{Lu-victor:Tleaves}.

\subsection{Generalised Bruhat cells} \label{subsec-gen-bru-cell}

Let $n \geq 1$ and $\bfu = (u_1, \ldots, u_n) \in W^n$. The submanifolds
\begin{align*}
\O^\bfu & = B u_ 1B \times_\sB \cdots \times_\sB B u_nB/B \subset F_n, \\ 
\O'^{-\bfu} & = B_- \backslash B_- u_1 B_- \times_{\sB_-} \cdots \times_{\sB_-} B_- u_n B_- \subset F'_{-n},
\end{align*}
are Poisson submanifolds of $(F_n, \pi_n)$ and $(F'_{-n}, \pi'_{-n})$, called in \cite{Lu-victor:Tleaves} {\it generalised Bruhat cells}. Let $\bar{\bfu} \in N_\sG(T)^n$ be a representative of $\bfu$. By an inductive use of the isomorphisms in \eqref{BwBcongCB}, the maps 
\begin{align*}
\varpi_{\sF_n} \mid_{C_{\bar{\bfu}}}: C_{\bar{\bfu}} \to \O^\bfu, \hs \text{and} \hs \varpi_{\sF'_n} \mid_{C_{\bar{\bfu}}}: C_{\bar{\bfu}} \to \O'^{-\bfu},
\end{align*}
are diffeomorphisms, hence in particular $\O^{\bfu} \cong \IC^{l(w_1) + \cdots + l(w_n)}$. Slightly abusing the notation, for $c = (c_1, \ldots, c_n) \in C_{\bar{\bfu}}$, we will write $[c]_{\sF_n} = \varpi_{\sF_n}(c)$ and $[c]_{\sF'_{-n}} = \varpi_{\sF'_n}(c)$.

\begin{lem} \label{lem-isom-I_u} 
The isomorphism 
$$
I_\bfu: (\O'^{-\bfu}, \pi'_{-n}) \to (\O^\bfu, \pi_n), \hs I_\bfu([c]_{\sF'_{-n}}) =  [c]_{\sF_n}, \hs c \in C_{\bar{\bfu}}, 
$$
is an anti-Poisson map. 
\end{lem}
\begin{proof}
Recall that a {\it Poisson pair} is a pair of Poisson maps $\phi_\sY: (X, \piX) \to (Y, \piY)$, $\phi_\sZ: (X, \piX) \to (Z, \piZ)$ between Poisson manifolds such that the map 
$$
\phi: (X, \piX) \to (Y \times Z, \piY \times \piZ), \hs \phi(x) = (\phi_\sY(x), \phi_\sZ(x)), \hs x \in X,
$$
is Poisson. By \cite[Lemma A.1]{Elek-Lu:BS}, if $X'$ is a coisotropic submanifold of $(X, \piX)$ such that $\phi_\sY \! \mid_{X'}: X' \to Y$ is a diffeomorphism, $\phi_\sZ \circ (\phi_\sY \! \mid_{X'})^{-1}: (Y, \piY) \to (Z, \piZ)$ is an anti-Poisson map. By \cite[Section 8]{Lu-Mou:mixed}, $\varpi_{\sF_n}$ and $\varpi_{\sF'_{-n}}$ form a Poisson pair, and $C_{\bbfu}$ is a coisotropic submanifold of $(G^n, \pist^n)$ contained in $G^{u_1, u_1} \times \cdots \times G^{u_n, u_n}$, thus Lemma \ref{lem-isom-I_u} follows by \cite[Lemma A.1]{Elek-Lu:BS}. Note that Lemma \ref{lem-isom-I_u} is proved in \cite[Proposition 5.15]{Elek-Lu:BS} when $\bfu$ consists of simple reflections in $W$. 
\end{proof}

The Poisson action $\lambda_{\sF_n}$ restricts to a Poisson action of $(B, \pist)$ on $(\O^{\bfu}, \pi_n)$, which we denote by $\lambda_\bfu: (B, \pist) \times (\O^{\bfu}, \pi_n) \to (\O^{\bfu}, \pi_n)$, and by Lemma \ref{lem-isom-I_u}, one also has a right action of $(B_-, - \pist)$ on $(\O^\bfu, \pi_n)$ given by 
$$
\rho_\bfu([c]_{\sF_n}, b_-) = I_\bfu( \rho_{\sF'_{-n}}([c]_{\sF'_{-n}}, b_-)), \hs b_- \in B_-, \;  c \in C_{\bbfu}.
$$
Consequently, via the diffeomorphisms $\varpi_{\sF_n} \mid_{C_{\bbfu}}$ and $ \varpi_{\sF'_n} \mid_{C_{\bfu}}$, one has a left action of $B$ on $C_{\bbfu}$ and a right action of $B_-$ on $C_{\bbfu}$, denoted by 
$$
(b, c) \mapsto b[c], \hs \hs (c, b_-) \mapsto c^{b_-}, \hs \hs b \in B, \; b_- \in B_-, \; c \in C_{\bbfu},  
$$
such that 
$$
\lambda_{\bfu}(b, [c]_{\sF_n}) = [b[c]]_{\sF_n} \hs \text{and} \hs \rho_\bfu([c]_{\sF_n}, b_-) = [c^{b_-}]_{\sF_n}, \hs b \in B, \; b_- \in B_-, \; c \in C_{\bbfu}. 
$$

\begin{lem} \label{lem-TO^bfu}
For $z \in \O^{\bfu}$, let $\Sigma_z \subset \O^{\bfu}$ be the $T$-orbit of symplectic leaves of $(\O^{\bfu}, \pi_n)$ containing $z$. Then 
\begin{equation} \label{tgtO^w}
T_z\O^{\bfu} = \lambda_\bfu(\b)(z) + T_z\Sigma_z. 
\end{equation}
\end{lem}
\begin{proof}
Denote the natural left action of $G$ on $G/B$ by $\lambda_1(g, g'.B) = gg'.B$, $g, g' \in G$, and consider the map 
$$
\mu_n: F_n \to G/B, \hs \mu_n([g_1, \ldots, g_n]_{\scriptscriptstyle F_n}) = g_1g_2 \cdots g_n.B, \hs g_j \in G. 
$$
By \cite[Theorem 1.1]{Lu-victor:Tleaves}, $\mu_n(T_z\Sigma_z) = \lambda_1(\b_-)(\mu_n(z))$. Thus \eqref{tgtO^w} follows, since $\mu_n$ is $B$-equivariant with respect to the actions $\lambda_{\sF_n}$ and $\lambda_1$ of $B$. 
\end{proof}

\subsection{The Poisson structures $\tilde{\pi}_{\pm n}$ on $B \bfu B$ and $B_- \bfu B_-$ as mixed products}

Let $n \geq 1$ and $\bfu \in W^n$ be as in $\S$\ref{subsec-gen-bru-cell} and let
\begin{align*}
B\bfu B & = B u_1 B \times_\sB \cdots \times_\sB B u_n B \subset \tilde{F}_n, \\
B_- \bfu B_- & = B_- u_1  B_- \times_{\sB_-} \cdots \times_{\sB_-} B_- u_n B_- \subset \tilde{F}_{-n}.
\end{align*}
As $(B, B)$- and $(B_-, B_-)$-cosets  in $G$ are Poisson submanifolds of $\pist$, $B \bfu B$ and $B_- \bfu B_-$ are Poisson submanifolds of $(\tilde{F}_n, \tilde{\pi}_n)$ and $(\tilde{F}_{-n}, \tilde{\pi}_{-n})$. Using \eqref{BwBcongCB} inductively, one has diffeomorphisms
\begin{align*}
J^+_{\bar{\bfu}} & : \O^\bfu \times B \to B \bfu B, \hs \hs \;\;\; J^+_{\bar{\bfu}} \left([c_1, \ldots, c_n]_{\scriptscriptstyle F_n}, \: b \right) = [c_1, \ldots, c_nb]_{\scriptscriptstyle \tilde{F}_n},     \\
J^-_{\bar{\bfu}} & : B_- \times \O^\bfu \to B_- \bfu B_-, \hs J^-_{\bar{\bfu}} \left(b_-, \: [c_1, \ldots, c_n]_{\scriptscriptstyle F_n} \right) =  [b_-c_1, \ldots, c_n]_{\scriptscriptstyle \tilde{F}_{-n}}, 
\end{align*}
where $(c_1, \ldots, c_n) \in C_{\bar{\bfu}}$, $b \in B$, and $b_- \in B_-$. Let 
$$
\lambda_+: (B, \pist) \times (B, \pist) \to (B, \pist), \hs \rho_-: (B_-, \pist) \times (B_-, \pist) \to (B_-, \pist),
$$
be respectively the action of $(B, \pist)$ on itself by left multiplication, and the action of $(B_-, \pist)$ on itself by right multiplication. The goal of this subsection is to prove the following 

\begin{pro} \label{lem-J^pm} 
One has 
\begin{align}
(J^+_{\bar{\bfu}})^{-1}(\tilde{\pi}_n) & = \pi_n \times_{(\rho_\bfu, \: \lambda_+)} \pist,  \label{eq-J^+_bfu}    \\
(J^-_{\bar{\bfu}})^{-1}(\tilde{\pi}_{-n}) & =  \pist \times_{(\rho_-, \: \lambda_{\bfu})} (-\pi_n), \label{eq-J^-_bfu} 
\end{align}
where the pair of dual Poisson Lie groups involved in \eqref{eq-J^+_bfu} and \eqref{eq-J^-_bfu} are respectively $((B_-, -\pist), (B, \pist))$ and $((B_-, \pist), (B, -\pist))$. 
\end{pro}

As the proofs of \eqref{eq-J^+_bfu} and \eqref{eq-J^-_bfu} are similar, we will only prove \eqref{eq-J^-_bfu}. The proof of Proposition \ref{lem-J^pm} is completely similar to the proof of \cite[Proposition 9]{Lu-Mou:double-B-cell}. In particular, Lemmas \ref{lem-K-bfu} and \ref{lem-Phi} below are completely analogous to \cite[Proposition 9]{Lu-Mou:double-B-cell} and \cite[Remark 9]{Lu-Mou:double-B-cell}.

\begin{lem} \label{lem-q-pm}
The maps 
\begin{align*}
q_{\bar{\bfu}}^-: & (B_-\bfu B_-, \tilde{\pi}_{-n}) \to (B_-, \pist), \;\; q_{\bar{\bfu}}^-([b_-c_1, \ldots, c_n]_{\tilde{\sF}_{-n}}) = b_-, \;\; (c_1, \ldots, c_n) \in C_{\bar{\bfu}}, \; b_- \in B_-,  \\
q_{\bar{\bfu}}^+: & (B \bfu B, \tilde{\pi}_n) \to (B, \pist), \hs q_{\bar{\bfu}}^+([c_1, \ldots, c_nb]_{\tilde{\sF}_n}) = b, \hs (c_1, \ldots, c_n) \in C_{\bar{\bfu}}, \; b \in B, 
\end{align*}
are Poisson. 
\end{lem}
\begin{proof}
When $n = 1$, Lemma \ref{lem-q-pm} is proven in \cite[Lemma 11]{Lu-Mou:double-B-cell}, and for $n > 1$, let $\bfu' = (u_2, \ldots, u_n)$. Then the statement for $q_{\bar{\bfu}}^-$ follows by induction as a consequence of the following commutative diagram 
$$
\begin{tikzcd}
(B_-u_1B_-, \pist) \times (B_- \bfu' B_-, \tilde{\pi}_{-(n-1)}) \arrow[d]  \arrow[r] & (B_-u_1B_-, \pist) \arrow[d, "q_{\bar{u}_1}^-"] \\
(B_- \bfu B_-, \tilde{\pi}_{-n})  \arrow[r, "q_{\bar{\bfu}}^-"]  &  (B_-, \pist),  
\end{tikzcd}
$$
where the top arrow is the map 
$$
(g_1, [g_2, \ldots, g_n]_{\tilde{\sF}_{-(n-1)}}) \mapsto g_1 q_{\bar{\bfu}'}^-([g_2, \ldots, g_n]_{\tilde{\sF}_{-(n-1)}}), \hs g_i \in B_-u_iB_-, 
$$
and the first vertical arrow the map
$$
(g_1, [g_2, \ldots, g_n]_{\tilde{\sF}_{-(n-1)}}) \mapsto [g_1, g_2, \ldots, g_n]_{\tilde{\sF}_{-n}}, \hs g_i \in B_-u_iB_-.
$$
The statement for $q_{\bar{\bfu}}^+$ is proven similarly. 
\end{proof}

By definition of $\tilde{\pi}_{n,n}$ it is clear that $\tilde{F}_n \times B_- \bfu B_-$ is a Poisson submanifold of $(\tilde{F}_{n,n}, \tilde{\pi}_{n,n})$. Let $\lambda_-: (B_-, \pist) \times (B_-, \pist) \to (B_-, \pist)$ be the left action of $(B_-, \pist)$ on itself by left multiplication. The following Lemma \ref{lem-K-bfu} is analogous to  \cite[Proposition 9]{Lu-Mou:double-B-cell}. 

\begin{lem} \label{lem-K-bfu}
Let $K_{\bbfu}: \tilde{F}_n \times B_- \bfu B_- \to F_{n} \times B_-$ be the map
$$
K_{\bbfu}([g_1, \ldots, g_n]_{\tilde{\sF}_n}, \: [b_-c_1, \ldots, c_n]_{\tilde{\sF}_{-n}})= ([g_1, \ldots, g_n]_{\sF_n}, b_-), \hs g_i \in G, \; c_i \in C_{\bar{u}_i}, \; b_- \in B_-.
$$
Then one has $K_{\bbfu}(\tilde{\pi}_{n,n}) =  \pi_n \times_{(-\lambda_{\sF_n}, \: \lambda_-)} \pist$, where the pair of dual Poisson Lie groups involved in the mixed product is $((B, -\pist), (B_-, \pist))$.   
\end{lem}
\begin{proof}
By definition of $\tilde{\pi}_{n,n}$, one has $\tilde{\pi}_{n,n} = (\tilde{\pi}_n, 0) + (0, \tilde{\pi}_{-n}) - \mu_1  - \mu_2$, where
$$
\mu_1 = \rho_{\tilde{\sF}_n}(x_i, 0) \wedge \lambda _{\tilde{\sF}_{-n}}(\xi^i, 0), \hs \text{and} \hs \mu_2 =  \rho_{\tilde{\sF}_n}(0, -x_i) \wedge \lambda _{\tilde{\sF}_{-n}}(0, \xi^i),
$$
where $(x_i)$ is any basis of $\b$ and $(\xi^i)$ the dual basis of $\b_-$ with respect to $\lara_{(\b, \b_-)}$. By Lemma \ref{lem-q-pm} one has $K_{\bbfu}(0, \tilde{\pi}_{-n}) = (0, \pist)$, and by definition of $K_{\bbfu}$, one has $K_{\bbfu}(\tilde{\pi}_n, 0) = (\pi_n, 0)$, $K_{\bbfu}(\mu_2) = 0$, and $K_{\bbfu}(\mu_1)$ coincides with the mixed term of $\pi_n \times_{(-\lambda_{\sF_n}, \: \lambda_-)} \pist$. This proves Lemma \ref{lem-K-bfu}. 
\end{proof}

The following Lemma \ref{lem-Phi} is a straightforward calculation completely similar to \cite[Remark 9]{Lu-Mou:double-B-cell}. 

\begin{lem} \label{lem-Phi}
Let $\Phi: F_n \times B_- \to B_- \times  F_n$ be the map 
$$
\Phi([g_1, \ldots, g_n]_{\sF_n}, \: b_-) = (b_-, \: [b_-^{-1}g_1, g_2, \ldots, g_n]_{\sF_n}), \hs g_i \in G, \; b_- \in B_-.
$$
Then one has $\Phi(\pi_n \times_{(-\lambda_{\sF_n}, \: \lambda_-)} \pist) = \pist \times_{(\rho_-, \lambda_{\sF_n})} (-\pi_n)$. 
\end{lem}

Let $\Lambda_{\g_{diag}, \g'} \in \wedge^2 (\gog)$ be the skew-symmetric $r$-matrix associated to the Lagrangian splitting $\gog = \g_{diag} + \g'$, and let $\Pist = \Lambda_{\g_{diag}, \g'}^L - \Lambda_{\g_{diag}, \g'}^R \in \XX^2(G \times G)$ be the corresponding holomorphic multiplicative Poisson structure on $G \times G$. Then the Poisson structure $\Pist^n$ on $(G \times G)^n$ descends to a well defined Poisson structure $\pi_{\scriptscriptstyle \tilde{\IF}_n} = \varpi_{\scriptscriptstyle \tilde{\IF}_n}(\Pist^n)$ on $\tilde{\IF}_n$, where 
$$
\tilde{\IF}_n = (G \times G) \times_{\sB \times \sB_-} \cdots \times_{\sB \times \sB_-} (G \times G), 
$$
and by \cite[Proposition 8.3]{Lu-Mou:mixed} the diffeomorphism $S_{\scriptscriptstyle \tilde{\IF}_n}: \tilde{\IF}_n \to \tilde{F}_{n, n}$, 
$$
S_{\scriptscriptstyle \tilde{\IF}_n}([(g_1, h_1), \ldots, (g_n, h_n)]_{\scriptscriptstyle \tilde{\IF}_n}) = \left( [g_1, \ldots, g_n]_{\tilde{\sF}_n}, \;  [h_1, \ldots, h_n]_{\tilde{\sF}_{-n}} \right), \hs g_i, h_i \in G,
$$
is a Poisson isomorphism between $(S_{\scriptscriptstyle \tilde{\IF}_n}, \pi_{\scriptscriptstyle \tilde{\IF}_n})$ and $(\tilde{F}_{n,n}, \tilde{\pi}_{n,n})$.

\noindent
{\it Proof of Proposition \ref{lem-J^pm}.}
As $g \mapsto (g, g)$, $g \in G$ is a Poisson embedding of $(G, \pist)$ into $(G \times G, \Pist)$ and as $B_-u_i B_-$ is a Poisson submanifold of $(G, \pist)$, \eqref{eq-J^-_bfu} is now a consequence of Lemmas \ref{lem-K-bfu} and \ref{lem-Phi}, and of the following commutative diagram 
$$
\begin{tikzcd}
B_-u_1B_- \times \cdots \times B_- u_n B_- \arrow[d, "\varpi_{\tilde{\sF}_{-n}}"]  \arrow[r] & (B_-u_1B_-)_{diag} \times \cdots \times (B_- u_n B_-)_{diag} \arrow[d, "\Phi \circ K_{\bbfu} \circ S_{\scriptscriptstyle \tilde{\IF}_n} \circ\varpi_{\scriptscriptstyle \tilde{\IF}_n}"] \\
B_- \bfu B_-  \arrow[r, "(J^-_{\bar{\bfu}})^{-1}"]  & B_- \times \O^\bfu,  
\end{tikzcd}
$$
where the top arrow is the map $g \mapsto (g, g)$, $g \in B_-u_iB_-$, applied on each factor. 
\qed

\subsection{The Poisson structure $\tilde{\pi}_{n,n}$ on $B \bfu B \times B_- \bfv B_-$ as a mixed product}

Let $n \geq 1$ and $\bfu \in W^n$ be as in $\S$\ref{subsec-gen-bru-cell}. For $c \in C_{\bar{\bfu}}$, $b \in B$ and $b_- \in B_-$, and recalling the notation in \eqref{eq-ddot}, let 
\begin{equation}
b_{\bbfu}(b, c) \in B \hs \text{and} \hs b_{-\bbfu}(b_-, c) \in B_-
\end{equation}
be the well defined elements such that 
$$
b \underline{c} = \underline{b[c]} \: b_{\bbfu}(b, c) , \hs \text{and} \hs \underline{c}b_- = b_{-\bbfu}(b_-, c) \underline{c^{b_-}}. 
$$
Recalling the groups $L$ and $L^*$ defined in \eqref{LL^*}, one has, via the diffeomorphisms $J^\pm_{\bbfu}$ the right action of $L^*$ on $\O^\bfu \times B$ and the left action of $L$ on $B_- \times \O^\bfu$ given by 
\begin{align*}
\tilde{\rho}_{\bbfu} \left( ([c]_{\scriptscriptstyle F_n}, b), (b_1, b_2) \right) & = (J^+_{\bbfu}
)^{-1} \left( \rho_{\scriptscriptstyle \tilde{F}_n} \left( J^+_{\bbfu}([c]_{\scriptscriptstyle F_n}, b), (b_1, b_2) \right) \right),   \\
\tilde{\lambda}_{\bbfu} \left((b_{-1}, b_{-2}), (b_-, [c]_{\scriptscriptstyle F_n}) \right) & = (J^-_{\bbfu})^{-1} \left( \lambda_{\scriptscriptstyle \tilde{F}_{-n}} \left(  (b_{-1}, b_{-2}), J^-_{\bbfu}(b_-, [c]_{\scriptscriptstyle F_n}) \right) \right),
\end{align*}
where $c \in C_{\bar{\bfu}}$, $b, b_1, b_2 \in B$, and $b_-, b_{-1}, b_{-2} \in B_-$. The next Lemma \ref{lem-rho_tilde} is straightforward. 

\begin{lem} \label{lem-rho_tilde}
One has 
\begin{align*}
\tilde{\rho}_{\bbfu} \left( ([c]_{\sF_n}, b), (b_1, b_2) \right) & = \left( \lambda_\bfu(b_1^{-1}, [c]_{\sF_n}), \; b_{\bbfu}(b_1^{-1}, c)bb_2 \right),   \\
\tilde{\lambda}_{\bbfu} \left((b_{-1}, b_{-2}), (b_-, [c]_{\sF_n}) \right) & = \left(b_{-1}b_-b_{-\bbfu}(b_{-2}^{-1}, c), \; \rho_\bfu([c]_{\sF_n}, b_{-2}^{-1} \right). 
\end{align*}
\end{lem}

Let $\bfv \in W^n$ and let $\bbfv \in N\sG(T)^n$ be a representative of $\bfv$. Since $B \bfu B \times B_- \bfv B_-$ is a Poisson submanifold of $(\tilde{F}_{n,n}, \tilde{\pi}_{n,n})$, let 
$$
J_{\bbfu, \bbfv}  = J^+_{\bbfu} \times J^-_{\bbfv}: \: \O^\bfu \times B \times B_- \times \O^\bfv \longrightarrow B \bfu B \times B_- \bfv B_-,
$$
and $\pi_{\bbfu, \bbfv} = (J_{\bbfu, \bbfv})^{-1}(\tilde{\pi}_{n,n})$. By Proposition \ref{lem-J^pm}, one has 
$$
\pi_{\bbfu, \bbfv} = \left( \pi_n \times_{(\rho_\bfu, \: \lambda_+)} \pist \right) \times_{(\tilde{\rho}_{\bbfu}, \: \tilde{\lambda}_{\bbfv})} \left( \pist \times_{(\rho_-, \: \lambda_\bfv)} (-\pi_n)  \right).  
$$
Let $(x_i)$ be a basis of $\b$ and $(\xi^i)$ the dual basis of $\b_-$ with respect to the bilinear form $\lara_{(\b, \b_-)}$. In details, one has 
\begin{align}
\pi_{\bbfu, \bbfv} = &\:  (\pi_n, \pist, \pist, -\pi_n)  \label{eq-pi_uu}  \\
   & - (\rho_\bfu(\xi^i), 0, 0, 0) \wedge (0, (x_i)^R, 0, 0) + (0, 0, (\xi^i)^L, 0) \wedge (0, 0, 0, \lambda_\bfv(x_i)) \notag  \\
    & - (\lambda_\bfu(x_i), b_{\bbfu}(x_i)^R, 0, 0) \wedge (0, 0, (\xi^i)^R, 0) + (0, (x_i)^L, 0, 0) \wedge (0, 0, b_{-\bbfv}(\xi^i)^L, \rho_\bfv(\xi^i)),   \notag
\end{align}
where 
$$
b_{\bar{\bfu}}(x_i) = \frac{d}{dt} \mid_{t = 0} b_{\bar{\bfu}}(\exp(tx_i), c) \in \b, \hs \text{and} \hs b_{-\bbfv}(\xi^i) = \frac{d}{dt} \mid_{t = 0} b_{-\bbfv}(\exp(t\xi^i), c_-) \in \b_-.
$$

\subsection{Generalised double Bruhat cells as Lie groupoids} \label{subsec-gpoid-Guu}

Let $n \geq 1$ and $\bfu, \bfv \in W^n$. The submanifold 
\begin{align*}
G^{\bfu, \bfv} = \{ &  \left( [g_1, \ldots, g_n]_{\tilde{\sF}_n}, \; [h_1, \ldots, h_n]_{\tilde{\sF}_{-n}} \right) \in B \bfu B \times B_- \bfv B_- :   \\
  & g_1 \cdots g_n = h_1 \cdots h_n, \hs g_j, h_j \in G \}  
\end{align*}
of $\tilde{F}_{n,n}$ is called a {\it generalised double Bruhat cell} in \cite{Lu-Mou:double-B-cell}, and it is shown therein that it is a Poisson submanifold for $\tilde{\pi}_{n,n}$, consisting of a unique $T$-orbit of symplectic leaves. When $n = 1$, $(G^{\bfu, \bfv}, \pi_{1,1})$ is naturally isomorphic to the double Bruhat cell $(G^{u, v}, \pist)$, if $\bfu = (u)$, $\bfv = (v)$. Fixing representatives $\bbfu, \bbfv \in N_\sG(T)^n$ for $\bfu$, $\bfv$, let 
$$
\G^{\bbfu, \bbfv}  = (J_{\bbfu, \bbfv})^{-1}(G^{\bfu, \bfv}) \subset  \O^\bfu \times B \times B_- \times \O^\bfv. 
$$
Recalling the notation introduced in \eqref{eq-ddot}, one has 
\begin{align} 
\G^{\bbfu, \bbfv}  & =  \{\left( [c]_{\sF_n}, \; b, \; b_-, \; [c_-]_{\sF_n} \right) \in \O^\bfu \times B \times B_- \times \O^\bfv: \;  \underline{c}b = b_-\underline{c_-}, \label{eq-G^uu} \\
   & \hs \hs c \in C_{\bbfu}, \;  c_- \in C_{\bbfv}, \; b \in B, \; b_- \in B_- \},  \notag
\end{align}
and $\G^{\bbfu, \bbfv}$ is a Poisson submanifold of $(\O^\bfu \times B \times B_- \times \O^\bfv, \: \pi_{\bbfu, \bbfv})$. Furthermore, when $\bfu = \bfv$, $\G^{\bar{\bfu}, \bar{\bfu}}$ has a structure of a groupoid over $\O^\bfu$ given by 

\begin{align*}
&\mbox{source map}: \; \theta_{\bar{\bfu}}([c]_{\sF_n}, \: b, \: b_-, \: [c_-]_{\sF_n}) = [c]_{\sF_n},  \\
&\mbox{target map}: \; \tau_{\bar{\bfu}}([c]_{\sF_n}, \: b, \: b_-, \: [c_-]_{\sF_n}) = [c_-]_{\sF_n},  \\
&\mbox{identity bisection}: \; \varepsilon_{\bar{\bfu}}([c]_{\sF_n}) = ([c]_{\sF_n}, \:e, \: e, \: [c]_{\sF_n})  \\
&\mbox{inverse map}: \; \iota_{\bar{\bfu}}([c]_{\sF_n},\: b, \: b_-, \: [c_-]_{\sF_n}) = ([c_-]_{\sF_n}, \:b^{-1}, \: b_-^{-1}, \: [c]_{\sF_n})  \\
&\mbox{multiplication}: \; \mbox{when} \;  [c_-]_{\sF_n} =  [c']_{\sF_n},  \\ 
& ([c]_{\sF_n}, \: b, \: b_-, \: [c_-]_{\sF_n}) ([c']_{\sF_n}, \: b', \: b'_-, \: [c'_-]_{\sF_n}) = ([c]_{\sF_n}, \: bb', \: b_-b'_-, \: [c'_-]_{\sF_n}). 
\end{align*}

\begin{pro} \label{lem-theta-submersion}
The above maps define a Lie groupoid structure on $\G^{\bar{\bfu}, \bar{\bfu}}$. 
\end{pro}
\begin{proof}
It is clear that the above maps define a set theoretic groupoid structure on $\G^{\bar{\bfu}, \bar{\bfu}}$ and that the identity bisection and the multiplication are holomorphic maps. Thus one only needs to check that $\theta_{\bar{\bfu}}$ and $\tau_{\bar{\bfu}}$ are submersions. 

Let $([c]_{\sF_n}, b, b_-, [c_-]_{\sF_n}) \in \G^{\bar{\bfu}, \bar{\bfu}}$ with $c, c_- \in C_{\bar{\bfu}}$, $b \in B$, $b_- \in B_-$, and let $\alpha \in T^*_{[c]_{\scriptscriptstyle F_n}}\O^\bfu$, and $x \in \b$. Viewing $x$ as a linear form on $\b_-$ via the pairing $\lara_{(\b, \b_-)}$ and using \eqref{eq-pi_uu} , one has 
\begin{align*}
\theta_\bfu \left((\pi_{\bar{\bfu}, \bar{\bfu}})^\sharp(\alpha, 0, r^*_{b_-^{-1}}x, 0) \right) & = (\pi_n)^\sharp(\alpha) - \lambda_\bfu(x)([c]_{\scriptscriptstyle F_n}),
\end{align*} 
thus by Lemma \ref{lem-TO^bfu}, $\theta_{\bar{\bfu}}$ is a submersion. Similarly, $\tau_{\bar{\bfu}}$ is also a submersion. 
\end{proof}

When $n = 1$, The groupoid structure on $\G^{\bbfu, \bbfu}$ coincides with the groupoid structure on the double Bruhat cell $G^{u,u}$ introduced in \cite{Lu-Mou:double-B-cell}, where $\bfu = (u)$.

\section{Poisson action of a double symplectic groupoid on generalised double Bruhat cells} \label{sec-lhd-BB_-}

Let $(G, \pist)$ be a standard complex semisimple Poisson Lie group as in $\S$\ref{sec-std-PL}, let $n \geq 1$, $\bfu, \bfv \in W^n$, and fix representatives $\bbfu, \bbfv \in N_\sG(T)^n$ of $\bfu, \bfv$. We construct in this section right Poisson actions of the symplectic groupoids $(\Gamma_\sB \rightrightarrows B, -\pi_\sGam)$ and $(\Gamma_{\sB_-} \rightrightarrows B_-, \pi_\sGam)$ on the Poisson maps 
\begin{align*}
\mu_+: (\G^{\bbfu, \bbfv}, \: \pi_{\bbfu, \bbfv}) \to (B, \pist), & \hs \mu_+([c]_{\sF_n}, \: b, \: b_-, \: [c_-]_{\sF_n}) = b,   \\
\mu_-: (\G^{\bbfu, \bbfv}, \: \pi_{\bbfu, \bbfv}) \to (B_-, \pist), & \hs \mu_-([c]_{\sF_n}, \: b, \: b_-, \: [c_-]_{\sF_n}) = b_-,  
\end{align*}
where $([c]_{\sF_n}, b,  b_-,  [c_-]_{\sF_n}) \in \G^{\bbfu, \bbfv}$.

\subsection{Twisted multiplicative actions of $\Gamma_\sB$ and $\Gamma_{\sB_-}$} \label{subec-gpoid-action}

Identifying $\b$ and $\b_-$ as dual vector spaces via the pairing $\lara_{(\b, \b_-)}$, one has the dressing actions 
\begin{align*}
\varrho_{\scriptscriptstyle \G^{\bbfu, \bbfv}}: \b_- \to \XX^1(\G^{\bbfu, \bbfv}), & \hs \varrho_{\scriptscriptstyle \G^{\bbfu, \bbfv}}(\xi) = \pi^\sharp_{\bbfu, \bbfv}(\mu_+^*\xi^L), \hs \xi \in \b_-,   \\
\vartheta_{\scriptscriptstyle \G^{\bbfu, \bbfv}}: \b \to \XX^1(\G^{\bbfu, \bbfv}), & \hs \vartheta_{\scriptscriptstyle \G^{\bbfu, \bbfv}}(x) = \pi^\sharp_{\bbfu, \bbfv}(\mu_-^*x^R), \hs x \in \b. 
\end{align*}

\begin{lem}
Let $\tilde{y} = ([c]_{\sF_n}, \: b, \: b_-, \: [c_-]_{\sF_n}) \in \G^{\bbfu, \bbfv}$, with $c \in C_{\bbfu}$, $c_- \in C_{\bbfv}$, $b \in B$, $b_- \in B_-$. For $x \in \b$ and $\xi \in \b_-$, one has 
\begin{align*}
\varrho_{\scriptscriptstyle \G^{\bbfu, \bbfv}}(\xi)(\tilde{y}) & = (\rho_\bfu(\Ad^*_{b^{-1}}\xi)([c]_{\sF_n}), \; \varrho_\sB(\xi)(b), \; \mu_-(\varrho_{\scriptscriptstyle \G^{\bbfu, \bbfv}}(\xi)(\tilde{z})), \;  \rho_\bfv(\xi)([c_-]_{\sF_n})),  \\
\vartheta_{\scriptscriptstyle \G^{\bbfu, \bbfv}}(x)(\tilde{y}) & = (\lambda_\bfu(x)([c]_{\sF_n}), \; \mu_+(\vartheta_{\scriptscriptstyle \G^{\bbfu, \bbfv}}(x)(\tilde{z})), \; \vartheta_{\sB_-}(x)(b_-), \;  \lambda_\bfv(\Ad^*_{b_-}x)([c_-]_{\sF_n})). 
\end{align*}
\end{lem}
\begin{proof}
As $\varrho_{\scriptscriptstyle \G^{\bbfu, \bbfv}}(\xi)$ and $\vartheta_{\scriptscriptstyle \G^{\bbfu, \bbfv}}(x)$ are tangent to $\G^{\bbfu, \bbfv}$, the $B_-$-component of $\varrho_{\scriptscriptstyle \G^{\bbfu, \bbfv}}(\xi)$ is determined by the three others, and similarly, the $B$-component of $\vartheta_{\scriptscriptstyle \G^{\bbfu, \bbfv}}(x)$ is determined by the other three. The Lemma follows by pairing $\pi^\sharp_{\bar{\bfu}, \bar{\bfu}}$ given in \eqref{eq-pi_uu} with $(0, \xi^L, 0, 0)$ and $(0, 0, x^R, 0)$.  
\end{proof}

The next Proposition \ref{pro-lhd-BB_-} is proven by straightforward calculation.

\begin{pro} \label{pro-lhd-BB_-}
One has a right Lie groupoid action of $\Gamma_\sB$ on $\mu_+$ given by 
$$
\tilde{y} \lhd_\sB \gamma = \left( [c^{u'}]_{\sF_n}, \;\;  b', \;\; b_{-\bbfu}(u', c)^{-1}b_- b_{-\bbfv}(u, c_-) , \;\; [c_-^u]_{\sF_n}\right),  
$$
where $\tilde{y} = ([c]_{\sF_n}, \: b, \: b_-, \: [c_-]_{\sF_n}) \in \G^{\bbfu, \bbfv}$, $\gamma = (b, u, u', b') \in \Gamma_\sB$, with $c \in C_{\bbfu}$, $c_- \in C_{\bbfv}$, $b, b' \in B$, $b_-, u, u' \in B_-$, and a right Lie groupoid action of $\Gamma_{\sB_-}$ on $\mu_-$, given by 
$$
\tilde{y} \lhd_{\sB_-} \gamma_- = \left( [g[c]]_{\sF_n}, \;\; b_{\bbfu}(g, c)b b_{\bbfv}(g', c_-)^{-1}, \;\; b'_-, \;\; [g'[c_-]]_{\sF_n} \right), 
$$
where $\tilde{y} = ([c]_{\sF_n}, \: b, \: b_-, \: [c_-]_{\sF_n}) \in \G^{\bbfu, \bbfv}$, $\gamma_- = (g, b_-, b'_-, g') \in \Gamma_{\sB_-}$, with $c \in C_{\bbfu}$, $c_- \in C_{\bbfv}$, $b, g, g' \in B$, $b_-, b'_- \in B_-$. When $\bfu = \bfv$, these Lie groupoid actions satisfy the twisted multiplicativity properties \eqref{eq-Y-twist} - \eqref{eq-Z-twist}. 
\end{pro}

The remainder of $\S$\ref{sec-lhd-BB_-} is devoted to proving the following 

\begin{thm} \label{thm-lhd-BB_-}
The actions $\lhd_\sB$ and $\lhd_{\sB_-}$ are Poisson actions of $(\Gamma_\sB, -\pi_\sGam)$ and $(\Gamma_{\sB_-}, \pi_\sGam)$. 
\end{thm}

As the case for $\lhd_\sB$ and $\lhd_{\sB_-}$ are completely parallel, we will only prove Theorem \ref{thm-lhd-BB_-} for $\lhd_\sB$.

\begin{lem} \label{lem-lhdB-diract1}
The action $\lhd_\sB$ satisfies \eqref{eq-dirac1}.  
\end{lem}
\begin{proof}
Let $\tilde{y} = ([c]_{\sF_n}, \: b, \: b_-, \: [c_-]_{\sF_n}) \in \G^{\bbfu, \bbfv}$ and $u \in B_-$. The Lemma is proven by differentiating 
$$
\tilde{y} \lhd_\sB \gamma_{b, u} =  \left( [c^{b[u]}]_{\sF_n}, \;\;  b^u, \;\; b_{-\bbfu}(b[u], c)^{-1}b_- b_{-\bbfv}(u, c_-) , \;\; [c_-^u]_{\sF_n}\right)
$$
in $u$ using \eqref{lem-Ad}. 
\end{proof}

\subsection{A model space for $\G^{\bbfu, \bbfv}$ with an action of $B_-$}

To prove that $\lhd_\sB$ satisfies \eqref{eq-dirac2}, and thus that it is a Poisson action, we will construct a Poisson immersion from $(\G^{\bbfu, \bbfv}, \pi_{\bbfu, \bbfv})$ into a Poisson manifold with a Poisson action of $(B_-, \pist)$, and apply Proposition \ref{pro-gp/gpoid-action}. We start with the following simplification. Let 
$$
f_{\bfu, \bfv}: \O^{\bfu} \times B \times B_- \times \O^{\bfv} \to   \O^{\bfu} \times B \times \O^{\bfv}, \hs f_{\bfu, \bfv}([c]_{\sF_n}, b, b_-, [c_-]_{\sF_n}) = ([c]_{\sF_n}, b,  [c_-]_{\sF_n}), 
$$
where $c \in C_{\bbfu}$, $c_- \in C_{\bbfv}$, $b \in B$, $b_- \in B_-$, so that
\begin{align*}
f_{\bfu, \bfv}(\pi_{\bbfu, \bbfv}) = (\pi_n, \: \pist, \: - \pi_n) - (\rho_\bfu(\xi^i), 0, 0) \wedge (0, (x_i)^R, 0) + (0, (x_i)^L, 0) \wedge (0, 0, \rho_\bfv(\xi^i))
\end{align*}
is a Poisson structure on $\O^{\bfu} \times B \times \O^{\bfv}$, where $(x_i)$ is basis of $\b$ and $(\xi^i)$ the dual basis with respect to $\lara_{(\b, \b_-)}$, and we identify $(\G^{\bbfu, \bbfv}, \pi_{\bbfu, \bbfv})$ with its image 
$$
f_{\bfu, \bfv}(\G^{\bbfu, \bbfv}) = \{ ([c]_{\sF_n}, b, [c_-]_{\sF_n}) : \underline{c}b (\underline{c_-})^{-1} \in B_-, \; c \in C_{\bbfu}, c_- \in C_{\bbfv}, \; b \in B\}
$$
in $(\O^{\bfu} \times B \times \O^{\bfv}, f_{\bfu, \bfv}(\pi_{\bbfu, \bbfv}))$.

Let $B_-^2$ act on the right of $\O^{\bfu} \times D$ by 
$$
([c]_{F_n}, \: d) \cdot (b_{-1}, b_{-2}) = ([c^{b_{-1}}]_{F_n}, \: (\bar{\bar{b}}_{-2})^{-1}d), \hs c \in C_{\bar{\bfu}}, \; d \in D, \; b_{-j} \in B_-,
$$
and let $Z_{\bfu, \sD} = \O^{\bfu} \times_{\sB_-} D$ be the quotient of $\O^{\bfu} \times D$ by the diagonal subgroup $(B_-)_{\diag}$ of $ B_-^2$. Let $\varpi: \O^{\bfu} \times D \to Z_{\bfu, \sD}$ be the quotient map and write 
$$
[[c]_{F_n}, \: d] = \varpi([c]_{F_n}, \: d), \hs c \in C_{\bar{\bfu}}, \; d \in D. 
$$
As the left multiplication $(D, -\piD) \times (D, \pi^+_\sD) \to (D, \pi^+_\sD)$ and the right multiplication $(D, \pi^+_\sD) \times (D, \piD) \to (D, \pi^+_\sD)$ are Poisson actions, and since $(B_-)_{\diag}$ is a coisotropic subgroup of $(B_-, -\pist) \times (B_-, \pist)$, the direct product Poisson structure $\pi_n \times \pi^+_\sD$ on $\O^{\bfu} \times D$ descends to a well defined Poisson structure $\pi_{\bfu, \sD}$ on $Z_{\bfu, \sD}$, and one has a right Poisson action of $(D, \piD)$ on $(Z_{\bfu, \sD}, \pi_{\bfu, \sD})$ given by 
$$
\rho_{\scriptscriptstyle Z_{\bfu, \sD}}([[c]_{F_n}, \: d], \: d') = [[c]_{F_n}, \: dd'], \hs c \in  C_{\bar{\bfu}}, \; d, d' \in D. 
$$
Recall from \eqref{eq-J^+_bfu} the mixed product Poisson structure $\pi_n \times_{(\rho_\bfu, \: \lambda_+)} \pist$ on $\O^{\bfu} \times B$. 
\begin{lem}
The map 
$$
\psi: (\O^{\bfu} \times B, \: \pi_n \times_{(\rho_\bfu, \: \lambda_+)} \pist) \to (Z_{\bfu, \sD}, \pi_{\bfu, \sD}), \hs \psi([c]_{\sF_n}, b) = [[c]_{\sF_n}, \bar{b}], \hs c \in C_{\bar{\bfu}}, \; b \in B,
$$
is a local diffeomorphism and a Poisson map. 
\end{lem}
\begin{proof}
The image of $\psi$ is the open subset $\im(\psi) = \varpi(\O^{\bfu} \times \bar{\bar{B}}_- \bar{B})$ of $Z_{\bfu, \sD}$, and $\psi([c]_{\sF_n}, b) = \psi([c']_{\sF_n}, b')$ if and only if $c' = c^h$ and $b' = hb$, for some $h \in T^{(2)}$. Thus $\psi$ is a local diffeomorphism. By \eqref{eq-comp-pi+}, for $b \in B$ one has 
$$
\pi^+_\sD(\bar{b}) = \piD(\bar{b}) - r_{\bar{b}}(\bar{\bar{\xi}}^i \wedge \bar{x}_i),
$$
where $(x_i)$ is a basis of $\b$ and $(\xi^i)$ the dual basis of $\b_-$ with respect to $\lara_{(\b, \b_-)}$. Thus for $c \in C_{\bar{\bfu}}$, one as 
\begin{align*}
\pi_{\bfu, \sD}([[c]_{\sF_n}, \bar{b}]) & = \varpi \left( \pi_n([c]_{\sF_n}), \: \piD(\bar{b}) - r_{\bar{b}}(\bar{\bar{\xi}}^i \wedge \bar{x}_i) \right)   \\
   & =  \varpi \left( (\pi_n([c]_{\sF_n}), \: \piD(\bar{b}) \right) - \varpi(\rho_{\bfu}(\xi^i), \: 0) \wedge \varpi(0, \:  r_{\bar{b}}\bar{x}_i)   \\
   & = \psi \left( (\pi_n \times_{(\rho_\bfu, \: \lambda_+)} \pist)([c]_{\sF_n}, b) \right),
\end{align*}
hence $\psi$ is a Poisson map. 
\end{proof}

Recall the right Poisson action $\rho_{\sF'_{-n}}$ of $(B_-, \pist)$ on $(F'_{-n}, \pi'_{-n})$. One thus has the mixed product Poisson structure 
$$
\pi : = (\pi_{\bfu, \sD}, \pi'_{-n}) + (\rho_{\scriptscriptstyle Z_{\bfu, \sD}}(\bar{x}_i), 0) \wedge (0, \rho_{\sF'_{-n}}(\xi^i))
$$
on $Z_{\bfu, \sD} \times F'_{-n}$, associated to the right Poisson action $\rho_{\scriptscriptstyle Z_{\bfu, \sD}} \! \mid_{\bar{\b}}$ of $(\b, \delta_{st})$ on $(Z_{\bfu, \sD}, \pi_{\bfu, \sD})$ and the left Poisson action $-\rho_{\sF'_{-n}}$ of $(\b_-, -\delta_{st})$ on $(F'_{-n}, \pi'_{-n})$. Using the quotient map given in \eqref{eq-q}, one sees that $\rho_{\sF'_{-n}}$ is the restriction to $B_- \cong \bar{\bar{B}}_-$ of  the right Poisson action of $(D, \piD)$ on $(F'_{-n}, \pi'_{-n})$ given by 
$$
\rho_{\sF'_{-n}, \sD}([g_1, \ldots, g_n]_{\sF'_{-n}}, \: (g, h)) = [g_1, \ldots, g_{n-1}, g_ng], \hs  g, g_i \in G, \; h \in T, 
$$
and thus by $\S$\ref{subsec-mixed-prod},
$$
\rho(a) = (\rho_{\scriptscriptstyle Z_{\bfu, \sD}}(a) , \: \rho_{\sF'_{-n}, \sD}(a)), \hs a \in \d, 
$$
is a right Poisson action of $(\d, \delta_{\d})$ on $(Z_{\bfu, \sD} \times F'_{-n}, \pi)$. Hence 
$$
([[c]_{\sF_n}, \bar{b}], \: z) \cdot u : = ([[c]_{\sF_n}, \bar{b}\bar{\bar{u}}], \; \rho_{\sF'_{-n}}(z, u)), 
$$
where $c \in C_{\bar{\bfu}}$, $z \in F'_{-n}$, $b \in B$, and $u \in B_-$, is a right Poisson action of $(B_-, \pist)$ on $(Z_{\bfu, \sD} \times F'_{-n}, \pi)$.

\begin{pro} \label{pro-model-Guu}
1) The map 
$$
\varphi: (\O^{\bfu} \times B \times \O^{\bfv}, \: f_{\bfu, \bfv}(\pi_{\bbfu, \bbfv})) \to (Z_{\bfu, \sD} \times F'_{-n}, \: \pi)
$$
given by 
$$
\varphi([c]_{\sF_n}, \: b,  \: [c_-]_{\sF_n}) = (\psi([c]_{\sF_n}, b), \: [c_-]_{\sF'_{-n}}), \hs c \in C_{\bbfu}, c_- \in C_{\bbfv}, \; b \in B, \; b_- \in B_-,
$$
is an immersive Poisson map. 

2) For $(\tilde{y}, \gamma) \in \G^{\bbfu, \bbfv} \ast \Gamma_\sB$, one has 
$$
\varphi(f_{\bfu, \bfv}(\tilde{y} \lhd_\sB \gamma)) = \varphi(f_{\bfu, \bfv}(\tilde{y})) \cdot \theta_{\sB_-}(\gamma). 
$$
\end{pro}
\begin{proof}
Part 1) is clear, since $\psi$ is a Poisson map and a local diffeomorphism intertwining the right action of $B$ on $\O^{\bfu} \times B$  by right multiplication in the second factor, and the action $\rho_{\scriptscriptstyle Z_{\bfu, \sD}}$ of $\bar{B} \cong B$, and 
$$
(\O^{\bfu}, -\pi_n) \to (F'_{-n}, \pi'_{-n}), \hs [c_-]_{\sF_n} \mapsto  [c_-]_{\sF'_{-n}}, \hs c_- \in C_{\bar{\bfu}},
$$
is a $B_-$-equivariant immersive Poisson map. As for part 2), write $\tilde{y} =  ([c]_{\sF_n}, \: b, \: b_-, \: [c_-]_{\sF_n})$ and $\gamma = (b, u, u', b')$. Then 
\begin{align*}
\varphi(f_{\bfu, \bfv}(\tilde{y} \lhd_\sB \gamma)) & = \left( [[c^{u'}]_{\sF_n}, \: \bar{b'}], \: [c_-^u]_{\sF'_{-n}} \right)  = \left( [[c]_{\sF_n}, \: \bar{\bar{u'}}\bar{b'}], \: [c_-^u]_{\sF'_{-n}} \right)    \\
   & = \left( [[c]_{\sF_n}, \: \bar{b}\bar{\bar{u}}], \: [c_-^u]_{\sF'_{-n}} \right) = \left( [[c]_{\sF_n}, \: \bar{b}], \: [c_-]_{\sF'_{-n}} \right)  \cdot u   \\
   & = \varphi(f_{\bfu, \bfv}(\tilde{y})) \cdot \theta_{\sB_-}(\gamma). 
\end{align*}
\end{proof}

\noindent
{\it Proof of Theorem \ref{thm-lhd-BB_-}.}
Applying Proposition \ref{pro-gp/gpoid-action} to $\varphi$ shows that $\lhd_\sB$ satisfies \eqref{eq-dirac2}. Thus by Lemma \ref{lem-lhdB-diract1} and Proposition \ref{pro-liu.wei.xu-adapt}, $\lhd_\sB$ is a Poisson action of $(\Gamma_\sB, -\pi_\sGam)$ on $\mu_+$. 
\qed

\section{The Poisson groupoid $(\G^{\bbfw, \bbfw}, \: \pi_{\bbfw, \bbfw})$} \label{sec-main-thm-Gww}

The last main result of this paper is Theorem \ref{thm-main-Guu} below, in which we show that $(\G^{\bar{\bfw}, \bar{\bfw}}, \: \pi_{\bar{\bfw}, \bar{\bfw}})$ is a Poisson groupoid over $(\O^\bfw, \pi_\bfw)$, where $\bfw$ is any finite sequence of Weyl group elements. As the proof of Theorem \ref{thm-main-Guu} will be by induction, we fix in this section integers $n, m \geq 1$ and $\bfu \in W^n$, $\bfv \in W^m$ with respective representatives $\bar{\bfu} \in N_\sG(T)^n$, $\bar{\bfv} \in N_\sG(T)^m$, and assume that $\pi_{\bar{\bfu}, \bar{\bfu}}$, $\pi_{\bar{\bfv}, \bar{\bfv}}$ are compatible with the groupoid structures on $\G^{\bar{\bfu}, \bar{\bfu}}$, $\G^{\bar{\bfv}, \bar{\bfv}}$, that is $(\G^{\bar{\bfu}, \bar{\bfu}}, \pi_{\bar{\bfu}, \bar{\bfu}})$ and $(\G^{\bar{\bfv}, \bar{\bfv}},  \pi_{\bar{\bfv}, \bar{\bfv}})$ are Poisson groupoids. By \cite{Lu-Mou:double-B-cell}, this is true if $n = m = 1$.

\subsection{The local Poisson groupoid $(\K_{\bbfu, \bbfv}, \pi_{\scriptscriptstyle \K_{\bbfu, \bbfv}})$} \label{subsec-K_uv}

Consider the Poisson actions $\lhd_\sB$ of $(\Gamma_\sB, -\pi_\sGam)$ on $\mu_+: (\G^{\bar{\bfu}, \bar{\bfu}}, \pi_{\bar{\bfu}, \bar{\bfu}}) \to (B, \pist)$ and $\lhd_{\sB_-}$ of $(\Gamma_{\sB_-}, \pi_\sGam)$ on $\mu_-: (\G^{\bar{\bfv}, \bar{\bfv}},  \pi_{\bar{\bfv}, \bar{\bfv}}) \to (B_-, \pist)$. By Theorems \ref{thm-lhd-BB_-} and \ref{main-thm-gpoid}, one has the local Poisson groupoid 
$$
(\G^{\bar{\bfu}, \bar{\bfu}} \times_\ssL \G^{\bar{\bfv}, \bar{\bfv}}  \rightrightarrows \O^{\bfu} \times \O^{\bfv}, \; \pi_{\bar{\bfu}, \bar{\bfu}} \times \pi_{\bar{\bfv}, \bar{\bfv}}). 
$$
We construct in $\S$\ref{subsec-K_uv} a quotient of this local Poisson groupoid. Let $T$ act on $G^{\bfu, \bfu} \times G^{\bfv, \bfv}$ by 
\begin{align*}
& \left( \left( [g_1, \ldots, g_n]_{\sF_n}, [h_1, \ldots, h_n]_{\sF'_{-n}} \right), \: \left( [g'_1, \ldots, g'_m]_{\sF_m}, [h'_1, \ldots, h'_m]_{\sF'_{-m}} \right) \right) \cdot t =  \\
& \left( \left( [g_1, \ldots, g_nt]_{\sF_n}, [h_1, \ldots, h_nt]_{\sF'_{-n}} \right), \: \left( [t^{-1}g'_1, \ldots, g'_m]_{\sF_m}, [t^{-1}h'_1, \ldots, h'_m]_{\sF'_{-m}} \right) \right), 
\end{align*}
where $g_i, g'_i, h_i, h'_i \in G$ and $t \in T$. Under the diffeomorphism $J_{\bar{\bfu}, \bar{\bfu}}^{-1} \times J_{\bar{\bfv}, \bar{\bfv}}^{-1}$, this translates to the action 
\begin{align*}
&\left( ( [c]_{\sF_n}, \: b, \: b_-, \: [c_-]_{\sF_n}), \; ( [c']_{\sF_m}, \: b', \: b'_-, \: [c'_-]_{\sF_m} ) \right) \cdot t =    \\
& \left( ( [c]_{\sF_n}, \: bt, \: b_-b_{-\bar{\bfu}}(t, c), \: [c_-^t]_{\sF_n}), \; ( [t^{-1}[c']]_{\sF_m}, \: b_{\bfv}(t^{-1}, c')b', \: t^{-1}b'_-, \: [c'_-]_{\sF_m} ) \right) 
\end{align*}
of $T$ on $\G^{\bar{\bfu}, \bar{\bfu}} \times \G^{\bar{\bfv}, \bar{\bfv}}$, where $c, c_- \in C_{\bar{\bfu}}$, $c', c'_- \in C_{\bar{\bfv}}$, $b, b' \in B$, and $b_-, b'_- \in B_-$. Let $K_{\bfu, \bfv}$ be the quotient of $G^{\bfu, \bfu} \times G^{\bfv, \bfv}$ by $T$ and let $\K_{\bar{\bfu}, \bar{\bfv}}$ be the quotient of $\G^{\bar{\bfu}, \bar{\bfu}} \times \G^{\bar{\bfv}, \bar{\bfv}}$ by $T$. Let 
$$
\varpi_{\scriptscriptstyle \K_{\bar{\bfu}, \bar{\bfv}}}: \G^{\bar{\bfu}, \bar{\bfu}} \times \G^{\bar{\bfv}, \bar{\bfv}} \to \K_{\bar{\bfu}, \bar{\bfv}}
$$
be the quotient map and denote elements of $\K_{\bar{\bfu}, \bar{\bfv}}$ as $[\tilde{y}, \tilde{z}] = \varpi_{\scriptscriptstyle \K_{\bar{\bfu}, \bar{\bfv}}}(\tilde{y}, \tilde{z})$, if $\tilde{y} \in \G^{\bar{\bfu}, \bar{\bfu}}$ and $\tilde{z} \in \G^{\bar{\bfv}, \bar{\bfv}}$. As $T$ acts by Poisson isomorphisms, $\tilde{\pi}_{n,n} \times \tilde{\pi}_{m,m}$ descends to a well defined Poisson structure $\pi_{\sK_{\bfu, \bfv}}$ on $K_{\bfu, \bfv}$, and $\pi_{\bar{\bfu}, \bar{\bfu}} \times \pi_{\bar{\bfv}, \bar{\bfv}}$ descends to a well defined Poisson structure $\pi_{\scriptscriptstyle \K_{\bar{\bfu}, \bar{\bfv}}}$ on $\K_{\bar{\bfu}, \bar{\bfv}}$.

\begin{pro} \label{pro-Kuv-gpoid}
1) One has a structure on $\K_{\bar{\bfu}, \bar{\bfv}}$ of a local groupoid over $\O^{\bfu} \times \O^{\bfv}$ given by 
\begin{align*}
&\mbox{source map}: \; \theta_{\bar{\bfu}, \bar{\bfv}}[\tilde{y}, \: \tilde{z}]= (\theta_{\bar{\bfu}}(\tilde{y}), \; \mu_+(\tilde{y})[\theta_{\bar{\bfv}}(\tilde{z})]), \hs \tilde{y} \in \G^{\bar{\bfu}, \bar{\bfu}}, \; \tilde{z} \in \G^{\bar{\bfv}, \bar{\bfv}},    \\
&\mbox{target map}: \; \tau_{\bar{\bfu}, \bar{\bfv}}[\tilde{y}, \tilde{z}] = (\tau_{\bar{\bfu}}(\tilde{y})^{\mu_-(\tilde{z})}, \; \tau_{\bar{\bfv}}(\tilde{z})),  \hs \tilde{y} \in \G^{\bar{\bfu}, \bar{\bfu}}, \; \tilde{z} \in \G^{\bar{\bfv}, \bar{\bfv}},   \\
&\mbox{identity bisection}: \; \varepsilon_{\bar{\bfu}, \bar{\bfv}}(y, z) = [\varepsilon_{\bar{\bfu}}(y), \; \varepsilon_{\bar{\bfv}}(z)], \hs y \in \O^{\bfu}, \; z \in \O^{\bfv}. 
\end{align*}
Let $\tilde{y} \in \G^{\bar{\bfu}, \bar{\bfu}}$, $\tilde{z} \in \G^{\bar{\bfv}, \bar{\bfv}}$, and let $b = \mu_+(\tilde{y})$, $b'_- = \mu_-(\tilde{z})$. When $\bar{b} \bar{\bar{b}}_- \in \bar{B}\bar{\bar{B}}_- \cap \bar{\bar{B}}_-\bar{B}$, the inverse map is given by 
$$
\iota_{\bar{\bfu}, \bar{\bfv}}[\tilde{y}, \tilde{z}] = \left[ \iota_{\bar{\bfu}}(\tilde{y} \lhd_\sB \gamma), \; \iota_{\bar{\bfv}}(\tilde{z} \lhd_{\sB_-} \gamma) \right],  
$$
where $\gamma \in \Gamma$ is any element of the form $\gamma = (b, b'_-, u, g)$, $g \in B$, $u \in B_-$. Let $\tilde{y}_1, \tilde{y}_2 \in \G^{\bar{\bfu}, \bar{\bfu}}$, $\tilde{z}_1, \tilde{z}_2 \in \G^{\bar{\bfv}, \bar{\bfv}}$ with $\tau_{\bar{\bfu}, \bar{\bfv}}[\tilde{y}_1, \tilde{z}_1] = \theta_{\bar{\bfu}, \bar{\bfv}}[\tilde{y}_2, \: \tilde{z}_2]$, and $b'_{-1} = \mu_-(\tilde{z}_1)$, $b_2 = \mu_+(\tilde{y}_2)$. When $\bar{\bar{b'}}_{-1}\bar{b}_2 \in \bar{B}\bar{\bar{B}}_- \cap \bar{\bar{B}}_-\bar{B}$, multiplication is given by 
$$
[\tilde{y}_1, \tilde{z}_1] \cdot [\tilde{y}_2, \tilde{z}_2] = \left[ \tilde{y}_1 \left(\gamma \rhd_\sB \tilde{y}_2 \right), \;\; \left( \gamma \rhd_{\sB_-} \tilde{z}_1 \right) \tilde{z}_2 \right],  
$$
where $\gamma \in \Gamma$ is any element of the form $\gamma = (g, u,  b'_{-1}, b_2)$.

2) $(\K_{\bar{\bfu}, \bar{\bfv}} \rightrightarrows \O^{\bfu} \times \O^{\bfv}, \pi_{\scriptscriptstyle \K_{\bar{\bfu}, \bar{\bfv}}})$ is a local Poisson groupoid and 
$$
\varpi_{\scriptscriptstyle \K_{\bbfu, \bbfv}}: (\G^{\bbfu, \bbfu} \times_\ssL \G^{\bbfv, \bbfv}, \: \pi_{\bbfu, \bbfu} \times \pi_{\bbfv, \bbfv}) \to (\K_{\bbfu, \bbfv}, \pi_{\scriptscriptstyle \K_{\bbfu, \bbfv}})
$$
is a morphism of local Poisson groupoids. 
\end{pro}
\begin{proof}
Showing that the structure maps given in 1) are well-defined, satisfy the axioms of a local groupoid and that $\varpi_{\scriptscriptstyle \K_{\bbfu, \bbfv}}$ is a morphism of local groupoids is straightforward. Since $\varpi_{\scriptscriptstyle \K_{\bbfu, \bbfv}}$ is a Poisson map by construction, the only thing left to prove is that $(\K_{\bbfu, \bbfv}, \pi_{\scriptscriptstyle \K_{\bbfu, \bbfv}})$ is a local Poisson groupoid.

Let $U \subset \Gamma$ be an open subset such that the map $p$ defined in \eqref{eq-p} restricts to a diffeomorphism $p \! \mid_U: U \to p(U)$. Then the diagonal copy of $U$, 
$$
U_{\diag} \subset (\Gamma_\sB, -\pi_\sGam) \times (\Gamma_{\sB_-},  \pi_\sGam)
$$
is a local Lagrangian bisection. In particular, 
\begin{equation} \label{eq-proofK_uv}
\{(\tilde{y}_1, \: \tilde{z}_1 \lhd_{\sB_-} \gamma, \: \tilde{y}_2 \lhd_\sB \gamma, \: \tilde{y}_1 \tilde{y}_2, \: \tilde{z}_1 \tilde{z_2}) : \;\; (\tilde{y}_1, \tilde{y}_2) \in (\G^{\bar{\bfu}, \bar{\bfu}})^{(2)}, \; (\tilde{z}_1, \tilde{z}_2) \in (\G^{\bar{\bfv}, \bar{\bfv}})^{(2)}, \; \gamma \in U \}
\end{equation}
is the image of the graph of the direct product groupoid $\G^{\bar{\bfu}, \bar{\bfu}} \times \G^{\bar{\bfv}, \bar{\bfv}}$ by a local Poisson diffeomorphism, hence is coisotropic for the Poisson structure $\pi_{\bar{\bfu}, \bar{\bfu}} \times \pi_{\bar{\bfv}, \bar{\bfv}} \times \pi_{\bar{\bfu}, \bar{\bfu}} \times \pi_{\bar{\bfv}, \bar{\bfv}} \times (-\pi_{\bar{\bfu}, \bar{\bfu}}) \times (-\pi_{\bar{\bfv}, \bar{\bfv}})$ on $(\G^{\bar{\bfu}, \bar{\bfu}} \times \G^{\bar{\bfv}, \bar{\bfv}})^3$. By definition of the multiplication in $\K_{\bar{\bfu}, \bar{\bfv}}$, the image by $(\varpi_{\scriptscriptstyle \K_{\bar{\bfu}, \bar{\bfv}}})^3$ of the subset in \eqref{eq-proofK_uv} is an open subset in the graph $Gr(\K_{\bar{\bfu}, \bar{\bfv}})$ of the multiplication in $\K_{\bar{\bfu}, \bar{\bfv}}$. By varying $U$, it follows that $Gr(\K_{\bar{\bfu}, \bar{\bfv}})$ has an open covering by submanifolds coisotropic for the Poisson structure $\pi_{\scriptscriptstyle \K_{\bar{\bfu}, \bar{\bfv}}} \times \pi_{\scriptscriptstyle \K_{\bar{\bfu}, \bar{\bfv}}} \times (-\pi_{\scriptscriptstyle \K_{\bar{\bfu}, \bar{\bfv}}})$, hence is coisotropic. Thus $(\K_{\bar{\bfu}, \bar{\bfv}}, \pi_{\scriptscriptstyle \K_{\bar{\bfu}, \bar{\bfv}}})$ is a local Poisson groupoid. 
\end{proof}

\subsection{The concatenation map $\kappa_{\bfu, \bfv}$} \label{subsec-concatenation}

Let 
$$
\bfw = (\bfu, \bfv) = (u_1, \ldots, u_n, v_1, \ldots, v_m), \hs \bar{\bfw} = (\bbfu, \bbfv) = (\bar{u}_1, \ldots, \bar{u}_n, \bar{v}_1, \ldots, \bar{v}_m),
$$
where $\bfu = (u_1, \ldots, u_n)$ and $\bfv = (v_1, \ldots, v_m)$. Let $\kappa_{\bfu, \bfv}: G^{\bfu, \bfu} \times G^{\bfv, \bfv} \to G^{\bfw, \bfw}$ be the map 
\begin{align*}
\kappa_{\bfu, \bfv} & \left( ([g_1, \ldots, g_n]_{\sF_n}, [h_1, \ldots, h_n]_{\sF'_{-n}}), \: ([g'_1, \ldots, g'_m]_{\sF_m}, [h'_1, \ldots, h'_m]_{\sF'_{-m}}) \right) =   \\
  & \left( [g_1, \ldots, g_n, g'_1, \ldots, g'_m]_{\sF_{n+m}}, [h_1, \ldots, h_n, h'_1, \ldots, h'_m]_{\sF'_{-n-m}} \right),  \hs g'_i, g'_i, h_i, h'_i \in G. 
\end{align*}

By \cite[Proposition 8.3]{Lu-Mou:mixed}, $\kappa_{\bfu, \bfv}: (G^{\bfu, \bfu}, \pi_{n, n}) \times (G^{\bfv, \bfv}, \pi_{m, m}) \to (G^{\bfw, \bfw}, \pi_{n+m, n+m})$ is a Poisson map.

\begin{lem} \label{lem-kappa_uv}
1) The image of $\kappa_{\bfu, \bfv}$ is the Zariski open subset 
\begin{align*}
(G^{\bfw, \bfw})_0 & = \{ ([g_1, \ldots, g_{n+m}]_{\sF_{n+m}}, [h_1, \ldots, h_{n+m}]_{\sF'_{-n-m}}) \in G^{\bfw, \bfw} : \\
   & \hs \hs (h_1 \cdots h_n)^{-1}g_1 \cdots g_n \in B_-B \}. 
\end{align*}

2) The map $\kappa_{\bfu, \bfv}$ descends to a diffeomorphism $\kappa_{\bfu, \bfv}: K_{\bfu, \bfv} \stackrel{\cong}{\to} (G^{\bfw, \bfw})_0$. 
\end{lem}
\begin{proof}
It is clear that $\kappa_{\bfu, \bfv}$ maps into $(G^{\bfw, \bfw})_0$. Conversely, if 
$$
([g_1, \ldots, g_{n+m}]_{\sF_{n+m}}, [h_1, \ldots, h_{n+m}]_{\sF'_{-n-m}}) \in (G^{\bfw, \bfw} )_0,
$$
write $(h_1 \cdots h_n)^{-1}g_1 \cdots g_n = b_-b^{-1}$, with $b_- \in B_-$ and $b \in B$. Then 
\begin{align*}
&([g_1, \ldots, g_{n+m}]_{\sF_{n+m}}, [h_1, \ldots, h_{n+m}]_{\sF'_{-n-m}}) = \\
& \kappa_{\bfu, \bfv}\left( ([g_1, \ldots, g_nb]_{\sF_n}, [h_1, \ldots, h_nb_-]_{\sF'_{-n}}), \: ([b^{-1}g'_1, \ldots, g'_m]_{\sF_m}, [b_-^{-1}h'_1, \ldots, h'_m]_{\sF'_{-m}}) \right).
\end{align*}
Part 2) is clear. 
\end{proof}

Let $\kappa_{\bar{\bfu}, \bar{\bfv}}: (\G^{\bbfu, \bbfu}, \pi_{\bbfu, \bbfu}) \times (\G^{\bbfv, \bbfv}, \pi_{\bbfv, \bbfv}) \to (\G^{\bar{\bfw}, \bar{\bfw}}, \pi_{\bbfw, \bbfw})$ be the Poisson map obtained by precomposing $\kappa_{\bfu, \bfv}$ with $J_{\bar{\bfu}, \bar{\bfu}} \times J_{\bar{\bfv}, \bar{\bfv}}$ and postcomposing with $J_{\bar{\bfw}, \bar{\bfw}}^{-1}$, that is 
\begin{align}
\kappa_{\bbfu, \bbfv} & \left( ([c]_{\sF_n}, \: b, \: b_-, \: [c_-]_{\sF_n}), \; ([c']_{\sF_m}, \: b', \: b'_-, \: [c'_-]_{\sF_m}) \right) = \label{eq-k_uv}   \\
   & \; ( [c, b[c']]_{\sF_{n+m}},\;  b_{\bbfv}(b, c')b', \; b_-b_{-\bbfu}(b'_-, c_-), \; [c_-^{b'_-}, c'_-]_{\sF_{n+m}}),   \notag
\end{align}
where $([c]_{\sF_n}, \: b, \: b_-, \: [c_-]_{\sF_n}) \in \G^{\bbfu, \bbfu}$ and $([c']_{\sF_m}, \: b', \: b'_-, \: [c'_-]_{\sF_m}) \in \G^{\bbfv, \bbfv}$. By Lemma \ref{lem-kappa_uv}, the image of $\kappa_{\bbfu, \bbfv}$ is $(\G^{\bbfw, \bbfw})_0 = J_{\bar{\bfw}, \bar{\bfw}}^{-1}((G^{\bfw, \bfw})_0)$ and $\kappa_{\bbfu, \bbfv}$ descends to an diffeomorphism $\kappa_{\bar{\bfu}, \bar{\bfv}}: \K_{\bar{\bfu}, \bar{\bfv}} \cong (\G^{\bar{\bfw}, \bar{\bfw}})_0$. As $(\G^{\bar{\bfw}, \bar{\bfw}})_0$ is an open neighborhood of the identity bisection of $\G^{\bar{\bfw}, \bar{\bfw}}$, one can {\it shrink} \cite{assoc_int} the groupoid structure on $\G^{\bar{\bfw}, \bar{\bfw}}$ to a local groupoid structure on $(\G^{\bar{\bfw}, \bar{\bfw}})_0$, where the multiplication is defined on 
$$
(\G^{\bar{\bfw}, \bar{\bfw}})_0^{(2)} := m_{\bar{\bfw}, \bar{\bfw}}^{-1}((\G^{\bar{\bfw}, \bar{\bfw}})_0) \cap \{(\tilde{x}_1, \tilde{x}_2) \in ((\G^{\bar{\bfw}, \bar{\bfw}})_0)^2 : \; \tau_{\bbfw}(\tilde{x}_1) = \theta_{\bbfw}(\tilde{x}_2)\},
$$
where $m_{\bar{\bfw}, \bar{\bfw}}$ is the multiplication map on $\G^{\bar{\bfw}, \bar{\bfw}}$, and the inverse is defined on 
$$
(\G^{\bar{\bfw}, \bar{\bfw}})_0^{(-1)} := (\G_{\bar{\bfw}, \bar{\bfw}})_0 \cap \iota_{\bar{\bfw}}((\G_{\bar{\bfw}, \bar{\bfw}})_0). 
$$
The remainder of $\S$\ref{subsec-concatenation} is devoted to proving the following 

\begin{pro} \label{pro-kappa_uv}
One has an isomorphism of local groupoids 
$$
\kappa_{\bbfu, \bbfv}: (\K_{\bbfu, \bbfv} \rightrightarrows \O^{\bfu} \times \O^{\bfv}) \cong ((\G^{\bar{\bfw}, \bar{\bfw}})_0 \rightrightarrows \O^{\bfw}). 
$$
\end{pro}

Fix for $i = 1,2$ elements 
\begin{align}
\tilde{y}_i & = ([c_i]_{\sF_n}, \: b_i, \: b_{-i}, \: [c_{-i}]_{\sF_n}) \in \G^{\bbfu, \bbfu}, \hs c_i, c_{-i} \in C_{\bbfu}, \; b_i \in B , \;  b_{-i} \in B_-, \label{eq-write-y_iz_i} \\
\tilde{z}_i & =   ([c'_i]_{\sF_m}, \: b'_i, \: b'_{-i}, \: [c'_{-i}]_{\sF_m})  \in \G^{\bbfv, \bbfv} , \hs c'_i, c'_{-i} \in C_{\bbfv}, \; b'_i \in B , \;  b'_{-i} \in B_-, \notag
\end{align}
such that $[\tilde{y}_1, \tilde{z}_1]$ and $[\tilde{y}_2, \tilde{z}_2]$ are composable in $\K_{\bbfu, \bbfv}$, that is 
\begin{equation} \label{eq-source=target}
(c_{-1}^{b'_{-1}}, c'_{-1}) = (c_2, b_2[c'_2]),
\end{equation}
and choose a 
\begin{equation} \label{eq-gamma-kappa_uv}
\gamma = (g, u, b'_{-1}, b_2) \in \Gamma, \hs \text{with} \hs  g \in B, \; u \in B_-. 
\end{equation}

\begin{lem} \label{lem1-kappa_uv}
One has
\begin{align*}
b_{\bbfv}(b_2^{-1}, c'_{-1})^{-1} & = b_{\bbfv}(b_2, c'_2),    \\
b_{-\bbfu}((b'_{-1})^{-1}, c_2)^{-1} & = b_{-\bbfu}(b'_{-1}, c_{-1}). 
\end{align*}
\end{lem}
\begin{proof}
By \eqref{eq-source=target} and recalling our notation in \eqref{eq-ddot}, 
\begin{align*}
b_2\underline{c'_2} & = \underline{b_2[c'_2]} b_{\bbfv}(b_2, c'_2) = \underline{c'_{-1}}b_{\bbfv}(b_2, c'_2), \\ 
(b_2)^{-1}\underline{c'_{-1}} & = \underline{(b_2)^{-1}[c'_{-1}]} b_{\bbfv}((b_2)^{-1}, c'_{-1}) = \underline{c'_2} b_{\bbfv}(b_2^{-1}, c'_{-1}),
\end{align*}
thus $b_{\bbfv}(b_2^{-1}, c'_{-1})^{-1} = b_{\bbfv}(b_2, c'_2)$, and the second relation is proven similarly. 
\end{proof}

\noindent
{\it Proof of Proposition \ref{pro-kappa_uv}.}
It is easily seen that $\kappa_{\bbfu, \bbfv}$ commutes with the respective source and target maps of $\K_{\bbfu, \bbfv}$ and $(\G^{\bar{\bfw}, \bar{\bfw}})_0$. By Proposition \ref{pro-Kuv-gpoid}, one has 
\begin{align*}
[\tilde{y}_1, \tilde{z}_1] \cdot [\tilde{y}_2, \tilde{z}_2] & = \left[ \tilde{y}_1 \left(\gamma \rhd_\sB \tilde{y}_2 \right), \;\; \left( \gamma \rhd_{\sB_-} \tilde{z}_1 \right) \tilde{z}_2 \right], \\
   & = \left[ ([c_1]_{\sF_n}, \: b_1g, \: b_{-1} \mu_{-\bfu}(\gamma \rhd_\sB \tilde{y}_2), \: [c_{-2}^{u^{-1}}]_{\sF_n}), \right.   \\
   & \hs \; \left.   ([g^{-1}[c'_1]]_{\sF_m}, \: \mu_{\bfv}(\gamma \rhd_{\sB_-} \tilde{z}_1)b'_2, \: ub'_{-2}, \: [c'_{-2}]_{\sF_m}) \right], 
\end{align*}
where recall that by definition of $\lhd_\sB$, $\lhd_{\sB_-}$, 
\begin{align*}
b_{-\bbfu}((b'_{-1})^{-1}, c_2)\mu_{-\bfu}(\gamma \rhd_\sB \tilde{y}_2) & = b_{-2} b_{-\bbfu}(u^{-1}, c_{-2}),   \\
\mu_{\bfv}(\gamma \rhd_{\sB_-} \tilde{z}_1) b_{\bbfv}(b_2^{-1}, c'_{-1}) & = b_{\bbfv}(g^{-1}, c'_1) b'_1. 
\end{align*}
Hence by Lemma \ref{lem1-kappa_uv}, 
\begin{align*}
b_1g \underline{g^{-1}[c'_1]} \mu_{\bfv}(\gamma \rhd_{\sB_-} \tilde{z}_1)b'_2 & = b_1\underline{c'_1}b_{\bbfv}(g^{-1}, c'_1)^{-1}\mu_{\bfv}(\gamma \rhd_{\sB_-} \tilde{z}_1)b'_2   \\   
    & = b_1\underline{c'_1}b'_1b_{\bbfv}(b_2^{-1}, c'_{-1})^{-1}b'_2   \\
   & = b_1\underline{c'_1}b'_1b_{\bbfv}(b_2, c'_2)b'_2 = \underline{b_1[c'_1]} b_{\bbfv}(b_1, c'_1)b'_1b_{\bbfv}(b_2, c'_2)b'_2,     \\
b_{-1}\mu_{-\bfu}(\gamma \rhd_\sB \tilde{y}_2)\underline{(c_{-2})^{u^{-1}}}ub'_{-2} & = b_{-1}\mu_{-\bfu}(\gamma \rhd_\sB \tilde{y}_2) b_{-\bbfu}(u^{-1}, c_{-2})^{-1}\underline{c_{-2}}b'_{-2}   \\
   & = b_{-1}b_{-\bbfu}((b'_{-1})^{-1}, c_2)^{-1}b_{-2}\underline{c_{-2}}b'_{-2}   \\  
   & = b_{-1}b_{-\bbfu}(b'_{-1}, c_{-1})b_{-2}\underline{c_{-2}}b'_{-2}   \\
   & = b_{-1}b_{-\bbfu}(b'_{-1}, c_{-1})b_{-2}b_{-\bbfu}(b'_{-2}, c_{-2})\underline{(c_{-2})^{b'_{-2}}},
\end{align*}
and so by \eqref{eq-k_uv}, one has 
\begin{align}
&\kappa_{\bbfu, \bbfv}([\tilde{y}_1, \tilde{z}_1] \cdot [\tilde{y}_2, \tilde{z}_2]) = \notag  \\
& ([c_1, b[c'_1]]_{\sF_{n+m}}, \; b_{\bbfv}(b_1, c'_1)b'_1b_{\bbfv}(b_2, c'_2)b'_2, \;\;  b_{-1}b_{-\bbfu}(b'_{-1}, c_{-1})b_{-2}b_{-\bbfu}(b'_{-2}, c_{-2}), \; [c_{-2}^{b'_{-2}}, c'_{-2}]_{\sF_{n+m}}) \label{eq-mult-y_iz_i} \\
& = \kappa_{\bbfu, \bbfv}([\tilde{y}_1, \tilde{z}_1]) \kappa_{\bbfu, \bbfv}([\tilde{y}_2, \tilde{z}_2)]. \notag 
\end{align}

Conversely, suppose given $\tilde{x}_i \in (\G^{\bbfw, \bbfw})_0$, $i = 1, 2$, composable in $\G^{\bbfw, \bbfw}$, and let $\tilde{y}_i \in \G^{\bbfu, \bbfu}$, $\tilde{z}_i \in \G^{\bbfv, \bbfv}$, such that $\tilde{x}_i = \kappa_{\bbfu, \bbfv}[\tilde{y}_i, \tilde{z}_i]$. Writing $\tilde{y}_i$ and $\tilde{z}_i$ as in \eqref{eq-write-y_iz_i}, $\tilde{x}_1 \tilde{x}_2$ is given by \eqref{eq-mult-y_iz_i}, so by Lemma \ref{lem-kappa_uv}, $\tilde{x}_1\tilde{x}_2$ lies in $(\G^{\bar{\bfw}, \bar{\bfw}})_0$ if and only if 
\begin{align*}
B_-B  & \ni (b_{-1}b_{-\bbfu}(b'_{-1}, c_{-1})b_{-2}b_{-\bbfu}(b'_{-2}, c_{-2}) \underline{c_2^{b'_{-2}}})^{-1} \underline{c_1}   \\
   & = (\underline{c_{-2}}b'_{-2})^{-1}b_{-2}^{-1}b_{-\bbfu}(b'_{-1}, c_{-1})^{-1}b_{-1}^{-1} \underline{c_1}   \\
   & =  (b'_{-2})^{-1}(b_{-2}\underline{c_{-2}})^{-1}b_{-\bbfu}((b'_{-1})^{-1}, c_2)b_{-1}^{-1} \underline{c_1}    \\
   & = (b_2b'_{-2})^{-1}\underline{c_2}^{-1}b_{-\bbfu}((b'_{-1})^{-1}, c_2)\underline{c_{-1}}b_1^{-1}  \\
   & = (b_2b'_{-2})^{-1} (b'_{-1})^{-1} (\underline{c_2^{(b'_{-1})^{-1}}})^{-1}\underline{c_{-1}}b_1^{-1} \\
   & = (b_2b'_{-2})^{-1} ( b_1b'_{-1})^{-1}.
\end{align*}
We have used Lemma \ref{lem1-kappa_uv} in the third line above, the definition \eqref{eq-G^uu} of $\G^{\bbfu, \bbfu}$ in the fourth line, and \eqref{eq-source=target} in the last. Thus $\tilde{x}_1 \tilde{x}_2 \in (\G^{\bar{\bfw}, \bar{\bfw}})_0$ precisely when $(b'_{-1}b_2)^{-1} \in B_-B$, which is equivalent to the existence of a $\gamma \in \Gamma$ as in \eqref{eq-gamma-kappa_uv}. Hence $[\tilde{y}_1, \tilde{z}_1]$ and $[\tilde{y}_1, \tilde{z}_1]$ are composable in $\K_{\bbfu, \bbfv}$ precisely when $\tilde{x}_1 \tilde{x}_2 \in (\G^{\bar{\bfw}, \bar{\bfw}})_0$, and in such a case one has $\tilde{x}_1 \tilde{x}_2 = \kappa_{\bbfu, \bbfv}([\tilde{y}_1, \tilde{z}_1]\cdot [\tilde{y}_1, \tilde{z}_1])$. A similar calculation shows that $\kappa_{\bbfu, \bbfv}$ commutes with the respective inverse groupoid maps, and that an element $[\tilde{y}, \tilde{z}] \in \K_{\bbfu, \bbfv}$ is invertible precisely when $\kappa_{\bbfu, \bbfv}[\tilde{y}, \tilde{z}]$ is invertible in $(\G^{\bar{\bfw}, \bar{\bfw}})_0$. This concludes the proof of Proposition \ref{pro-kappa_uv}.
\qed

\begin{cor}
The map
$$
I_{\bfu, \bfv}: (\O^{\bfu} \times \O^{\bfv}, \pi_n \times_{(\rho_{\bfu}, \lambda_{\bfv})} \pi_m) \to (\O^{\bfw}, \pi_{n+m}), \hs I_{\bfu, \bfv}([c]_{\sF_n}, [c']_{\sF_m}) = [c, c']_{\sF_{n+m}}, 
$$
where $c \in C_{\bbfu}$, $c' \in C_{\bbfv}$, 
is an isomorphism of Poisson manifolds. 
\end{cor}
\begin{proof}
By Proposition \ref{pro-Kuv-gpoid} and \ref{pro-kappa_uv} $((\G^{\bar{\bfw}, \bar{\bfw}})_0 \rightrightarrows \O^{\bfw}, \pi_{\bbfw, \bbfw})$ is a local Poisson groupoid over $(\O^{\bfw}, \pi_{n+m})$ and $\kappa_{\bbfu, \bbfv}: (\K_{\bbfu, \bbfv} \rightrightarrows \O^{\bfu} \times \O^{\bfv}, \pi_{\scriptscriptstyle \K_{\bbfu, \bbfv}}) \cong ((\G^{\bar{\bfw}, \bar{\bfw}})_0 \rightrightarrows \O^{\bfw},\pi_{\bbfw, \bbfw})$ is an isomorphism of local Poisson groupoids. Then $I_{\bfu, \bfv}$ is precisely the map between the bases of the two local groupoids covered by the isomorphism $\kappa_{\bbfu, \bbfv}$. 
\end{proof}

\subsection{The Poisson groupoid $(\G^{\bbfw, \bbfw}, \: \pi_{\bbfw, \bbfw})$}

\begin{thm} \label{thm-main-Guu}
Let $l \geq 1$, $\bfw \in W^l$, and let $\bbfw \in N_\sG(T)^l$ be a of representative of $\bfw$. Then $(\G^{\bbfw, \bbfw} \rightrightarrows \O^{\bfw}, \: \pi_{\bbfw, \bbfw})$ is a Poisson groupoid over $(\O^{\bfw}, \pi_l)$. 
\end{thm}
\begin{proof}
By \cite{Lu-Mou:double-B-cell} the Theorem is true for $n = 1$. By induction, one can assume that $\bfw = (\bfu, \bfv)$, where $\bfu \in W^n$ and $\bfv \in W^m$, and such that the Theorem holds for $\bfu$ and $\bfv$. Then by Propositions \ref{pro-Kuv-gpoid} and \ref{pro-kappa_uv}, $((\G^{\bbfw, \bbfw})_0 \rightrightarrows \O^{\bfw}, \: \pi_{\bbfw, \bbfw})$ is a local Poisson groupoid, that is 
$$
Gr((\G^{\bar{\bfw}, \bar{\bfw}})_0) = \{ (\tilde{x}_1, \: \tilde{x}_2, \: \tilde{x}_1 \tilde{x}_2) : \hs (\tilde{x}_1, \tilde{x}_2) \in (\G^{\bar{\bfw}, \bar{\bfw}})_0^{(2)} \}
$$
is a coisotropic submanifold of $(\G^{\bbfw, \bbfw})^3$, equipped with the Poisson structure $\pi_{\bbfw, \bbfw} \times \pi_{\bbfw, \bbfw} \times (-\pi_{\bbfw, \bbfw})$. But as $(\G^{\bar{\bfw}, \bar{\bfw}})_0$ is Zariski open in the irreducible algebraic variety $\G^{\bar{\bfw}, \bar{\bfw}}$, $(\G^{\bar{\bfw}, \bar{\bfw}})_0^{(2)}$ is open and dense in $(\G^{\bbfw, \bbfw})^{(2)}$, thus $Gr((\G^{\bar{\bfw}, \bar{\bfw}})_0)$ is open and dense in 
$$
Gr(\G^{\bar{\bfw}, \bar{\bfw}}) = \{ (\tilde{x}_1, \: \tilde{x}_2, \: \tilde{x}_1 \tilde{x}_2) : \hs (\tilde{x}_1, \tilde{x}_2) \in (\G^{\bar{\bfw}, \bar{\bfw}})^{(2)} \}. 
$$
Hence $Gr(\G^{\bar{\bfw}, \bar{\bfw}})$ is coisotropic for the Poisson structure $\pi_{\bbfw, \bbfw} \times \pi_{\bbfw, \bbfw} \times (-\pi_{\bbfw, \bbfw})$, that is $(\G^{\bbfw, \bbfw} \rightrightarrows \O^{\bfw}, \: \pi_{\bbfw, \bbfw})$ is a Poisson groupoid. 
\end{proof}


\bibliographystyle{plain} 

\end{document}